\documentclass[11pt]{amsart}

\usepackage{amssymb,amsmath,amsthm,amsfonts,amscd,latexsym, dsfont, color}
\usepackage[dvipsnames]{xcolor}

\usepackage[american]{babel} 

\usepackage{array}
\usepackage[citecolor=ForestGreen,linkcolor=blue,urlcolor=blue]{hyperref} 
 \usepackage{graphics}
\usepackage{wrapfig}
\usepackage{pst-node}
\usepackage{tikz-cd} 
\hypersetup{colorlinks}
\usepackage[normalem]{ulem}

\usepackage{setspace} 
\usepackage{tcolorbox}

\usepackage{footnote}

\usepackage{cancel}

\newcommand{\dev}{\mathsf{dev}}
\newcommand{\hol}{\mathsf{hol}} 

\newcommand{\Teich}{\mathsf{Teich}}

\usepackage{MnSymbol}
\usepackage{enumitem}
\newcommand{\Fix}{\mathsf{Fix}}

\newcommand{\proj}{\mathsf{proj}} 
\newcommand{\spann}{\mathsf{span}}

\newcommand{\V}{\mathcal{E}} 
\definecolor{bluegreen}{RGB}{5, 170, 204}

\definecolor{defnyellow}{RGB}{255,224,102} 
\definecolor{lightblue}{RGB}{170, 220, 255} 
\definecolor{darkblue}{RGB}{0, 125, 230}

\newcommand{\Gtwo}{\mathsf{G}_2}

\newcommand{\Gtwosplit}{\mathsf{G}_2^{\mathsf{'}}}

\newcommand{\Sym}{\mathsf{Sym}}

\newtcolorbox{bluebox}[1]{colback=lightblue,colframe=darkblue,fonttitle=\bfseries,title=#1}
\newtcolorbox{redbox}[1]{colback=red!5!white,colframe=red!75!black,fonttitle=\bfseries,title=#1}

\usepackage{tikz}
\usepackage{pgfplots}

\usepgfplotslibrary{colormaps}
\pgfplotsset{
    compat=newest,
    colormap={mycolormap}{color=(lightgray) color=(white) color=(lightgray) } }

\definecolor{mydarkblue}{RGB}{37, 42, 200}

\newcommand{\sff}{\mathrm{I\!I}}
\newcommand{\tff}{\mathrm{I\!I\!I}}

\definecolor{mygreen}{RGB}{0, 150, 50}

\newcommand{\radialtorus}{\mathbb{S}^1 \times \mathbb{S}^1 \times \R_+} 
\newcommand{\Ein}{\mathsf{Ein}}

\newcommand{\imoct}{\mathsf{Im}(\Oct')}

\newcommand{\SO}{\mathsf{SO}}

\newcommand{\Gr}{\mathsf{Gr}} 
\newcommand{\Stab}{\mathsf{Stab}} 
\usepackage{mathrsfs}
\usepackage[utf8]{inputenc}
\usepackage[T1]{fontenc}
\newcommand{\R}{\mathbb R}
\newcommand{\RP}{\mathbb R \mathbb P}

\newcommand{\C}{\mathbb C}
\newcommand{\Z}{\mathbb Z}

\newcommand{\Ha}{\mathbb{H}} 
\newcommand{\F}{\mathbb{F}}
\newcommand{\g}{\mathfrak{g}}
\newcommand{\frakk}{\mathfrak{k}} 
 
\newcommand{\h}{\mathfrak{h}}

\newcommand{\K}{\mathcal{K}}
\newcommand{\End}{\mathsf{End}}

\newcommand{\Stwofour}{\mathbb{S}^{2,4}} 
\newcommand{\quadric}{\hat{ \mathbb{S}}^{2,4}}

\newcommand{\PSL}{\mathsf{PSL}}
\newcommand{\Sp}{\mathsf{Sp}}
\newcommand{\GL}{\mathsf{GL}}
\newcommand{\gl}{\mathfrak{gl}}

\newcommand{\id}{\text{id}}

  \newcommand{\Aut}{\mathsf{Aut}}

 \newcommand{\frakt}{\mathfrak{t}} 
 
 \newcommand{\Hom}{\mathsf{Hom}}
  
 \newcommand{\Fr}{\mathsf{Fr}\,}

 \newcommand{\Ad}{\text{Ad}}

\newcommand{\delbar}{\overline{\partial}}

\newcommand{\Ann}{\mathsf{Ann}}
 
\newcommand{\Iso}{\mathsf{Iso}} 
\usepackage{mathtools}

\newcommand{\Isom}{\mathsf{Isom}}

\newcommand{\der}[1]{\frac{\partial}{\partial #1}} 
\newcommand{\deriv}[2]{\frac{\partial #1}{\partial #2}}

\newcommand{\zbar}{\bar{z}}

\newcommand{\rank}{\text{rank}}

\newcommand{\im}{\text{im}}

\newcommand{\sig}{\mathsf{sig}}

\newcommand{\Eintwothree}{\mathsf{Ein}^{2,3}} 
\newcommand{\Oct}{\mathbb{O}}

\newcommand{\Hit}{\mathsf{Hit}}

\usepackage{thmtools, thm-restate}
\newcommand{\Diff}{\mathsf{Diff}}

\theoremstyle{plain}

\newtheorem{theorem}{Theorem}[section]
\newtheorem{proposition}[theorem]{Proposition}
\newtheorem{lemma}[theorem]{Lemma}
\newtheorem{corollary}[theorem]{Corollary}

\newtheorem{definition}[theorem]{Definition}
\newtheorem{remark}[theorem]{Remark}

\newcommand{\alignL}{\begin{flushleft}}
\newcommand{\alignLend}{\end{flushleft}}

\usepackage{soul}

\makeatletter
\@namedef{subjclassname@2020}{\textup{2020} Mathematics Subject Classification}
\makeatother

\usepackage[
backend=bibtex,
date = year,
maxnames = 4,
style=alphabetic,
]{biblatex}
\renewbibmacro{in:}{}
\usepackage{bbm}

\title{Geometric Structures for the $\Gtwosplit$-Hitchin Component}
\author{Parker Evans}
\thanks{The author acknowledge(s) support from U.S. National Science Foundation grants DMS-2005551 and NSF DMS-1745670.}
\address{Department of Mathematics, Rice University, 
Houston, TX 77005}
\email{pge1@rice.edu}

\usepackage[margin= 1in]{geometry}
\numberwithin{equation}{section}

\begin{document}

\maketitle 

\begin{abstract}
We give an explicit geometric structures interpretation of the $\Gtwosplit$-Hitchin component $\Hit(S, \Gtwosplit) \subset \chi(\pi_1S,\Gtwosplit)$ of a closed oriented surface $S$ of genus $g \geq 2$. In particular, 
we prove $\Hit(S, \Gtwosplit)$ is naturally homeomorphic to a moduli space $\mathscr{M}$ of $(G,X)$-structures for $G = \Gtwosplit$ and $X = \Eintwothree$ on a fiber bundle $\mathscr{C}$ over $S$ via the descended holonomy map. Explicitly, $\mathscr{C}$ is the direct sum of fiber bundles $\mathscr{C} =UT S \oplus UT S \oplus \uline{\R_+}$ 
with fiber $\mathscr{C}_p = UT_p S \times UT_p S \times \R_+$, where $UT S$ denotes the unit tangent bundle. \\

The geometric structure associated to a $\Gtwosplit$-Hitchin representation $\rho$ is explicitly constructed from the unique associated $\rho$-equivariant alternating almost-complex curve $\hat{\nu}: \tilde{S} \rightarrow \quadric$; we critically use recent work of Collier-Toulisse on the moduli space of such curves. 
Our explicit geometric structures are examined in the $\Gtwosplit$-Fuchsian case and shown 
to be unrelated to the $(\Gtwosplit,\Eintwothree)$-structures of Guichard-Wienhard. 
\end{abstract} 

\newpage 

\tableofcontents 

\newpage

\section{Introduction} 

Let $S := S_g$ be a closed oriented surface of genus $g \geq 2$. 
The Hitchin component $\mathsf{Hit}(S, G)$ of a split real simple (adjoint) Lie group $G$ is a connected component 
 in the character variety $\chi(\pi_1 S, G) = \Hom^{\mathsf{Red}}( \pi_1 S, G)/G$, first defined by Hitchin in 1992. The component $\Hit(S,G)$ is a generalization of 
 Teichm\"uller space $T(S) = \Hit(S, \PSL_2\R)$, and is defined precisely so that it contains an embedded copy of $T(S)$ 
 called the $G$-Fuchsian locus.
 There are many robust similarities between $\Hit(S,G)$ and $T(S)$. For example, $\Hit(S,G)$ is simply connected: it is homeomorphic to $\R^{(2g-2)\dim G}$.  
Moreover, Hitchin representations are irreducible \cite{Hit92} and discrete and faithful, as proven independently by Labourie \cite{Lab06} and Fock and Goncharov
\cite{FG06}. Thus, Hitchin components are examples of so-called \emph{higher Teichm\"uller spaces} \cite{Wie18}, connected components of $\chi(\pi_1S,G)$ for a semisimple Lie group $G$ of $\rank \, G \geq 2$, that consist exclusively
of discrete and faithful representations. The Hitchin components remain a central object of investigation of higher Teichm\"uller theory.

A classical fact of Teichm\"uller theory is that $T(S)$ is homeomorphic to the moduli space $\mathscr{H}(S)$ of marked hyperbolic structures on $S$.
Since any closed hyperbolic surface is locally isometric to hyperbolic space $\Ha^2$, a marked hyperbolic structure on $S$ is equivalently a $(\PSL_2\R, \Ha^2)$-structure on $S$ in the language of $(G,X)$-structures. 
In fact, the holonomy map $\hol: \mathscr{H}(S) \rightarrow T(S)$ is a homeomorphism. 
In 1993, Choi and Goldman proved that $\mathsf{Hit}(S, \mathsf{SL}_3 \R)$ parametrizes the deformation space of \emph{convex projective structures}
on $S$, namely certain $( \mathsf{SL}_3 \R, \RP^2)$-structures, thereby realizing the component $\Hit(S,\mathsf{SL}_3\R)$ as holonomies of geometric structures in the same fashion
as the earlier identification $T(S) \cong \mathscr{H}(S)$ \cite{Gol90, CG93}. 
An important aim of higher Teichm\"uller theory, promoted by Hitchin in \cite{Hit92}, is to realize all components $\Hit(S,G)$ as holonomies
of $(G,X)$-structures. We remark on a subtlety here. In the foundational examples of $G \in \{ \PSL_2\R, \mathsf{SL}_3\R\}$, the $(G,X)$-structures associated to Hitchin representations live on the surface $S$ itself. 
However, for $G = \PSL_n\R$ when $n \geq 4$ and for $G = \Gtwosplit$, the split real (adjoint) form of the exceptional complex simple Lie group $\Gtwo$, any flag manifold $X = G/H$ associated to a parabolic subgroup $H <G$ has $\dim X \geq 3$.
If one is set on using flag manifolds
as the model spaces for such a geometric structures interpretation, then the geometric structures cannot live on the surface $S$ in the case of $\Gtwosplit$ or $\PSL_n\R$ for $n \geq 4$. 

Since Choi and Goldman, few other \emph{explicit} examples of geometric structures interpretations of $G$-Hitchin components have arisen -- only the cases of 
$\PSL_4\R$ and $\mathsf{P}\Sp_4\R \cong \SO_0(2,3)$ have been addressed. We briefly recall these results for context now. Guichard and Wienhard proved in 2008 that $\mathsf{Hit}(S, \PSL_4\R)$ parametrizes the deformation space of \emph{properly convex foliated (PCF) projective structures},
on the unit tangent bundle $UTS$ \cite{GW08}. These PCF structures on $UTS$ are $(\PSL_4\R, \RP^3)$-structures with extra conditions related to a canonical foliation of $UTS$. The holonomy is trivial along the fibers of $UTS \rightarrow S$ and so $\hol: \pi_1 UTS \rightarrow \mathsf{PSL}_4\R$ descends to a map $\overline{\hol}: \pi_1S \rightarrow \mathsf{PSL}_4\R$. Guichard and Wienhard prove that the descended
holonomy map $\overline{\hol}$ is a homeomorphism from the deformation space of PCF structures on $UTS$ to $\Hit(S, \PSL_4\R)$. Now, the Hitchin component $\mathsf{Hit}(S, \mathsf{PSp}_4\R) $ 
includes in $\mathsf{Hit}(S, \PSL_4\R)$. As a corollary to the previous geometric structures result, Guichard and Wienhard show that $\mathsf{Hit}(S, \mathsf{PSp}_4\R)$ is homeomorphic to the moduli space of 
PCF $(\mathsf{PSp}_4\R, (\RP^3, \alpha))$-structures on $UTS$, where $\alpha$ is the standard contact structure on $\RP^3$. 
On the other side of the isomorphism $\mathsf{PSp}_4\R \cong \SO_0(2,3)$, Collier, Tholozan, and Toulisse re-interpreted the PCF projective structures of \cite{GW08} as $( \SO_0(2,3), \mathsf{Pho}(\R^{2,3}) )$-structures, where $\mathsf{Pho}(\R^{2,3})$ is the space of isotropic planes in $\R^{2,3}$, specifically 
\emph{fibered photon structures} on $UTS$, where the fibering of $UTS \rightarrow S$ is now explicitly involved in the geometry, unlike
in the PCF structures perspective \cite{CTT19}. In \cite{CTT19}, they also describe another explicit geometric structures interpretation for $\Hit(S, \mathsf{SO}_0(2,3))$ via  
\emph{fibered conformally flat Lorentz (CFL)} structures on $UTS$, which are, in particular, $(\SO_0(2,3), \Ein^{1,2})$-structures. The CFL structures on $UTS$ are shown to be 
an explicit realization of certain abstract geometric structures described by \cite{GW12}. We will discuss the work \cite{GW12} more momentarily. 

In the spirit of the aforementioned works, we consider the problem of finding an explicit geometric structures interpretation of the $\Gtwosplit$-Hitchin component as holonomies of $(G,X)$-structures on an $F$-fiber bundle $M$ over the surface $S$, which factor through the bundle projection $\pi_*: \pi_1 M \rightarrow \pi_1 S$. We emphasize all the indeterminates: $X$, $F$, $M$, and the explicit conditions on the geometric structures. 
Our main result is as follows. 
The manifold on which the geometric structures live is the $(\mathbb{S}^1 \times \mathbb{S}^1 \times \R_+)$-bundle $\mathscr{C} := \, UTS \oplus UTS \oplus \uline{\R_+}$ over $S$,
the direct sum of fiber bundles of the unit tangent bundle with itself and a trivial $\R_+$-bundle. In other words, $\mathscr{C}$ has fiber $\mathscr{C}_p = UT_p S \times UT_p S \times \R_+$.
We consider $(\Gtwosplit, X)$-structures on $\mathscr{C}$ for $X =( \Ein^{2,3} ,\mathscr{D})$, where $\Eintwothree = \mathbb{P}Q_0(\R^{3,4})$, the projective null quadric in $\R^{3,4}$, and $\mathscr{D}$ is the 
canonical (2,3,5)-distribution it carries. The group $\Gtwosplit$ is the full automorphism group of the pair $(\Eintwothree, \mathscr{D})$, as proven by Cartan \cite{Car10} (cf. \cite[Section 8.2]{Eva22}). Putting all of these ingredients together, along with the geometric conditions that characterize the moduli space $\mathscr{M}$ of geometric structures,
we show: \\

\textbf{Theorem A}: \emph{The descended holonomy map $\alpha: \mathscr{M} \rightarrow \Hit(S, \Gtwosplit)$ given by $[ (\dev, \hol )] \mapsto [ \, \overline{\hol} \, ]$, where $\hol = \overline{\hol} \, \circ \, \pi_*$ and $\pi: \mathscr{C} \rightarrow S$
is the bundle projection, is a homeomorphism from the moduli space $\mathscr{M}$ of cyclic-fibered, compatible, radial $(\Gtwosplit, (\Eintwothree, \mathscr{D}))$-structures on
$ UTS \oplus UTS \oplus \uline{\R_+}$ onto the $\Gtwosplit$-Hitchin component $\Hit(S,\Gtwosplit)$}. \\

The notions of \emph{cyclic-fibered, compatible,} and \emph{radial} (CCR) are technical and are defined precisely in Section \ref{GeometricStructuresDefinition}. 
However, we can think of these conditions roughly as follows. The \emph{cyclic-fibered} condition ensures that the holonomy factors through $\pi_1S$ and that each 
such geometric structure has an associated equivariant alternating almost-complex curve $\hat{\nu}: \tilde{S} \rightarrow \quadric$. 
A crucial remark here is that the existence of a $\rho$-equivariant
alternating almost-complex curve $\hat{\nu}: \tilde{S} \rightarrow \quadric$ does not imply $[\rho] \in \Hit(S,\Gtwosplit)$ (\cite{CT23} Theorem 5.9 and Theorem 5.11), as was the case in the analogous setting 
of $\rho$-equivariant affine spheres $\tilde{S} \rightarrow \R^3$ with $\rho \in \Hom(\pi_1S, \mathsf{SL}_3\R)$. The notions of 
\emph{compatible} and \emph{radial} constrain the relationship between the developing map and the almost-complex curve further. In particular, the CCR conditions serve both to 
ensure that such a geometric structure is completely determined by its associated almost-complex curve and also that the associated almost-complex curves have $\Gtwosplit$-Hitchin 
holonomy. 

On the other hand, Guichard and Wienhard gave an \emph{abstract} geometric structures interpretation of every Hitchin component in \cite{GW12}. There, they proved the following existence theorem: for any split real simple (adjoint) Lie group $G$, there is a compact manifold $M$ with a homomorphism $\phi: \pi_1 M \rightarrow  \pi_1 S$, a homogeneous $G$-space $X$, as well as a connected component $\mathscr{G}$ of the deformation space $\mathscr{D}_{(G,X)}(M)$ of $(G,X)$-structures
on $M$, such that the holonomy of a geometric structure in $\mathscr{G}$ factors as $\hol = \overline{\hol} \circ \phi$, and the descended holonomy map $[ \, (\dev, \hol) \, ] \mapsto [ \overline{\hol} ]$ is a homeomorphism from $\mathscr{G}$ onto $\Hit(S,G)$. The manifold $M$ is realized as $M_{\rho}: = \rho(\pi_1S) \backslash \Omega_{\rho}$, a quotient of a co-compact domain of discontinuity 
$\Omega \subset X$, such that the topology of $M_{\rho}$ is independent of $\rho$. Confirming a conjecture of Guichard-Wienhard, both \cite{AMTW23} and \cite{Dav23} recently proved independently that $M$ is, in fact, a fiber bundle over $S$. 

While Guichard and Wienhard identify the spaces $X$ explicitly depending on the group $G$, the topology of the manifold $M$ is quite difficult to determine in general and so are the geometric conditions on the $(G,X)$-structures that distinguish the connected component $\mathscr{G}$.
Identifying the topology of these manifolds $M$ is an important and challenging open question. There have been some results towards this end recently \cite{ADL21, AMTW23, DS20, Dav23}.
See also the survey \cite{Kas18} for many results on geometric structures related to higher rank Teichm\"uller theory. 

Our results are somewhat orthogonal to \cite{GW12}: while our geometric structures are highly explicit, the
developing maps exhibit very complicated and strange behavior. 
In Section \ref{GW}, we examine our construction in the $\Gtwosplit$-Fuchsian case in detail. We are able to see much of the structure of the developing map and how complicated its behavior is even in this simple setting. 
In general, our developing map $\dev$ descends from $\mathscr{C}$ to $\overline{\mathscr{C}}: = UT\tilde{S} \oplus UT\tilde{S} \oplus \uline{\R_+}$. This descended developing map $\overline{\dev}: \overline{\mathscr{C}} \rightarrow \Eintwothree$ is injective 
on fibers. In the $\Gtwosplit$-Fuchsian case, we show $\overline{\dev}$ is finite-1, with each point in the image having either 1, 2, or 3 pre-images. The map $\overline{\dev}$ is very much only
locally injective -- for any $p \neq q \in \tilde{S}$, the developed fibers $\overline{\dev}(\overline{\mathscr{C}}_p)$ and $\overline{\dev}(\overline{\mathscr{C}}_q)$ intersect 1-dimensionally. We also consider the relationship between the image $U := \mathsf{image}(\overline{\dev})$ and the Guichard-Wienhard $(\Gtwosplit, \Eintwothree)$-domain $\Omega := \Omega_{\rho}$ (which is the same for all $\Gtwosplit$-Fuchsian $\rho$). Namely, $\Omega_{U} := U \cap \Omega \neq \emptyset$ is a proper open subset of $\Omega$ and $K_{U}: = U \cap (\Eintwothree \backslash \Omega) \neq \emptyset$ as well. The intersection $K_U$ is an explicitly described 3-dimensional locus. Thus, this paper does not elucidate the \cite{GW12} geometric structures in the $(\Gtwosplit, \Eintwothree)$-setting. 
However, we believe the methods here are natural, as there are strong analogies with the \cite{CTT19} construction of conformally flat Lorentz structures on $UTS$ for $\Hit(S, \mathsf{SO}_0(2,3))$. In the next part of the introduction, we explore these ideas in more detail. 

\subsection{Remarks on the Proof of the Main Theorem} 

The problem of a geometric structures interpretation is really twofold: 
\begin{enumerate}
	\item[{(1)}] Determine a construction from which a $(\Gtwosplit,X)$-structure on $M$ can be associated to any $\Gtwosplit$-Hitchin representation.
	\item[{(2)}]  Define an appropriate moduli space $\mathscr{M}$ of geometric structures containing those structures from step (1), and, crucially, so that the holonomy map on 
	$\mathscr{M}$ factors through $\Hit(S,\Gtwosplit)$. 
\end{enumerate}
We first remark on the methods used for (1). 

The construction of (1) is done with the geometry of harmonic maps, and the proof of (1) relies on Higgs bundles. Our work
in both steps is influenced by \cite{CTT19}. We summarize our perspective on \cite{CTT19} to highlight some analogies with the present work.  
Fix a complex structure $\Sigma = (S, J)$ on $S$. By the non-abelian Hodge correspondence, associated to 
$\rho \in \Hit(S,G)$ is a unique harmonic map $f_{\rho, \Sigma}: \tilde{\Sigma} \rightarrow G/K$, up to isometry. 
In the case $G \in \{ \Gtwosplit, \mathsf{SL}_3\R, \mathsf{SO}_0(2,3) \} $ of $G$ a split real simple rank two Lie group, Labourie proved in \cite{Lab17} that for each $\rho$ there is 
a unique complex structure $\Sigma$ such that $f_{\rho, \Sigma}$ is conformal. 
We denote $f_{\rho}: \tilde{\Sigma} \rightarrow G/K$ as this unique minimal surface. 
In each of the rank two settings, $f_{\rho}$ has a factorization via a Gauss map construction applied to another equivariant harmonic map $g_{\rho}: \tilde{\Sigma}  \rightarrow G/H$, whose target
is a different homogeneous space $G/H$ than that of $f_{\rho}$. For $G= \mathsf{SO}_0(2,3)$, the map $f_{\rho}$ factors through a maximal spacelike surface $\hat{\sigma}: \tilde{\Sigma} \rightarrow \hat{\Ha}^{2,2} = Q_-\R^{2,3}$, where we denote $Q_{\pm} V := \{ x \in V \; | \; q(x) = \pm 1\}$. In this case, $G/K \cong_{\mathbf{Diff}} \Gr_{(2,0)}\R^{2,3}$
and the map $f_{\rho}$ is just given as follows: $f_{\rho}(p) = d\hat{\sigma}(T_p \tilde{S})$. 
In fact, the developing map $\dev$ of the CFL $(\mathsf{SO}_0(2,3), \Ein^{1,2})$-structure of \cite{CTT19} on $UTS$ is simply described with $\hat{\sigma}$. Indeed, since the holonomy of 
$\dev$ factors through $\pi_1S$, the map $\dev$ descends from the universal cover $\widetilde{UTS}$ to define a map $\dev: UT\tilde{S} \rightarrow \Ein^{1,2}$ given by 
$\dev( p,X) = [ \hat{\sigma}(p) + d\hat{\sigma}(X)] $, where we equip $\tilde{S}$ with the pullback metric from $\hat{\sigma}$ and $[ \cdot ] $ denotes a point in projective space. Since $\hat{\sigma}$ is spacelike
and $\hat{\sigma}(T_p\tilde{S})$ is orthogonal to $\hat{\sigma}(p)$, the Pythagorean identity says $\dev( p,X) \in \Ein^{1,2} = \mathbb{P}Q_0(\R^{2,3})$, where $Q_0(V,q) = \{ v \in V \; | \; q(v)= 0\}$. Then \cite{CTT19} uses Higgs bundles along with an application of the maximum principle to the Hitchin system to prove $\dev$ is a local diffeomorphism.

Our developing map construction is a natural analogue of the developing map construction of \cite{CTT19}. 
To define $\dev$, we introduce another central object: almost-complex curves in the psuedosphere $\quadric = Q_+(\R^{3,4})$.
The space $\quadric$ has a canonical almost-complex structure $J_{\quadric}$ coming from the \emph{cross-product} $\times_{3,4}: \R^{3,4} \times \R^{3,4} \rightarrow \R^{3,4}$, whose existence is intimately linked with the group $\Gtwosplit$. Here, one can imagine $\Gtwosplit$ as defined via $\Gtwosplit := \Aut(\R^{3,4}, \times_{3,4})$. 
Starting with a Hitchin representation $\rho$, we use the harmonic map factorization in the $\Gtwosplit$-setting, which runs through the unique $\rho$-equivariant alternating\footnote{Here, alternating means loosely that the tangent space is timelike and the second fundamental form outputs only spacelike vectors.} almost-complex curve $\hat{\nu}: \tilde{S} \rightarrow \quadric$ \cite{CT23}.\footnote{We sweep aside two subtleties here: the uniqueness is due to \cite{CT23}, and we can define `almost-complex' to mean the image is an almost-complex submanifold such that $\hat{\nu}^*J_{\quadric}$ is compatible the orientation of $\tilde{S}$.} 
This harmonic map factorization is discussed in detail in \cite{CT23, Nie24} as well as in the case of $S = \mathbb{C}$ in \cite[Section 3.4]{Eva22}. We briefly recall the construction. The symmetric space $\Gtwosplit/K$ identifies as $\Gr_{(3,0)}^\times(\R^{3,4}) := \{ P \in \Gr_{(3,0)}(\R^{3,4}) \; | \; P \times_{3,4} P = P \}$. Then $f_{\rho}: \tilde{S} \rightarrow \Gtwosplit/K$ is given by $f_{\rho}(p) = \R \{\hat{\nu}(p) \} \oplus \mathsf{image}(\sff_p)$, where $\sff$ is the second fundamental form of $\hat{\nu}$. 
In our context, the developing map descends from the universal cover $\tilde{\mathscr{C}}$ to the $\pi_1S$-cover $UT\tilde{S} \oplus UT\tilde{S}  \oplus \uline{\R_+}$ of $\mathscr{C}$ and obtains the form 
\begin{align}\label{DevIntro}
	\dev(p, X, Y, r) =  \left[ \, \hat{\nu}(p)+ (r^2+1)^{1/2} \, d\hat{\nu}(X)+ r \, \frac{ \sff(\, d\hat{\nu}(X), d\hat{\nu}(Y) \,) }{ q(\sff( \, d\hat{\nu}(X), d\hat{\nu}(Y) ) \, )^{1/2} } \; \right],
\end{align} 
where $\sff$ is the second fundamental form of $\hat{\nu}$ and we equip $\tilde{S}$ with the pullback metric $ \hat{\nu}^*g_{\quadric}$.  
The essential difference between the construction here that of \cite{CTT19} is that we need the 2-jet of $\hat{\nu}$ to construct the developing map. However, the proof that \eqref{DevIntro} is a local diffeomorphism uses the same main ingredients as that of \cite{CTT19} -- namely cyclic Higgs bundles and the maximum principle, and the overall construction
is quite similar. In fact, as discuss in Section \ref{DevReinterpret}, as $r \rightarrow 0$, the developing map from \eqref{DevIntro} becomes a precise analogue of the CTT developing map, with $\hat{\sigma}$ replaced by $\hat{\nu}$. 

The developing map \eqref{DevIntro} constructs $(\Gtwosplit, \Eintwothree)$-structures associated to $\Gtwosplit$-Hitchin representations via the associated equivariant alternating almost-complex curves, solving problem (1). To establish a converse association, to solve (2), we define a moduli space $\mathscr{M}$ of geometric structures on $\mathscr{C}$ that share key technical properties with the geometric structures constructed via \eqref{DevIntro}. In particular, we demand the holonomy is trivial along the fibers. Hence, 
we may define a \emph{descended holonomy map} $\alpha: \mathscr{M} \rightarrow \Hom(\pi_1S, \Gtwosplit)/\Gtwosplit.$ 
The cyclic-fibered, compatible, and radial conditions, as mentioned earlier, serve the purpose of forcing the descended holonomy to take values in $\Hit(S, \Gtwosplit)$. 
We explain briefly the idea now. 
Each geometric structure in $\mathscr{M}$ induces an associated almost-complex curve as a consequence of the cyclic-fibered condition. That is, there is a natural continuous map $H: \mathscr{M} \rightarrow \mathcal{H}$, where $\mathcal{H} :=\mathcal{H}(S) $ the moduli space of equivariant alternating almost-complex curves in $\quadric$. In other words,
$\mathcal{H} $ is the space of pairs $(\hat{\nu}, \rho)$ with $\hat{\nu}$ a $\rho$-equivariant alternating almost-complex curve and $\rho \in \Hom(\pi_1S, \Gtwosplit)$, up to isomorphism. 
The relationship between $\dev$ and its associated curve $\hat{\nu} = H(\dev)$ is constrained by the compatibility condition. This condition implies that the second fundamental form 
$\sff$ of $\hat{\nu}$ is non-vanishing, which by \cite{CT23} is equivalent to $\hat{\nu}$ having $\Gtwosplit$-Hitchin holonomy. Finally, the radial condition along with the compatibility, pins down the relationship between $\dev$ and $\hat{\nu}$, forcing that $\dev$ is entirely determined by $\hat{\nu}$.  

Collier and Toulisse show that the sub-locus $\{ [ (\hat{\nu}, \rho)] \in \mathcal{H} \; | \; [\rho] \in \Hit(S,\Gtwosplit) \}$ of almost-complex curves that are $\Gtwosplit$-Hitchin equivariant is the level set 
$\mathcal{H}(S)_{6g-6} := \{[ (\hat{\nu}, \rho)] \in \mathcal{H} \; | \; b(\hat{\nu})= 6g-6\}$ of a discrete invariant $b$
that stratifies $\mathcal{H}$ \cite{CT23}. The space $\mathcal{H}(S)_{6g-6}$ serves as a key intermediary between representations and geometric structures; both the map $s: \Hit(S, \Gtwosplit) \rightarrow \mathscr{M}$ assigning geometric structures to representations and the descended holonomy map $\alpha: \mathscr{M} \rightarrow \Hit(S, \Gtwosplit)$ by $[ (\dev, \hol) ] \mapsto [ \, \overline{\hol} \, ]$ 
factor through $\mathcal{H}(S)_{6g-6}$. This factorization is critical to our solution of both problem (1) and (2). 
The construction of $s$ is motivated by the fact that $\alpha \circ s = \id_{ \Hit(S, \Gtwosplit) }$ essentially by definition. The remainder of the proof of the main theorem is a small technical argument to show that the geometric structure is determined entirely from its associated almost-complex curve, i.e., $s \circ \alpha = \id_{\mathscr{M}}$. See Figure \ref{GeomStrDiagram}, which diagrammatically summarizes the relationship between main maps of interest.

\begin{figure}[ht]
\centering
\includegraphics[width = .50\textwidth]{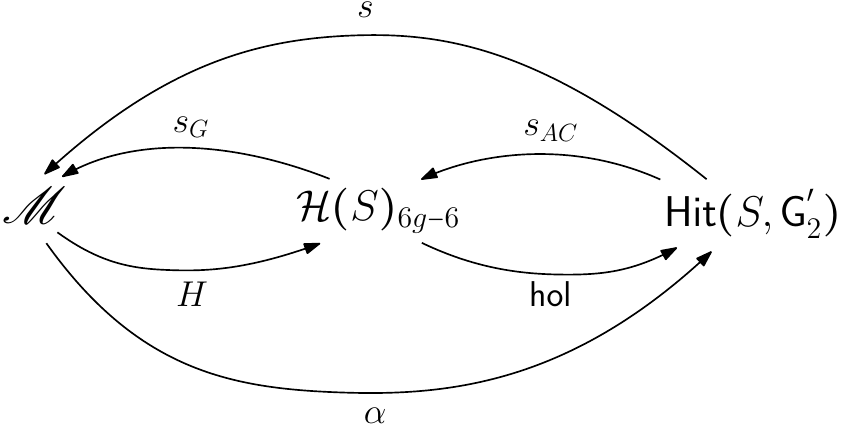}
\caption{\emph{\small{Commutative diagram summarizing the main theorem. The maps between $\Hit(S, \Gtwosplit)$ and $\mathscr{M}$ factor through $\mathcal{H(S)}_{6g-6}$. By \cite{CT23}, the  maps $s_{AC}, \hol$ are homeomorphisms. The central constructions are $s_G, H$, and $\mathscr{M}$. } }}
\label{GeomStrDiagram}
\end{figure} 

\subsection{Organization} 

We now discuss the organization of the paper. In Section \ref{Prelims}, we discuss the necessary preliminaries on $\Gtwosplit$ and Higgs bundles. 
In Section \ref{Sec:RepntoGeom}, we review the construction of $\rho$-equivariant alternating almost-complex curves for $[\rho] \in \Hit(S, \Gtwosplit)$. Using these curves,
we explicitly build an equivariant developing map on $\widetilde{\mathscr{C}}$ that defines an $(\Gtwosplit, \Eintwothree)$-structure on $\mathscr{C}$. In Section \ref{GXToRepn}, we present the technical definition of the moduli space $\mathscr{M}$ of geometric structures and reinterpret the work of Section \ref{Sec:RepntoGeom} now as defining a continuous map $s: \Hit(S,\Gtwosplit) \rightarrow \mathscr{M}$. We then 
present the descended holonomy map $\alpha: \mathscr{M} \rightarrow \Hit(S,\Gtwosplit)$ and show it is the inverse of $s$. In Section \ref{FuchsianCase}, we examine our geometric structures in the $\Gtwosplit$-Fuchsian case. We also show that our geometric structures are unrelated to the Guichard-Wienhard $(\Gtwosplit, \Eintwothree)$-structures from \cite{GW12}, and discuss some curious 
behavior of the developing map in this setting. 

\subsection{Acknowledgements} 

I would like to thank Mike Wolf and Alex Nolte for reading a draft of this paper and proving many comments. 
I also thank Brian Collier and Jeremy Toulisse for sharing a draft of their paper \cite{CT23}, which was essential to the creation of this work. 
Finally, I thank Xian Dai and Christos Mantoulidis for helpful conversations on analysis. 

\section{Preliminaries}\label{Prelims}

\subsection{$\Gtwo$ Preliminaries} \label{G2Prelims}

In this subsection, we provide a brisk introduction to $\Gtwo$ from the perspective of the octonions.
The reader can find a more comprehensive introduction to $\Gtwo$ in the excellent article \cite{Fon18}.\\

We denote $\Gtwo := \Gtwo^\C$ as the exceptional simply-connected complex Lie group with Lie algebra $\g_2^\C$. We write 
$\Gtwosplit := \Gtwo^\R$ for the split real (adjoint) form of $\Gtwo$ as well as $\Gtwo^\F$ for $\F = \R $ or $\C$ as appropriate. 
We will define the split-octonions $(\Oct')^\F$ and then see that $\Gtwo^\F$ is realized as $ \Gtwo^\F = \Aut_{\mathbf{\F-alg}}( (\Oct')^\F)$. 

For simplicity, we focus on the case of the $\R$-algebra $\Oct' := (\Oct')^\R$. To this end, we recall both the Cayley Dickson process CD and the split Cayley Dickson process CD'.
Let $\mathcal{A}$ be an $\R$-algebra with unit $1_{\mathcal{A} }$ and the following structures: a non-degenerate quadratic form $q: \mathcal{A} \rightarrow \R$ and an algebra involution
$*_{\mathcal{A}} : \mathcal{A} \rightarrow \mathcal{A}$ such that $(xy)^* = y^*x^*$ and $\Fix(*_{\mathcal{A}} ) = \R \{1_{\mathcal{A}} \}$. The Cayley Dickson processes produce a new unital algebra $\mathcal{B}$ over the vector space $\mathcal{A} \oplus \mathcal{A}$ with a non-degenerate quadratic form $q_{\mathcal{B}}$ and involution $*_{\mathcal{B}}$. First, define $*_{\mathcal{B}}(a, b) = (*_{\mathcal{A}} a, -b)$. Then given $(a,b), (c,d) \in \mathcal{A} \oplus \mathcal{A}$, the new algebra multiplication $\odot$ on $\mathcal{B}$, under CD', is given by
 \begin{align}\label{CDMultiplication}
  (a,b) \odot (c,d) := (ac + db^*, a^*d+ cb) ,
 \end{align}
 where $x^* := *_{\mathcal{A}}(x)$. Note that $1_{\mathcal{B}} = (1,0)$. We may write $xy $ for brevity to denote $x \odot y$. 
 Next, $q_{\mathcal{B}}(x) := x x^*$. 
 The multiplication in $\Oct'$ is not associative, but is \emph{alternative} -- the subalgebra generated by any two elements is associative.
 An easy consequence of the alternativity is that $q(x y) = q(x)q(y)$ for $x,y \in ( \Oct')^\F$, meaning $(\Oct')^\F$ is a \emph{composition algebra} - an algebra with unit equipped
 with a non-degenerate, multiplicative quadratic form. Composition algebras are very special. In fact, Hurwitz's theorem asserts that the only composition $\R$-algebras are $\R,\C,\Ha, \Oct$, and their split counterparts $\C', \Ha', \Oct'$.\footnote{See (\cite{HL82} Theorem A.12) for an elegant proof using the CD process, 
 in the case that the quadratic form is Euclidean (the proof is nearly the same as when $q$ is non-degenerate).}
 
 One may view the CD process in a more intrinsic way as an algebra extension of $\mathcal{A}$ by a new element $x$, given by $x= (0,1) \in \mathcal{A} \oplus \mathcal{A}$,
 which satisfies $x^2 = - 1$ in CD and $x^2 =+1$ with CD'.\footnote{In CD, the algebra formula \ref{CDMultiplication} changes by just one sign to 
 $ (a,b) \odot (c,d) := (ac - db^*, a^*d+ cb)$. However, this sign determines the signature of the quadratic form on the vector subspace $(0, \mathcal{A})$.} The pair $(a,b) \in \mathcal{A} \oplus \mathcal{A}$ corresponds to $a+xb$ and \eqref{CDMultiplication} gives the multiplicative relations between the subalgebra $\mathcal{A} := ( \mathcal{A}, 0)$ of $\mathcal{B}$ and $x$. The multiplication formula is slightly different if one instead uses $(a,b)$ to denote $a+bx$. In the case that $q_{\mathcal{B}}$ is multiplicative, $q_{\mathcal{B}}(a,b) = q_{\mathcal{A}}(a) \pm q_{\mathcal{A}}(b)$, according to CD or CD', respectively.
 
 The split octonions $\Oct' $ may be defined as the output of CD' applied to the quaternions $\Ha$. In fact, we have a sequence of $\R$-algebras 
 $\R \stackrel{\mathsf{CD}}{\longrightarrow} \C  \stackrel{\mathsf{CD}}{\longrightarrow} \Ha \stackrel{\mathsf{CD'}}{\longrightarrow} \Oct'$.\footnote{Here, $*_{\R}: \R \rightarrow \R$ is just 
 the identity map.} While we maintain this perspective on $\Oct'$, it is worth mentioning that one may also realize $\Oct'$ via $\R \stackrel{\mathsf{CD}}{\longrightarrow} \C  \stackrel{\mathsf{CD'}}{\longrightarrow} \Ha' \stackrel{\mathsf{CD}}{\longrightarrow} \Oct'$, where we pass through the \emph{split} quaternions $\Ha'$ (cf. \cite[Section 2.2]{CT23} for more details). We will explain shortly a simpler way to multiply in $\Oct'$ rather than use the formula \eqref{CDMultiplication}. 

The $+1$-eigenspace of $*_{\Oct'}$ defines a distinguished real subalgebra isomorphic to $\R$, which by abuse we refer to as $\R := \R \{ 1_{\Oct'} \}$. The $(-1)$-eigenspace of $*_{\Oct'}$ is 
called the \emph{imaginary} split-octonions, denoted $\imoct$, in analogy with the imaginary complex numbers. We may write $\mathsf{Re}, \mathsf{Im}$ for the orthogonal projections from $\Oct'$ onto the subspaces
$\R$ and $\imoct$, respectively. The non-degenerate quadratic form $q$ on $\Oct'$ is of split signature (4,4) since $(\Ha, 0) \cong \R^{4,0}$ and $(0, \Ha) \cong \R^{0,4}$ as normed vector spaces. 
In a standard abuse, we write $q$ for the bilinear form $q(x,y) = \mathsf{Re}(x y^*)$ as well as the induced quadratic form. For notational simplicity, we may write $x \cdot y $ rather than $q(x,y)$ if the context is clear.

Since any algebra automorphism of $\Oct'$ fixes the real axis pointwise,
it is standard to consider the action of $\Gtwosplit$ on $\imoct$ instead of $\Oct'$. In fact, the representation $\g_2' \rightarrow \mathfrak{gl}(\imoct)$ is one of the two fundamental representations of $\g_2'$ -- the other is the adjoint representation. As a consequence, the representation of $\g_2'$ on $\imoct$ is the lowest dimensional irreducible representation. 
Moreover, $\imoct$ carries two more algebraic structures of interest: a cross-product and a calibration 3-form.

The cross-product $\times: \imoct \times \imoct \rightarrow \imoct$ is defined by
\begin{align}
	x \times y : = \mathsf{Im}(xy) = x \odot y \; - \; (x \cdot y) 1_{\Oct'}. \label{CrossProductDefinition}
\end{align}
For $u \in \imoct$, we write $\mathcal{C}_u: \imoct \rightarrow \imoct$ as the cross-product endomorphism $\mathcal{C}_u(v) = u \times v$ of $u$. 
The \emph{double cross-product identity} says that 
\begin{align}\label{DCP}
	u \times (u \times v) = -q(u) v + q(u,v) u. 
\end{align} 
Note that $\imoct$ is not closed under $\odot$, so that $\times$ is the default binary operation on $\imoct$. 
Denote the cross-product of sets $A, B \subset \imoct$ as usual: $A \times_{\imoct} B := \{ a \times b \; | \; a \in A, \, b \in B\}$.  
We attempt to write $\times_{\imoct}$ whenever necessary so as to distinguish the cross-product of sets from the ordinary set-theoretic product. 
The map $\times$ is called a cross-product since $\times$ is bilinear, alternating, and normalized by $q(x \times y) = q(x)q(y) - (x\cdot y)^2$, just like the standard cross-product $\times_{\R^3}$ on $\R^3$. 
Extending $\Phi \in \GL(\imoct)$ linearly to $\tilde{\Phi} \in \GL(\Oct')$ by demanding $\tilde{\Phi}|_{\imoct} = \Phi$ and $\tilde{\Phi}(1_{\Oct'}) =1_{\Oct'}$, one finds that 
\begin{align}\label{G2CrossProduct}
	\Gtwosplit = \{ \Phi \in \GL(\imoct) \; | \; \Phi(u \times v) = \Phi(u) \times \Phi(v) \}.
\end{align} 
In fact, (non-obviously) if $\Phi \in \GL(\imoct)$ satisfies $\Phi(u \times v) = \Phi(u) \times \Phi(v)$, then $\Phi \in O(3,4) = O(\imoct, q)$ and $ \Phi \in \mathsf{SL}(\imoct)$ \cite{Fon18}. 
Since $\Gtwosplit$ is connected, this means $\Gtwosplit < \mathsf{SO}_0(3,4)$, where $G_0$ denotes the connected component of the identity of $G$.
Moreover, equation \eqref{G2CrossProduct} yields a description of the Lie algebra as the set of derivations of the cross product: $$\g_2'  = \{ \varphi \in \gl(\imoct) \; | \; \varphi(u \times v) = \varphi(u) \times v + u \times \varphi(v) \} .$$ 

Finally, $\Oct'$ carries a $\Gtwosplit$-invariant 3-form $\Omega$ given by the scalar triple product: $\Omega(x,y,z) = (x \times y) \cdot z$. The form $\Omega $ is a calibration in the sense of Harvey-Lawson \cite{HL82}, and is also \emph{generic} in the sense that 
its $\GL_7\R$ orbit by pullback is open in the space $\Lambda^3((\R^7)^*)$ of 3-forms.\footnote{Remarkably, there are just two open $\GL_7\R$ orbits
in $\Lambda^3((\R^7)^*)$: one is $\Omega$, the calibration 3-form for $\imoct$, and the other is $\Omega_c$, the calibration 3-form defined completely analogously, 
on $\mathsf{Im} \, \Oct$, where $\Oct = \mathsf{CD}(\Ha)$ is standard octonions. See (\cite{Fon18} Theorem 4.9.)} The 3-form provides another perspective on $\Gtwosplit $ as $\Gtwosplit = \Stab_{\GL(\imoct)}(\Omega)$. The surprising fact is that one may recover $q|_{\imoct}$ and $\times$ just from $\Omega$. See \cite{Fon18} for further details on $\Gtwo$ from the 3-form perspective. Going forward,
we will take for granted the equivalence of these three perspectives on $\Gtwosplit$: 
\begin{enumerate}
	\item $ \Aut_{\R-\mathbf{algebra}}(\Oct') $
	\item $\Aut(\R^{3,4}, \times_{3,4})$
	\item $ \Stab_{\GL_7\R}\Omega$.
\end{enumerate} 

The complex split-octonions $(\Oct')^\C = \Oct' \otimes_{\R} \C$ form a $\C$-algebra and carry an algebra product $\odot$ and quadratic 
form $q^\C$ by complex linear and bilinear extension, respectively, of the same structures on $\Oct'$.\footnote{We exclusively use $i$ to denote a split-octonion and
use $\sqrt{-1}$ for the new scalar added by the complexification.}  The cross-product and 3-form are defined analogously on $\imoct^\C$. \\

We now define a ``standard'' basis for $\Oct'$.\footnote{One can consider this basis as standard insofar as one views $\Oct'$ as being
defined via the Cayley-Dickson sequence $\R \stackrel{\mathsf{CD}}{\longrightarrow} \C  \stackrel{\mathsf{CD}}{\longrightarrow} \Ha \stackrel{\mathsf{CD}'}{\longrightarrow} \Oct'$,
which gives distinguished algebra generators $i, j, l$ for $\Oct'$, where $\C = \R[i], \Ha = \C[j], \Oct = \Ha[l]$.}
 Consider, for abuse of notation, $i = (i,0)$, $j := (j, 0), l:= (0, 1) $ in $\Ha \oplus \Ha$. 
Defining the element $k := i\odot j$, we obtain a vector space basis 
\begin{align} \label{MultiplicationBasis}
	\mathcal{M} = (1, i, j, k, l, li, lj, lk) = (m_i)_{i=0}^7
\end{align} 
for $\Oct'$. We will call this basis 
$\mathcal{M} = (m_i)_{i=0}^7$ the \emph{standard multiplication basis} for $\Oct'$. The multiplication table for $\Oct'$ in this basis is shown in Figure \ref{SplitOctonionMult}.

\begin{figure}[ht]
\centering
\includegraphics[width = .5\textwidth]{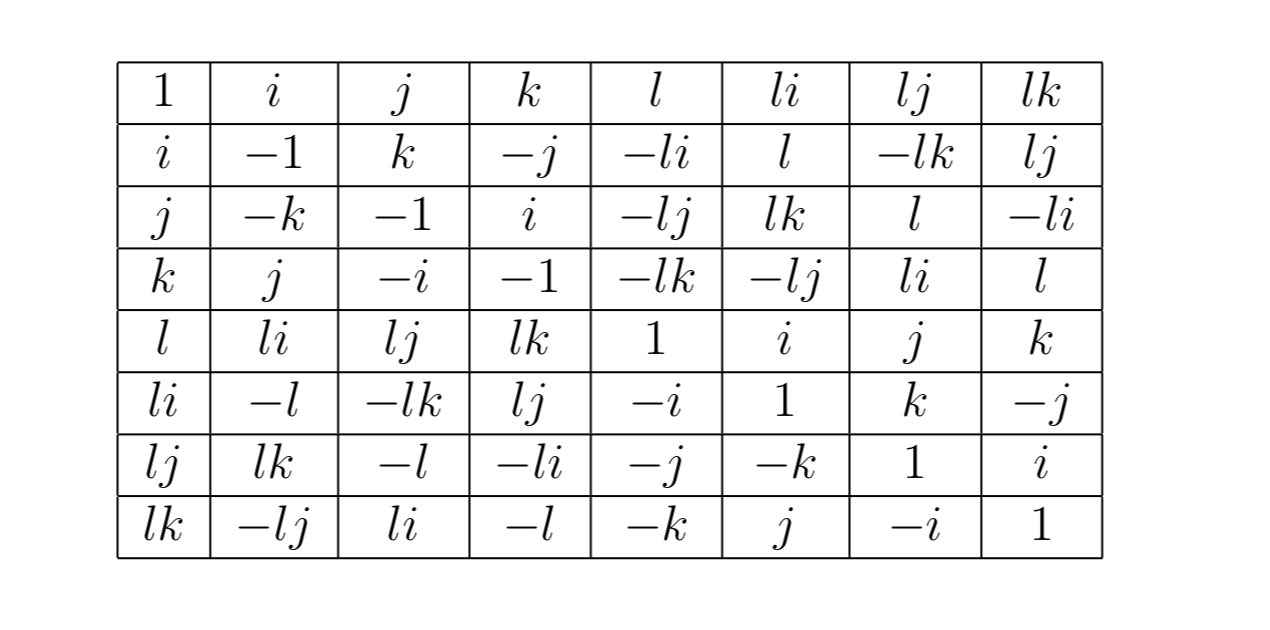}
\caption{\small{Multiplication table for the split octonions $\Oct'$, from \cite{Bar10}.} }
\label{SplitOctonionMult}
\end{figure} 
To prove the transitivity of the $\Gtwo^\F$-action on different spaces, the following Stiefel-triplet model is often useful; this 
model is very well-known. 
\begin{proposition}[\cite{Eva22} Proposition 2.3] \label{StiefelTripletModel}
 $\Gtwo^\F$ acts simply transitively on the Stiefel manifold  
$$ V_{(+,+,-)}(\imoct^\F) :=\{ \, (x,y,z) \in (\imoct^{\F})^3 | \; q(x) = q(y) = +1= -q(z), \, x \, \cdot \, y \, =  x \, \cdot \, z =\,  y \, \cdot \, z = 0 , \, z \, \cdot \, (x\times y) = 0 \}. $$
\end{proposition} 
The idea of the Proposition is that associated to $(x,y,z) \in V_{(+,+,-)}$ is an ordered basis $\mathcal{B}_{x,y,z} := (x,y, xy, z, zx, zy, z(xy))$ for $\imoct$
and a unique map $\varphi \in \Gtwo^\F$ such that $\varphi(i) =x, \varphi(j) = y, \varphi(l) = z$. In fact, $\varphi$ is the $\F$-linear map sending the multiplication frame $\mathcal{M}$ to $\mathcal{B}_{x,y,z}$.\footnote{The statement of (\cite{Eva22} Proposition 2.3) is only for the $\Gtwosplit$ case, but the same argument goes through for $\Gtwo^\C$.} 

The pseudosphere $\quadric$ is defined as $\quadric := Q_+(\imoct)$, and it inherits a signature $(2,4)$ pseudo-Riemannian metric from the ambient vector space, where
 $Q_{\pm}(V,q) = \{ v \in V \; | \; q(v) = \pm 1\}$. 
Moreover, the cross product defines a non-integrable (cf. \cite[Lemma 3.3]{CT23} and the surrounding discussion) almost-complex structure $J = J_{\quadric}$ on $\quadric$ as follows. Take $x \in \quadric$ and identify in standard fashion $T_{x} \quadric \cong [ x^\bot \subset \imoct]$. Then define $J_{x} := \mathcal{C}_{x}$, where $\mathcal{C}_x(y) = x \times y$ is the cross-product endomorphism of $x$. The double cross-product identity
\eqref{DCP} says for $u \in \imoct$ that 
$(\mathcal{C}_u \circ \mathcal{C}_u)|_{u^{\bot}} = - q(u) \mathsf{id}_{u^{\bot}}$, so that $J$ defines an almost-complex structure. While the group of isometries of the
pseudosphere is $\Isom(\quadric) = O(3,4)$, the restricted isometry group that preserves the almost-complex structure is $\Isom(\quadric, J_{\quadric}) = \Gtwosplit$. It is easy to show $\Gtwosplit$ acts transitively on $\quadric$ by Proposition \ref{StiefelTripletModel}. 
Later, we will also denote the projective negative and positive quadrics by $\Ha^{3,3} := \mathbb{P}Q_-(\imoct) $ and $\Stwofour := \mathbb{P}Q_+(\imoct)$.  \\

Consider the space 
$$\Eintwothree :=  \mathbb{P}Q_0 \imoct = \{ \, [ v] \in \mathbb{P} \imoct \; | \; q(v) = 0\}.$$
Topologically, $\Eintwothree \cong (\mathbb{S}^2 \times \mathbb{S}^3)/ \sim$, where $(x,y) \sim (-x,-y)$. In fact, a clever identification shows
$\Eintwothree \cong_{\mathbf{Diff}} \RP^2 \times \mathbb{S}^3$ \cite[Proposition 2]{BH14}. 
The intrinsic space $\Eintwothree$ can be $\Gtwosplit$-equivariantly identified with $\Gtwosplit/P_1$, where $P_1$
is a maximal parabolic subgroup of $\Gtwosplit$. Using the latter perspective, just from the root diagram of $\g_2'$, one finds $\Gtwosplit/P_1$ carries a canonical
$(2,3,5)$ distribution $\mathscr{D} \subset T\Eintwothree$. That is, $\mathscr{D}$ is a maximally non-integrable $2$-plane distribution in the sense that $\mathscr{D}^{(1)} = \mathscr{D} + [\mathscr{D}, \mathscr{D}]$ is 3-dimensional 
and $\mathscr{D}^{(2)} = \mathscr{D}^{(1)} + [\mathscr{D}^{(1)}, \mathscr{D}^{(1)}]$ is 5-dimensional. The paper \cite{Sag06} gives a beautiful and concrete interpretation of the distribution $\mathscr{D}$,
showing how one can see $\mathscr{D}_p$ at $T_p \Eintwothree$ intrinsically in the model $\Eintwothree$, without reference to the abstract space $\Gtwosplit/P_1$. We 
recall a re-interpretation of her work shortly. A crucial fact is that $\Aut(\Eintwothree, \mathscr{D} ) = \{ \, \varphi \in \Diff(\Eintwothree) \; | \; \varphi^* \mathscr{D} = \mathscr{D} \}$, as first proven by 
Cartan \cite{Car10}. We will later consider $(G,X)$-structures for $G = \Gtwosplit$ and $X = (\Eintwothree, \mathscr{D})$. 

There is one more important algebraic construction related to $\Eintwothree$. 
\begin{definition}
Given $u \in \imoct$, we define the \textbf{annihilator} of $u$ as $$\Ann(u) := \{ v \in \imoct \; | \; u \times v = 0 \}.$$
\end{definition} 
Now, note that $\Ann(u) = \Ann(c \, u) $ for $c \in \R^*$. Hence, $\Ann( [u] )$ makes sense. 
An important fact we shall need on annihilators is stated below: 

\begin{proposition}\label{Annihilators}
For $\F \in \{\R,\C\}$, if $u \in \imoct^\F$ has $q^\F(u) = 0$, then $\Ann(u) \subset Q_0(\imoct^\F)$ and $\dim_\F \Ann(u) =3$.\footnote{If $q(w) \neq 0$, then $\Ann(w) = \R \{w\}$ is uninteresting.} 
\end{proposition} 

 Since $\Gtwo^\F$ acts transitively on $Q_0(\imoct^\F)$, one can prove this by checking it holds on a model point. For example, in the basis $B$' from \eqref{RealCrossProductBasis}, labeling the elements $(x_i)_{i=3}^{-3}$ in order, one finds $\Ann(x_3) = \spann_{\R} \langle x_3, x_2, x_1 \rangle$.\footnote{cf. \cite{Eva22} Section 8.3 on the notion of a \emph{real cross-product basis} for $\imoct$ and its relation to annihilators.} By Proposition \ref{Annihilators}, for $l \in \Eintwothree$, the projective annihilator $\mathbb{P} \Ann( l)$ is an embedded
copy of $\RP^2$ in $\Eintwothree$. These projective planes yield the distribution $\mathscr{D}$ as follows:
$\mathscr{D}_{l} = T_{l} \mathbb{P}\Ann( l)$. That is, the (2,3,5)-distribution $\mathscr{D}$ is realized at $l \in \Eintwothree$ as the tangent space of its (projective) annihilator
submanifold $\mathbb{P}\Ann( l).$ See \cite[Section 8.3]{Eva22} for more details.

 \subsection{Higgs Bundle Preliminaries} \label{HiggsBundlePreliminaries}
 In this section, we recall some definitions and theory on Higgs bundles, specifically on how they can be used to parametrize $\Hit(S, \Gtwosplit)$. 
 These Higgs bundles will serve as a useful tool in the next section for understanding the Frenet frame of the almost-complex curves of interest. 
For the definition below, let $\Sigma = (S, J)$ be a Riemann surface on $S$. We denote $\K := (T^{1,0} \Sigma)^*$ as the canonical line bundle on $\Sigma$.
 
 \begin{definition} 
 A ($\mathsf{GL}_n\C$-) \textbf{Higgs bundle} on $\Sigma$ is a pair $(\V, \varphi)$ consisting of a rank $n$ holomorphic vector bundle $\V \rightarrow \Sigma$ and a \textbf{Higgs field} $\varphi \in H^0(\K \otimes \End(\V))$. 
 
 For a degree zero holomorphic vector bundle $\V$, we call $(\V, \varphi)$ \textbf{stable} when any proper holomorphic $\varphi$-invariant sub-bundle $\mathcal{F} \subset \mathcal{E}$, a sub-bundle
 such that $\varphi(\mathcal{F}) \subset \mathcal{F} \otimes K$, 
 satisfies $\deg(\mathcal{F}) < 0$.
 Call $\V$ \textbf{polystable} when $\V = \bigoplus_{i=1}^n (\V_i, \varphi_i)$ is a direct sum of degree zero stable Higgs sub-bundles. 
 \end{definition} 
There is a more general abstract definition of a $G$-Higgs bundle for $G$ a real semisimple Lie group from \cite{GPGMiR12} that we revisit momentarily,
as well as notions of stability of $G$-Higgs bundles.\footnote{If $G$ is a complex Lie group, then $G$ is regarded a real Lie group by restriction of scalars and the definition actually ends up simplifying.} 
The Higgs bundles of interest will be seen as special cases of $\mathsf{GL}_7\C$-Higgs bundles. 

The non-abelian Hodge (NAH) correspondence, developed by Hitchin \cite{Hit87}, Simpson \cite{Sim92}, Corlette \cite{Cor88}, and Donaldson \cite{Don87}, gives a homeomorphism  $\mathsf{NAH}_{\Sigma}$ between
a moduli space $\mathcal{M}_{G}(\Sigma)$ of \emph{polystable} $G$-Higgs bundles over $\Sigma$, up to gauge equivalence, and the $G$-character variety $\chi(\pi_1 S, G):= \Hom^{\mathsf{Red}}(\pi_1S, G)//G$ of \emph{reductive} representations, up to conjugation. Here, $\rho: \pi_1 M \rightarrow G$ is \emph{reductive} if $\Ad \circ \rho: \pi_1M \rightarrow \mathsf{GL}(\g)$ is a direct sum of 
irreducible representations. The map $\mathsf{NAH}_{\Sigma}$ is defined by building a flat connection $\nabla$ on the Higgs bundle
and then taking its holonomy $\hol(\nabla) \in \chi(\pi_1S, G)$. 
We refer the reader to the survey \cite{Col19} for more details on the application of the NAH correspondence in higher rank Teichm\"uller theory. 
The relevant portion of the NAH correspondence is Hitchin's smooth parametrization of the $G$-Hitchin component $\Hit(S,G)$ for $G$ a split real simple Lie group. For $G= \Gtwosplit$, Hitchin's map is of the form $F_{\Sigma}: H^0(\Sigma, \K_{\Sigma}^2) \oplus H^0(\Sigma, \K_{\Sigma}^6) \rightarrow \Hit(S,\Gtwosplit)$. A drawback of Hitchin's parametrization $F_{\Sigma}$ (in general) is that 
 depends highly on the choice of complex structure $\Sigma$. Labourie 
rectified this situation in the case $G$ is of rank two in \cite{Lab17}. Labourie's work gives a canonical mapping class group-equivariant diffeomorphism $\Psi: \V_6 \rightarrow \Hit(S,\Gtwosplit)$
where $\V_6 \rightarrow T(S)$ is the bundle over Teichm\"uller space with fiber $(\V_6)_{ \Sigma } \cong H^0(\Sigma, \K_{\Sigma}^6)$. Labourie's map is
$\Psi( \Sigma, q_6) = F_{\Sigma}(0, q_6)$. Thus, naturally associated to $\rho \in \Hit(S, \Gtwosplit)$
is a pair $(\Sigma, q_6)$, where $q_6$ is a holomorphic sextic differential on $\Sigma$. The data of this pair will be used to define a (stable) $\mathsf{GL}_7\C$-Higgs bundle below,
whose structure group reduces to $\Gtwosplit.$

Hitchin's map $F_{\Sigma}$ factors as $F_{\Sigma} = \mathsf{NAH}_{\Sigma} \circ s_{\Sigma}$, where
$s_{\Sigma}: H^0(\Sigma, \K_{\Sigma}^2) \oplus H^0(\Sigma, \K_{\Sigma}^6) \rightarrow \mathcal{M}_{\Gtwosplit}(\Sigma)$ is called the \emph{Hitchin section}. 
We now define the Higgs bundle $s_{\Sigma}(0, q_6) := (\V, \varphi(q_6))$ of interest.
The holomorphic vector bundle underlying the Higgs bundle is $\V = \bigoplus_{i=3}^{-3} \K_{\Sigma}^i $. Note that $\K_{\Sigma}^0 = \mathcal{O}$ is a trivial complex line sub-bundle
of $\V$. Then we identify in any
local coordinate $z$ for $\Sigma$ the local frame $(dz^i)_{i=3}^{-3}$ for $ \V$ and the following fixed frame $(u_i)_{i=3}^{-3}$ for $\imoct^\C$ from \eqref{BaragliaBasis} below
by $dz^i \leftrightarrow u_i$. 
\begin{align} \label{BaragliaBasis}
	\begin{cases} 
	u_{3} &= \frac{1}{\sqrt{2}}( \, jl + \sqrt{-1} kl).\\
	u_{2} &=\frac{1}{\sqrt{2}}( \, j + \sqrt{-1} k).\\
	 u_{1} &= \frac{1}{\sqrt{2}}( \, l + \sqrt{-1} il).\\
	 u_{0} &= i.\\
	 u_{-1} &= \frac{1}{\sqrt{2}}( \, l- \sqrt{-1} il).\\
	 u_{-2}&= \frac{1}{\sqrt{2}}( \, j - \sqrt{-1} k).\\
	 u_{-3} &= \frac{1}{\sqrt{2}}( \, jl - \sqrt{-1} kl).
	 \end{cases} 
\end{align}
This identification locally defines a cross-product on the fibers of $\V$, which yield a global cross-product $\times_{\V}: \V \times \V \rightarrow \V$ since the local transitions 
between $(dz_i)_{i=3}^{-3}$ and $(dw_i)_{i=3}^{-3}$ are of the form $\mathsf{diag}(\zeta^3, \zeta^2, \zeta, 1, \zeta^{-1}, \zeta^{-2}, \zeta^{-3})$ for $\zeta \in \C^*$ and hence respect the cross-product. Indeed, the $\g_2^\C$-transformations diagonal the basis $(u_i)$ are of the form $\mathsf{diag}(r+s, r, s, 0, -s,-r,-r-s)$ for $r,s \in \mathbb{C}$ \cite[Proposition 2.6]{Eva22}. By considering only frames for $\V$ that are fiber-wise cross-product compatible with the above model frames,
we reduce the structure group of $\V$ to $\Gtwo^\C$. That is, we consider the holomorphic principal $\Gtwo^\C$-frame bundle
$$ \Fr^{\times}(\V) :=  \{T \in \Hom_{\C}(\,\uline{\imoct^\C}\,, \V)  \; | \;  T(u \, \times_{\imoct^\C}  \,v) = T(u) \times_{\V} T(v)  \},$$
where $\uline{E}:= \Sigma \times E $ denotes a trivial vector bundle. 

We may write $q$ instead of $q_6$ when there is no opportunity to confuse $q_6$ with $q= q_{\imoct}$ on $\imoct$. 
Following Hitchin's original Lie algebra recipe, one finds (cf. \cite[Section 2.6, Appendix A]{Eva22}) the Higgs field $\varphi := \varphi(0,q)$ is of the form 
\begin{align}
		\varphi &= \begin{pmatrix} 0&  & & & & q&\\
				     \sqrt{3} & 0&  & & & & q\\
			  	      & \sqrt{5}&0 & & & &\\
			 	      & & -\sqrt{-6}&0 && &\\
				      & & & -\sqrt{-6}& 0& &\\
				      & & & &\sqrt{5} &0&\\
				      & & & & & \sqrt{3}&0\\
				     \end{pmatrix} .	
	\end{align}
Given a Higgs bundle $(\V, \varphi) \in \mathsf{image}(s_{\Sigma})$, \emph{Hitchin's equations} \eqref{Hitchin1}, \eqref{Hitchin2} for $(\V, \varphi)$ uniquely determine a hermitian metric $h$ on $\V$ such that the connection $\nabla := \nabla_{\delbar, h} + \varphi + \varphi^{*h}$ is flat and has holonomy in $\Gtwosplit$ \cite{Hit92}. 
\begin{align}	
 	F_{h} + [\varphi, \varphi^{*h} ] &= 0 .\label{Hitchin1} \\
	\nabla_{\delbar, h}^{0,1} \, \varphi &= 0. \label{Hitchin2} 
\end{align}
The metric $h$ is diagonal in the standard holomorphic coordinates since $(\V, \varphi)$ is \emph{cyclic} \cite{Bar10}. 
In particular,  
$h$ is of the form 
\begin{align}\label{HitchinMetric}
	h( \mathbf{x}, \mathbf{y} ) = \mathbf{\overline{x}}^T \, \mathsf{diag} (r^{-1}s^{-1}, r^{-1}, s^{-1}, 1, s, r, rs)\, \mathbf{y} =  \mathbf{\overline{x}}^T \,H \, \mathbf{y},
\end{align} 
where $r, s > 0$ are positive. Here, $s \in \Gamma(\K \otimes \overline{\K})$ and $r \in \Gamma(\K^2 \otimes \overline{\K}^2)$. In a local coordinate $z$, %with $\Delta :=  \del_z \circ \del_{\zbar}$, 
Hitchin's equation \eqref{Hitchin2} is equivalent to $q_{\zbar} = 0$ and 
\eqref{Hitchin1} is equivalent to the coupled elliptic system of PDE \eqref{HitchinsEquations_rs} \cite[page 16]{Eva22}. 
 \begin{align}\label{HitchinsEquations_rs}
 	\begin{cases} 
 		\partial_{z} \partial_{\zbar} \log r  =  5 \frac{r}{s} - 3s- \, \frac{|q|^2}{r^2 s} \\%\label{HitchinsEquation1} \\ 
		\partial_z \partial_{\zbar} \log s  = 6s-  5\frac{r}{s} .%\label{HitchinsEquation2}
	\end{cases} 
 \end{align} 
Denote $u_1= \log r + \frac{1}{2}\log s$, $v_1 = \frac{1}{2}\log s$ and the system \eqref{HitchinsEquations_rs} becomes 
  \begin{align}\label{HitEuc}
 	\begin{cases}  
		 2\partial_{z} \partial_{\zbar}  \, u_1= 5 e^{(u_1-3u_2)} - 2 e^{-2u_1}\,|q|^2 .\\
		2 \partial_{z} \partial_{\zbar} \, u_2 =  6e^{2u_2} -5e^{(u_1-3u_2) }. \end{cases} 
 \end{align} 
 We can also write Hitchin's equations in a global form. Let $\sigma = \sigma(z) |dz|^2$ denote any conformal metric on $\Sigma$. Given a (local) solution
$\boldsymbol{u} = (u_1, u_2)$ of \eqref{HitEuc}, we define $\boldsymbol{\psi} = (\psi_1, \psi_2)$ by 
 $$\begin{cases}
 	e^{\psi_1} \sigma^{5/2} &= e^{u_1} \\
	e^{\psi_2} \sigma^{1/2} &= e^{u_2} .
 \end{cases} $$
 Then one finds $\boldsymbol{u}$ solves \eqref{HitEuc} if and only if $\boldsymbol{\psi}$ solves \eqref{G2Hitchin_GeneralMetric}. 
Thus, solving Hitchin's equations is equivalent to finding global functions $\psi_1, \psi_2: S \rightarrow \R $ that solve \eqref{G2Hitchin_GeneralMetric}. 
\begin{align}\label{G2Hitchin_GeneralMetric} 
	\begin{cases}  
		2 \Delta_\sigma \psi_1 &= 5e^{\psi_1 - 3 \psi_2} - 2 |q|^2_{\sigma} \, e^{-2\psi_1} + \frac{5}{2} \kappa_{\sigma} \\
		2 \Delta_\sigma \psi_2 &= -5e^{\psi_1 - 3\psi_2} + 6 e^{2\psi_2} + \frac{1}{2} \kappa_\sigma  \end{cases} ,
\end{align} 
where $\Delta_{\sigma} = \frac{1}{\sigma} \partial_{z} \partial_{\zbar}$, $|q|^2_\sigma := \frac{q \bar{q}}{\sigma^6}$, and $\kappa_{\sigma} = - \frac{2}{\sigma} \partial_{z} \partial_{\zbar} \log \sigma$. 
In particular, there is a unique solution to the system \eqref{G2Hitchin_GeneralMetric}. \\
 
 We now discuss $\nabla$-parallel real forms on $\V$ and $\End(\V)$. 
 For the following discussion, we imagine the following structures on $\g_2^\C$ as being fixed:
 \begin{itemize}
 	\item the Cartan subalgebra $\mathfrak{h}$ of diagonal transformations in the basis $(u_i)$ from \eqref{BaragliaBasis}
	\item a choice of simple roots $\Pi= \{\alpha, \beta \} \subset \Delta$ 
	\item Chevalley basis $(t_{\delta}, e_{\delta})_{\delta \in \Delta}$ such that $[t_{\delta}, e_{\delta} ] = 2 e_{\delta},\, [t_{\delta}, e_{-\delta} ] = -2 e_{\delta}, \, [e_{\delta}, e_{-\delta}] = t_{\delta}$.\footnote{For an explicit choice of such basis, see \cite[Appendix A]{Eva22}.}
\end{itemize} 
  First, consider 
 $$\End^{\times}(\V) := \{ \psi \in \End(\V) \; | \; \psi(u \times_{\V} v) = \psi(u) \times_{\V} v + u \times_{\V} \psi(v)\; \},$$
 a $\g_2^\C$-fibered sub-bundle of $\End(\V)$. There is a natural vector bundle isomorphism $\End^{\times} \V \cong \bigoplus_{i=-5}^{5} \uline{\g_i} \times \K^i$, where $\g_i = \bigoplus_{\delta \in \Delta, \mathsf{height}_{\Pi}(\delta) =i} \g_\delta$ is the sum of root spaces of height $i$,
 and $\uline{\g_i}$ denotes the trivial bundle $\Sigma \times \g_i$. For example, 
consider the $\g_2^\C$ transformation $e_{\alpha} \in \g_{\alpha}$ given by $u_1 \mapsto u_{2}$,\, $u_{-2} \mapsto u_{-1}$, and $u_i \mapsto 0$ otherwise.
Given both the constant section $\widehat{e_{\alpha}} \in \Gamma(\Sigma, \uline{\g_i})$ by $\widehat{e_{\alpha}}(p) = e_{\alpha}$, and $X \in \Omega^{1,0}(\Sigma)$, 
the pair induces the transformation $(e_{\beta} \otimes X) \in \End^{\times}(\V)$ by mapping $\K^1 \mapsto \K^2$ via $Y \mapsto X \otimes Y$ and mapping $\K^{-2} \mapsto \K^{-1}$ via 
$Y \mapsto X \otimes Y$. Using this identification, we can reduce the structure group of $\End^{\times}(\V)$ from $\Gtwo^\C$ to $\Gtwosplit$ via real forms (and do the same for $\V$.)

The hermitian metric $h$ induces a fiber-wise choice of compact real on $\End^{\times}(\V)$. First, note that any $A_p \in \End(\V_p)$ has
a unique $h$-adjoint $A_p^{*h} \in \End(\V_p)$ satisfying $ h( A_p X, Y) = h(X, A_p^{*h} Y)$. Thus, we may define a $\C$-anti-linear Lie algebra bundle involution 
$\hat{\rho}: \End(\V_p) \rightarrow \End(\V_p)$ by $\hat{\rho}(A) = -A^{*h}$. The involution $\hat{\rho}$ leads to the sub-bundle $\End^{\times}_{\hat{\rho}}(\V): = \{ X \in \End^{\times}(\V) \; | \; \hat{\rho}(X) = X\}$ of $\mathfrak{su}(h)$ endomorphisms. Fiberwise, the map $\hat{\rho}$ is the involution of a compact real form of $\g_2^\C$. 
In the local coordinates $(dz^i)$, the involution obtains the form 
$\hat{\rho}(A) = -H^{-1}\overline{A}^T H$. 

Similarly, we define a bundle-wise Cartan involution, whose fixed point set is fiberwise a copy of $\frakk^\C$, where $\frakk < \g_2'$ is a 
copy of the maximal compact in $\g_2'$. Unlike $\rho$, the map $\sigma$ is $\C$-linear. On the level of $\g_2^\C$, first define 
$\sigma(e_{\alpha} ) = (-1)^{\mathsf{height}(\alpha)} e_{\alpha}$ and $\sigma|_{\mathfrak{h} } = \id_{\mathfrak{h}}$, where $\mathfrak{h} = \g_0$ is the Cartan subalgebra. 
Hitchin proved this map is a Lie algebra involution whose fixed point set is $\mathfrak{k}^\C$ in \cite[Proposition 6.1]{Hit92}. 
We extend $\sigma$ to a bundle involution on $\hat{\sigma}: \End^{\times}(\V) \rightarrow \End^{\times}(\V) $ by defining it on simple tensors via $ \hat{\sigma}( e_{\alpha} \otimes X ) = \sigma(e_{\alpha } ) \otimes X$. 
In the coordinates $(dz^i)$, the involution $\hat{\sigma}$ is of the form $\hat{\sigma}(A) = -QA^TQ $ 
where $Q= \begin{pmatrix} & & & & & &1\\
					& & & & &1&\\
				   	& & & & 1& & \\
				  	& & & 1& & & \\
				  	& & 1&& & &\\
				   	& 1& & & & & \\
				   	1 & & & & & & \end{pmatrix}.$ 
The involutions $\hat{\rho}, \hat{\sigma}$ pointwise commute and hence define a $\C$-anti-linear conjugation $\hat{\tau}$ of a split real form of $\g_2^\C$. 
In fact, $ \End^\times_{\hat{\tau}}(\V) = \{ \psi \in \End^\times \V \; | \; \hat{\tau}(\psi) = \psi \}$ is a $\nabla$-parallel sub-bundle of $\End(\V)$. 
Indeed, $\hat{\tau}$ is $\nabla_{\delbar,h}$-parallel \cite{Hit92} and $\varphi + \varphi^{*h}$ is directly found to satisfy $\hat{\tau}(\varphi + \varphi^{*h}) = \varphi + \varphi^{*h}$.

There is one more structure on $\V$ of interest, a non-degenerate complex bilinear form $B: \V \times \V \rightarrow \C$. 
Write $\V = \mathcal{O} \oplus \bigoplus_{i=1}^{3} (\K^i \oplus \K^{-i})$. Then set $B|_{\mathcal{O}} = 1$ and $B|_{\K^i \oplus \K^{-i}} = \begin{pmatrix} 0 & 1 \\ 1 & 0\end{pmatrix}$ to be
the natural dual pairing. We then declare $\mathcal{O}$ and each sub-bundle $\K^i \oplus \K^{-i} $ for $i \in \{1,2,3\}$ to be $B$-orthogonal to each other. In the frame $(dz^i)$,
the bilinear form $B$ is represented by the matrix $Q$. 
We are then led to consider $B$-orthonormal frames that also respect the cross-product. That is, define 
$$\Fr_{B}^{\times}(\V) = \{ T \in \Fr^{\times}(\V) \; | \; T(e_i) \; \text{is} \; B-\text{orthonormal}\}.$$
The bilinear form $B$ is also related to the involution $\hat{\sigma}: \End(\V) \rightarrow \End(\V)$. In fact, $\hat{\sigma}$ is equivalently given by $\hat{\sigma}(A) = -A^{*B}$, where 
$(\cdot)^{*B}$ denote the $B$-adjoint, i.e., $B(As, t) = B(s, A^{*B} t)$. \\

Now, we give the definition of a general $G$-Higgs bundle for $G$ a real reductive Lie group. Start with a Cartan decomposition $\g = \mathfrak{k} \oplus \mathfrak{p}$,
where $\mathfrak{k}$ is a maximal compact subalgebra of $\g$. Then complexify to get $\g^\C = \mathfrak{k}^\C \oplus \mathfrak{p}^\C$. A $G$-Higgs bundle
is then a pair $(\mathcal{P}, \Phi)$, with $\mathcal{P}$ a holomorphic principal $K^\C$-bundle and $\Phi \in H^0( \mathscr{E} \otimes \K_{\Sigma})$ is the Higgs field,
where $\mathscr{E} = \mathcal{P} \times_{\Ad} \mathfrak{p}^\C$. 

 In the present case, $s_{\Sigma}(0, q_6)$ is a $\Gtwosplit$-Higgs bundle, just the construction just has been factored
through the holomorphic vector bundle $\V$. Here, the principal $K^\C$-bundle is $\mathcal{P} = \mathsf{Fr}^\times_{B}(\V)$. Then view $\mathcal{P} \times_{\Ad} \g^\C \cong \End^{\times} (\V)$ and we find $\mathcal{P} \times_{\Ad} \mathfrak{p}^\C \cong \End^\times_{-\hat{\sigma}}(\V) =  \{ \,X \in \End^\times(\V) \; | \; \hat{\sigma}(X) = -X\}$.
Since $\hat{\sigma}(\varphi) = - \varphi $ and $\varphi( \der{z} ) \in \End^\times(\V)$ in the local frame $(dz^i)$, the Higgs field is just $\Phi = \varphi$. Hitchin proved in \cite{Hit92} that the Higgs bundles in the Hitchin section are stable. We refer the reader to \cite{CT23} for more 
comprehensive details on cyclic $\Gtwosplit$-Higgs bundles and to \cite{Col19} for more details on stability. \\

There is an involution $\hat{\tau}: \V \rightarrow \V$, by abuse also called $\hat{\tau}$, such that $\V^\R := \Fix(\hat{\tau})$ is a $\nabla$-parallel sub-bundle with fibers isomorphic to 
$\imoct$. The real forms $\hat{\tau}$ are multiplicative in the sense that $\hat{\tau}(\psi s) = \hat{\tau}(\psi)\hat{\tau}(s) $ for any sections $\psi \in \Gamma(\Sigma, \End(\V)\, ), s \in \Gamma(\Sigma, \V)$
\cite{Bar10}. Locally, $\hat{\tau}$ is found by $\hat{\tau}(x) = H^{-1}Q\overline{x}$, where $x $ is a vector in the local coordinates $(dz^i)$. 
More explicitly, the sub-bundle $\V^\R$ is realized as the (real) span of the basis $(w_i)_{i=1}^7$
from \eqref{HUnitaryMultiplicationFrame}. 
\begin{align} \label{HUnitaryMultiplicationFrame}
		\begin{cases}
		w_1 &= i = u_0.\\
		w_2 &= \,\frac{1}{\sqrt{2}} \, ( \, r^{1/2} u_{2} + r^{-1/2} u_{-2} \, ) \\ %  \;   \; \; \sim j$ \\ % in the model case. \\
		w_3 &=  - \, \frac{ \sqrt{-1} }{\sqrt{2}} ( \, r^{1/2} u_{2} - r^{-1/2} u_{-2} \, ) \\ % $ \; \; $\sim k$% in the model case.\\
		w_4& = \,\frac{1}{\sqrt{2}} \, ( \, s^{1/2} u_{1} + s^{-1/2} u_{-1} \, ) \\% \; ~ \; $\sim l$% in the model case.
		w_5& = \,\frac{ \sqrt{-1} }{\sqrt{2}} \, ( \, s^{1/2} u_{1} - s^{-1/2} u_{-1} \, ) \\% $ \; ~ \; $\sim li$% in the model case. 
		w_6& = -\,\frac{1}{\sqrt{2}} \, ( \, (rs)^{1/2} u_{3} + (rs)^{-1/2} u_{-3} \, ) \\ % \; ~ \; $\sim lj$% in the model case.
		w_7 &= \,\frac{ \sqrt{-1} }{\sqrt{2}} \, ( \, (rs)^{1/2} u_{3} - (rs)^{-1/2} u_{-3} \, ) \\ % $ \; ~\; $\sim lk$% in the model case.
		\end{cases} 
	\end{align}%Here, $(dz^i)_{i=3}^{-3}$ is a local frame for $\V$, where $dz^i \leftrightarrow u_i$ are identified and $(u_i)_{i=3}^{-3}$ is a fixed basis for $\imoct$.  
We also recall that $(w_i)_{i=1}^7$ is an $h$-unitary multiplication frame for $\V^\R$ -- that
is, the $\C$-linear map $(m_i \mapsto w_i) \in \Gtwo^\C$, where $(m_i)$ is from \eqref{MultiplicationBasis} \cite[Section 3]{Eva22}. Hence, $(w_i)_{i=1}^7$ is orthonormal under the $\imoct$ norm with 
$q(w_i) = +1$ for $i \in \{1,2,3\}$ and $q(w_i) = -1$ for $i \in \{4,5,6, 7\}$. 

\section{From Representations to Geometric Structures} \label{Sec:RepntoGeom}

Let $[\rho] \in \mathsf{Hit}( S, \Gtwosplit) $ be a Hitchin representation. In this section, we construct a $(G,X)$-structure on the 
($\mathbb{S}^1 \times \mathbb{S}^1 \times \R_+$)-bundle $\mathscr{C}:= UTS \oplus UTS \oplus \uline{\R}_+$ over $S$ with $G = \Gtwosplit$
and $X = (\Eintwothree, \mathscr{D})$, where $\mathscr{D}$ is the (2,3,5)-distribution on $\Eintwothree$ defined in Section \ref{G2Prelims}, such that the holonomy $\mathsf{hol}$ of the geometric structure factors through the projection $\pi: \mathscr{C} \rightarrow S$ and satisfies
$\hol = \rho \circ \pi_*$. Later, the work of this section will be repackaged as a map $s: \Hit(S, \Gtwosplit) \rightarrow \mathscr{M}$, once we have given the technical
conditions defining the moduli space $\mathscr{M}$. 

Recall from the previous subsection that  
associated to $\rho \in \Hit(S,\Gtwosplit)$ is a pair $(\Sigma, q_6) $ in the bundle $ \V_6 \rightarrow T(S)$. 
We then realize $\rho$ as the holonomy of the flat connection $\nabla$ associated to the Higgs bundle $(\V, \varphi_{q_6})$ 
on $\Sigma = (S,J)$. 
The geometric structure on $\mathscr{C}$ is first constructed with Higgs bundle methods, but we then describe the developing map
just in terms of the unique $\rho$-equivariant alternating almost-complex curve $\hat{\nu}: \tilde{S} \rightarrow \Stwofour$. 

\subsection{The Almost-Complex Curve $\hat{\nu}$ and Associated Maps} \label{AC_Curve} 
In this section, we define the almost-complex curves $\hat{\nu}$ of interest as well as a number of relevant associated maps. 
As an important foundational principle, we review an important 
intrinsic-extrinsic equivalence that allows us to view these curves $\hat{\nu}$ inside the aforementioned Higgs bundles. This technique comes 
from \cite{Bar10}, who applied this idea to $G$-Hitchin representations for rank two split real simple Lie groups, i.e.,
$G \in \{ \mathsf{SL}_3\R, \Sp(4,\R), \Gtwosplit\}$. \\

We first recall some notions on flatness. Let $F \rightarrow M$ be an $(X, G)$-fiber bundle. Such a bundle is necessarily of the form $F = \bigsqcup_{\alpha} (U_\alpha \times X)/\sim $, over an atlas $(U_{\alpha})$ covering $M$, where on the overlaps for each $p \in U_{\alpha} \cap U_{\beta}$, we have identified via $ (p,x) \in U_{\alpha} \times X$ with $(p, g_{\alpha \beta}(p) x) \in U_{\beta} \times X$
for some $g_{\alpha \beta}(p) \in G$. The transition data $U_{\alpha} \cap U_{\beta} \rightarrow G$ by $p \mapsto g_{\alpha \beta}(x)$ is continuous and we have the cocycle identity
$g_{\alpha \beta} g_{\beta \gamma} = g_{\alpha \gamma}$. 
When the atlas can be chosen in such a way that the transition data $g_{\alpha \beta}$ is locally constant, the bundle is \emph{flat}. 
 In this case, the local horizontal foliation in each product $U_{\alpha} \times X$ extends to a foliation of the total space of the bundle. Flat $(X,G)$-bundles can be used to construct equivariant developing maps, as we now recall. 

\textbf{Fundamental Correspondence.}\footnote{The following correspondence is well-known and was explained by Goldman in \cite{Gol88} as well as in a survey of Alessandrini on geometric structures \cite{Ale19}.}  Let $X$ be a topological space and $G \rightarrow \mathsf{Homeo}(X)$
a faithful homomorphism of Lie group $G$. 
Let $\mathscr{X}_{\rho}$ be a flat $(X,G)$-bundle over a smooth manifold $M$
with holonomy $\rho:\pi_1 M \rightarrow G$ in a fixed trivialization $\tau$. 
There is a homeomorphism between $\Gamma(M, \mathscr{X}_{\rho} )$ and the space $\mathsf{Equiv}_{\rho}(\tilde{M}, X)$ of $\rho$-equivariant maps $\tilde{M} \rightarrow X$, where each space is equipped with the $C^{\infty}$-topology. The correspondence
 $ \sigma \in \Gamma(M, \mathscr{X}_{\rho}) \mapsto f_{\sigma} \in \mathsf{Equiv}_{\rho}(\tilde{M}, X)$ is as follows. 
Pull back $\sigma \in \Gamma(M, \mathscr{X}_{\rho}) $ 
to $\tilde{\sigma} \in \Gamma(\tilde{M}, \pi^* \mathscr{X}_{\rho} ) $, where $\pi: \tilde{M} \rightarrow M$ denotes the universal covering. 
The pullback section $\tilde{\sigma}: \tilde{M} \rightarrow \tilde{M} \times X$ is a $\rho$-equivariant map in the trivialization induced by $\tau$. 
Denoting $f_{\rho} $ as the projection of $\tilde{\sigma}$ onto $X$ gives the desired map. 
If we instead allow the trivializations to vary, we get the $G$-orbit $(L_g \circ f_{\rho})_{g \in G}$ of maps associated to $\sigma \in \Gamma(M, \mathscr{X}_{\rho} )$,
where $L_g:X \rightarrow X$ denotes the $G$-action. The map $L_g \circ f_{\rho}$ is instead equivariant with respect to $C_g \circ \rho$, where $C_g: G \rightarrow G$ denotes $g$-conjugation, so 
there is only one representative from $(L_g \circ f_{\rho})_{g \in G}$ with holonomy $\rho$. 
The constructions of this section on the level of representatives will later be more naturally viewed on the level of moduli spaces in Lemma \ref{Continuity}. 

We now give the relevant definitions on equivariant alternating almost-complex curves. In particular, we use a slightly non-standard definition of `almost-complex' from \cite{CT23}, which
is justified after the definition. 

\begin{definition}
Let $\mathsf{S}$ be an oriented smooth surface. 
\begin{itemize}
	\item An \textbf{almost-complex curve} $\nu: \mathsf{S} \rightarrow \quadric$ is an immersion such that
the tangent space $d\nu(T_p\mathsf{S})$ is $J$-invariant and moreover $\nu^*J$ is compatible with the orientation on $\mathsf{S}$. 
	\item Write the pullback tangent bundle as $\nu^*T\quadric = T_{\nu} \oplus T_{\nu}^{\bot}$, where $T_{\nu}|_p= \mathsf{image}(d\nu_p)$. 
	Call $\nu$ \textbf{alternating} when $T_{\nu}|_p$ is a (0,2)-plane for all points $p \in \mathsf{S}$ and the second fundamental form $\sff: T_{\nu} \times T_{\nu} \rightarrow T_{\nu}^\bot$ of $\nu$ 
is not identically zero and is positive-definite in the sense that $\mathsf{image}(\sff_p)$ is a (possibly-empty) positive-definite subspace of $T_{\nu}^\bot|_p$. 
	\item When $\mathsf{S} = \tilde{S}$ is the universal cover of a surface $S$, call $\nu$ \textbf{equivariant} when there exists $\rho \in \Hom(\pi_1S, \Gtwosplit)$
such that $\nu(\gamma \cdot x) = \rho(\gamma) \cdot \nu(x)$ for all $x \in \tilde{S}$ and $\gamma \in \pi_1S$. 
\end{itemize} 
\end{definition} 

\begin{remark}
If $\nu: \mathsf{S} \rightarrow \quadric$ is an almost-complex curve under this definition, define $j:= \nu^*J_{\quadric}$, and then $\nu: (\mathsf{S}, j) \rightarrow \quadric$
is an almost-complex curve in the conventional sense: $d\nu \circ j = J_{\quadric} \circ d\nu$. The definition here is a minor convenience, 
since if $[\rho] \in \Hit(S,\Gtwosplit)$, there is a unique complex structure $[\Sigma] \in T(S)$ 
such that there is a $\rho$-equivariant alternating almost-complex curve $\nu: \tilde{\Sigma} \rightarrow \quadric$ in the conventional sense. This is proven by \cite[Theorem B]{CT23}, which shows that
$\mathcal{H}(S)$ fibers over Teichm\"uller space $T(S)$ via the map $[ \,(\hat{\nu}, \rho) \, ]\mapsto [ \hat{\nu}^*J_{\quadric} ] $. 
\end{remark} 

\begin{remark}
If $\sff \nequiv 0$, then $\sff$ vanishes only at isolated points and one may define a unique 2-dimensional sub-bundle $N_{\nu} \subset T_{\nu}^\bot$, called the \textbf{normal line}, such that $N_{\nu}|_p = \mathsf{image}(\sff_p)$, where $\sff_p \neq 0$ \cite[Proposition 3.15]{CT23}. However, as we shall see shortly,
$\sff$ is pointwise non-vanishing for the almost-complex curves we are interested in. The \textbf{binormal line} $B_{\nu}$ is then defined by $B_{\nu} : =\left[ (T_{\nu} \oplus N_{\nu})^\bot \subset \nu^*T\quadric \right]. $
\end{remark} 

Let $\rho \in \Hit(S,\Gtwosplit)$ and we now construct the unique $\rho$-equivariant alternating almost-complex curve $\hat{\nu}: \tilde{S} \rightarrow \quadric$
via the fundamental correspondence. This section addresses only the existence; uniqueness is explained later by Theorem \ref{HitchinAC}. The construction starts with the tautological section $\hat{\mu}: \Sigma \rightarrow \mathcal{O} \hookrightarrow \V$
by $\hat{\mu}(p) := 1 \in \mathcal{O}_p$. The almost-complex curve $\hat{\nu}$ corresponds to $\hat{\mu}$, which is seen after some intermediary identifications \cite{Bar10}. Observe that $q(\hat{\mu}) = +1$ and $\hat{\mu}_p \in \V_p^\R$ by the identification \eqref{HUnitaryMultiplicationFrame}.
Hence, view $\hat{\mu}$ as a section of the bundle $Q_+ \V^\R \cong \tilde{S} \times_{\rho} \quadric$, an $\quadric$-fibered flat bundle with fibers $ Q_+ \V_p^\R := \{x \in \V_p^\R \; | \; q_{\imoct}(x) = +1 \}$.
Thus, $\hat{\mu}$ corresponds to a $\rho$-equivariant map $\hat{\nu}: \tilde{S} \rightarrow \quadric$. To see $\hat{\nu}$ is an almost-complex curve, one checks
that $\hat{\nu} \times \hat{\nu}_z = \sqrt{-1} \hat{\nu}_z$, which is verified on the Higgs bundle side by computing $\hat{\mu} \times_{\V} \nabla_{z} \hat{\mu}  = \sqrt{-1} \nabla_{z} \hat{\mu} $. 

If $\nu$ is an alternating almost-complex curve, then the \emph{binormal line} is also realized by $B_p := T_p \times_{\imoct} N_p$, and is a timelike two-plane. 
The pullback bundle decomposes orthogonally as a direct sum $\nu^*T\quadric = T_{\nu} \oplus N_{\nu} \oplus B_{\nu}$. 
The terminology ``line'' is justified as each of the real 2-planes $T,N,B$ is a complex line in $\nu^*T\quadric$. 
Of course, we may replace $\hat{\nu}$ with $\nu$
in the case of taking tangents, normals, and binormals. 
One can see that $\hat{\nu}$ defined above is alternating through the fundamental correspondence 
applied to the tangent, normal, and binormals as follows: 

\begin{itemize} 
	\item (Tangent Line)  $T_{\nu}: \tilde{S} \rightarrow \Gr_{(0,2)}(\imoct)$ by $p \mapsto T_{\nu}(p)$ corresponds to $T_{\mu} \in \Gamma(S, \Gr_{(0,2)}(\V^\R))$,
	where $T_{\mu} = \im( \nabla \mu )$, where $\nabla \mu: TS \rightarrow \V$ is given by $(p, X) \mapsto \nabla_X \mu(p)$. 
	\item (Normal Line) $N_{\nu}: \tilde{S} \rightarrow \Gr_{(2,0)}(\imoct) $ by $p \longmapsto N_{\nu}(p)$ 
	corresponds to $N_{\mu} \in \Gamma(S,  \Gr_{(2,0)}(\V^\R))$ by 
	$N_{\mu}(p) = \im ({\sff_{\V}}|_p)$. 
	\item (Binormal Line) $B_{\nu}: \tilde{S} \rightarrow \Gr_{(0,2)}(\imoct) $ by $ B_{p}:= T_{\nu}(p) \times_{\imoct} N_{\nu}(p)$ corresponds to $T_{\mu} \times_{\V} N_{\mu}$.
\end{itemize} 
We may conflate $T, N, B$ with the sub-bundles of the pullback bundle $\nu^*T\quadric$, so that each of $T, N,B$ is itself an equivariant map (as above) and also a complex line bundle over $\tilde{\Sigma}$. 
Taking the images of the maps $T_{\mu}, N_{\mu}, B_{\mu}$ defines sub-bundles of $\V^\R$, denoted by $\mathcal{T}, \mathcal{N}, \mathcal{B} $. These can be found in a local coordinate
$z = x+ iy$ on $\Sigma$ as follows: 
\begin{align}\label{TNB_Bundle}
	\begin{cases} 
		 \mathcal{T} &= \spann_{\R} \langle \nabla_{\der{x} } \hat{\mu}, \nabla_{\der{y}} \hat{\mu} \rangle\\
		 \mathcal{N} &= \spann_{\R} \langle \sff_{\V}( \der{x}, \der{x}), \sff_{\V}( \der{x}, \der{y} ) \rangle \\
		 \mathcal{B} &:= \mathcal{T} \times_{\V} \mathcal{N}
	\end{cases}.
\end{align}

In terms of the frame \eqref{HUnitaryMultiplicationFrame}, $\mathcal{T} = \spann_{\R}  \langle w_4, w_5 \rangle , \; \mathcal{N} = \spann_{\R} \langle w_2, w_3 \rangle , \; \mathcal{B} = \spann_{\R} \langle w_6, w_7 \rangle$; these identities follow from Higgs bundle calculations (cf. \cite[Section 3]{Eva22}) and prove the desired signatures of $T_p, N_p, B_p$, verifying that $\hat{\nu}$ is alternating.
The name \emph{alternating} comes from
\cite{Nie24}, who considered a more general kind of such harmonic maps and related them to cyclic $\mathsf{SO}(p,p+1)$, and $\Gtwosplit$-Higgs bundles.

The bundle analogue $\sff_{\V}$ of $\sff$ is the 1-form $\sff_\V \in \Omega^1(S, \Hom(\mathcal{T}, \mathcal{N}) )$. However, as $X \mapsto \nabla_X \hat{\mu}$ is a
vector bundle isomorphism between $TS$ and $\mathcal{T}$,
regard $\sff_{\V}$ as a map $\sff_\V: TS \times TS \rightarrow \mathcal{N}$ by 
$$\sff_{\V}( X, Y) := \proj_{(\mathcal{O} \oplus \mathcal{T})^\bot} \nabla_{Y} \nabla_X \hat{\mu} .$$
The \emph{third fundamental form} of $\nu$ is the map $\tff \in \Omega^1(S, \Hom(N_{\nu}, B_{\nu} ))$ given by $\tff(X)(Y) = \overline{D}_X Y \mod ( T_{\nu} \oplus N_{\nu})$, 
where $\overline{D}$ is the tangential connection on $\quadric$ induced by the trivial connection $D$ on $\imoct$. The Higgs bundle analogue of $\tff$ is the bilinear map $\tff_\V: TS \times \mathcal{N} \rightarrow \mathcal{B}$ 
via $\tff_\V(X)(Y) = \nabla_X Y \mod (\mathcal{O} \oplus \mathcal{T} \oplus \mathcal{N})$. The reader can find more details on $\sff, \tff$, and the differential geometry of 
alternating almost-complex curves in $\quadric$ in \cite[\S 3]{CT23}. 

The tuple $( \hat{\nu}, T_{\nu}, N_{\nu}, B_{\nu} )$, called the \emph{Frenet Frame} of $\hat{\nu}$, consists of pairwise orthogonal elements and yields a block-decomposition 
$\imoct= \R \{\hat{\nu}\} \oplus T_{\nu} \oplus N_{\nu} \oplus B_{\nu} $. The Frenet frame lifts both $\hat{\nu}$ and the Hitchin harmonic map 
$f_{\rho}: \tilde{\Sigma} \rightarrow \Gtwosplit/K$ to the $\Gtwosplit$-symmetric space. To discuss the lift, we recall a result from \cite[Lemma 3.8]{Eva22} on a geometric model for the space $\Gtwosplit/T$, where $T < K$ is the maximal torus in the maximal compact subgroup $K < \Gtwosplit$. 
\begin{lemma}\label{GTGeometricModel} $\Gtwosplit/T$ is $\Gtwosplit$-equivariantly diffeomorphic to the space $\mathscr{Y}$ of pairwise orthogonal tuples $(x,T,N,B)$ such that $x \in \quadric, T \in \Gr_{(0,2)}(\imoct),$ $N \in \Gr_{(2,0)}(\imoct)$, $B \in \Gr_{(0,2)}(\imoct)$ and moreover $T,N,B$ are closed under cross-product with $x$.
Thus, there is a natural $\Gtwosplit$-equivariant projection $\pi: \Gtwosplit/T \rightarrow \quadric$ by $(x, T, N, B) \longmapsto x$.
\end{lemma}

The map $\mathscr{F}_{\rho}: \tilde{\Sigma} \rightarrow \Gtwosplit/T$ by $p \mapsto ( \hat{\nu}(p), T_{p}, N_{p}, B_{p} )$ is a harmonic mutual lift of $\hat{\nu} = \hat{\nu}_{\rho}$ and $f_{\rho}$ (cf. \cite[\S 2.3.1]{Bar10},  \cite{Bar15}, and \cite[Section 3.4]{Eva22}).\footnote{See 
(\cite{CT23} Theorem 6.7) for the general case on the relation between
$f_{\rho}$ and $\hat{\nu}$ and a complexified Frenet frame $\mathscr{F}^\C$. We discuss the map $\mathscr{F}^\C$ in Section \ref{CyclicSpace}.}

Later, it will also be of interest to complexify the Frenet frame as in \cite{CT23}. Let us split it each of $T,N,B$ into holomorphic and anti-holomorphic line bundles. That is,
write $T^\C = T' \oplus T'' $, where $T' , T''$ are the $\pm \sqrt{-1}$ eigenspaces, respectively, of the $\C$-linear endomorphism $X \mapsto \hat{\nu} \times X $,
restricted to $T^\C$. Similarly, we obtain splittings 
$N^\C = N' \oplus N''$ and $B^\C = B' \oplus B''$. We write $\mathcal{T}^\C =\mathcal{T}' \oplus\mathcal{T}''$, $\mathcal{N}^\C =\mathcal{N}' \oplus\mathcal{N}''$,  $\mathcal{B}^\C =\mathcal{B}' \oplus\mathcal{B}''$ 
for the analogous splittings in $\V$. In the local frame \eqref{BaragliaBasis}, the lines $(\mathcal{B}', \mathcal{N}'', \mathcal{T}'', \mathcal{O}, \mathcal{T}', \mathcal{N}', \mathcal{B}'')$ correspond in order to the lines $( \C\{ u_i \} )_{i=3}^{-3}$. One then finds that $T'' = \overline{T'} $, $N'' = \overline{N'}$, and $B'' = \overline{B'}$,
under the complex structure on $\imoct^\C$ induced by the real subspace $\imoct$. Later, we
give a geometric interpretation of the space $\Gtwo^\C/T$ to relate the model in Lemma \ref{GTGeometricModel} to a model for the complex Frenet frame.
See Section \ref{CyclicSpace}, particularly Lemma \ref{GeometricCyclicSpace}. 

The final notion we need is the moduli space $\mathcal{H}(S)$ of equivariant alternating almost-complex curves in $\quadric$. 
First, consider the space of pairs of equivariant alternating almost-complex curves and their holonomies: 
\begin{align}\label{LiftedModuliHS}
\mathcal{A}(S) = \{ \, (\hat{\nu}, \rho) \; | \; \rho \in \Hom(\pi_1S, \Gtwosplit), \; \hat{\nu}:\tilde{S} \rightarrow \quadric \; \text{is a} \;\rho-\text{equivariant alternating a.c. curve} \}.
\end{align} 
Then $\mathcal{H}(S) = \mathcal{A}(S) /( \Diff_0(S) \times \Gtwosplit)$, where the action of $\Diff_0(S) \times \Gtwosplit$ is given by $(f,g) \cdot (\hat{\nu}, \rho) = (L_g \circ \hat{\nu} \circ \tilde{f} , C_g \circ \rho)$,
where $L_g$ denotes the $\Gtwosplit$-action on $\quadric$ and $C_g: \Gtwosplit \rightarrow \Gtwosplit$ is $g$-conjugation. 
The space $\mathcal{H}(S) $ inherits the quotient topology from the topology on $\mathcal{A}(S)$ determined by the $C^{\infty}$-convergence the almost-complex curves on compacta. 
\begin{remark}
Let $\hat{\nu}: \tilde{S} \rightarrow \quadric$ be a $\rho$-equivariant alternating almost-complex curve. If $\hat{\nu}$ is $\theta$-equivariant for some $\theta \in \Hom(\pi_1S, \Gtwosplit)$, then $\theta = \rho$. That is,
the pair $(\hat{\nu}, \rho) \in \mathcal{A}(S)$ is determined by $\hat{\nu}$. Moreover, if $(\hat{\nu}_k, \rho_k) \rightarrow (\hat{\nu}, \rho)$ in $\mathcal{A}(S)$ in the topology defined above, then
$\rho_k \rightarrow \rho$ pointwise. This explains why we ignore the representations in the topology.
\end{remark} 

There is a natural holonomy map on $\mathcal{H}(S)$. Indeed, we may define $\hol: \mathcal{H}(S) \rightarrow \Hom(\pi_1S, \Gtwosplit)/\Gtwosplit$ by $\hol( \,[(\hat{\nu}, \rho)] \,) = [\rho]$. 
A priori, if $\rho \in \Hom(\pi_1S,\Gtwosplit)$ is the holonomy of an equivariant alternating almost-complex curve $\hat{\nu}$, then 
$\rho$ might not be reductive. But, in fact, \cite[Theorem 6.7]{CT23} and \cite[Theorem 4.13]{Nie24} show that there is a $\rho$-equivariant harmonic map $f_{\rho}: \tilde{\Sigma} \rightarrow \Gtwosplit/K$, related to
$\hat{\nu}$, which shows $\rho$ is reductive by Corlette's part of the non-abelian Hodge correspondence \cite{Cor88}. Thus, the holonomy map actually takes values in $\chi(\pi_1S, \Gtwosplit)$. 

We recall some useful results on the structure of the moduli space $\mathcal{H}(S)$ from \cite{CT23}. There is a continuous map $b: \mathcal{H}(S) \rightarrow \Z$ given by $[(\nu, \rho)]  \mapsto \deg(\mathcal{B}'_{\nu})$, where $\mathcal{B}'_{\nu}$ is the holomorphic binormal line bundle of the Frenet frame of $\nu$  (recall Section \ref{AC_Curve}.) Then, the integer invariant $0 \leq b \leq 6g-6$ stratifies the space $\mathcal{H}(S)$ \cite[Theorem 5.4]{CT23}. A crucial result we will use going forward is the following: 

\begin{theorem}[\cite{CT23}Theorem 5.11]\label{HitchinAC}
$\hol: \mathcal{H}(S)_{6g-6} := b^{-1}( 6g-6 )\rightarrow \Hit(S, \Gtwosplit)$ is a homeomorphism.\footnote{While \cite{CT23} prove directly that $\hol$ is a continuous bijection, 
the continuity of the inverse is not explicitly addressed in the topology we have used here; we discuss the continuity of the inverse later in the proof of Lemma \ref{s_continuity}. Instead, \cite{CT23} equips $\mathcal{H}(S)$ with a topology by pullback from a moduli space of $\Gtwo^\C$-Higgs bundles.}
\end{theorem}

\subsection{The Developing Map}\label{Subsection:DevMap}
The work of this subsection implicitly defines the map $s_G$ from Figure \ref{GeomStrDiagram}. The map $s_G$ will be explicitly addressed in Section \ref{ModuliSpace}. 
More specifically, we define a geometric structure for the pair $G = \Gtwosplit$ and $ X = (\Eintwothree, \mathscr{D})$ (recall $\mathscr{D}$ from Section \ref{G2Prelims}) 
on the direct sum of fiber bundles $\mathscr{C} = UTS \oplus UTS \oplus \uline{\R_+}$. For the sake of notational brevity, we suppress $\mathscr{D}$ and write just $\Eintwothree$ going forward.\\ 

We recall the relevant construction of a `direct sum' of fiber bundles over the same base manifold. We remark that \cite[page 21]{Coh89} calls this construction the ``Whitney sum.''

\begin{definition}
Let $\pi_i: F_i \rightarrow M$ be smooth $X_i$-fiber bundles over a smooth manifold $M$ with $i \in \{1,2\}$. The \textbf{direct sum} bundle $\pi: F_1 \oplus F_2 \rightarrow M$
is defined as follows. Let $\Delta: M \hookrightarrow M \times M$ be the diagonal embedding and then define $F_1 \oplus F_2: = \Delta^*(F_1 \times F_2)$, where $\pi_1 \times \pi_2: F_1 \times F_2 \rightarrow M \times M$ 
is the projection realizing the Cartesian product $F_1 \times F_2$ as an $(X_1 \times X_2)$-bundle over $M \times M$. 
\end{definition}   

Denote $\uline{\R_+} := S \times \R_+$ as the trivial $\R_+$-bundle over $S$. 
Consider the direct sum bundle $\mathscr{Q}_{\rho} := Q_-\mathcal{T} \oplus  Q_-\mathcal{T}  \oplus \uline{\R_+}$ over $S$, with fibers
$\mathscr{Q}_{p} = Q_-\mathcal{T}_p \times Q_-\mathcal{T}_p \times \R_+ $, 
where we denote $Q_{\pm} E_p := \{ x \in E_p \; | \; q(x) = \pm 1\}$ for $E$ a sub-bundle of $\V^\R$. Here, $\mathscr{Q}:= \mathscr{Q}_{\rho}$ depends on $\rho$, unlike
$\mathscr{C}$. All $(\radialtorus)$-bundles in this paper will be regarded as having structure group $\mathsf{SO}(2) \times \mathsf{SO}(2) \times \Diff^+(\R_+)$. 
Recall that $\nabla \hat{\mu}: TS \rightarrow \mathcal{T}$ by $X \mapsto \nabla_X \hat{\mu}$ is a vector bundle isomorphism. Hence, $\nabla \hat{\mu}$ induces an isomorphism of circle bundles $UTS \cong Q_-\mathcal{T}$, which then induces an isomorphism $\mathscr{C} \cong \mathscr{Q}$ of $(\mathbb{S}^1 \times \mathbb{S}^1 \times \R_+)$-bundles over $S$. 

To define the desired $(\Gtwosplit, \Eintwothree)$-structure on $\mathscr{C}$, we construct a developing map $\dev: \tilde{\mathscr{C}} \rightarrow (\Eintwothree, \mathscr{D} )$
that is equivariant with respect to a holonomy map $\hol: \pi_1 \mathscr{C} \rightarrow \Gtwosplit$ such that $\hol = \rho \circ \pi_*$, where $\pi: \mathscr{C} \rightarrow S$. 
By our fundamental correspondence, constructing $\dev$ is equivalent to defining a section $s \in \Gamma(\mathscr{C}, \mathscr{X} )$, where $\mathscr{X}$ is a flat $\Eintwothree$-bundle over $\mathscr{C}$ with holonomy $ \rho \circ \pi_*$. We construct this section now. First, define the $\Eintwothree$ bundle $Ein \rightarrow \Sigma$ with fibers $Ein_p = \{ \; [x] \subset \V^\R_p \; | \; q(x) = 0 \}$, 
In other words, $Ein = \mathbb{P} Q_0(\V^\R) \cong \tilde{S} \times_{\rho} \Ein^{2,3}$.
We then define $\mathscr{X} := \pi^*(Ein)$. 

Here, to simplify identifications, we use the (extrinsic)
space $\mathscr{Q}$, rather than the intrinsic space $\mathscr{C}$. 
For $[\rho] \in \Hit(S, \Gtwosplit)$, we define a section $\sigma:= \sigma_{\rho} \in \Gamma(\mathscr{Q}, \mathscr{X})$, which yields the desired developing map. 
The following proof uses some techniques motivated by the proof of \cite[Proposition 4.21]{CTT19}. 

We first summarize the idea of the proof. To show $f:= f_{\sigma}$ is an immersion, we use the block-decomposition $\V = \mathcal{O} \oplus \mathcal{T}^{\C} \oplus \mathcal{N}^{\C} \oplus \mathcal{B}^\C$ closely related to the (complex) Frenet frame of $\nu$. 
 The fiber derivatives are straightforward to calculate and are clearly linearly independent, but the calculations are more complicated for $f_{z}, f_{\zbar}$ . 
 The linear independence of $(f_z, f_{\zbar})$ from the fiber
 derivatives boils down to an application of the maximum principle to the Hitchin system \eqref{HitchinsEquations_rs}. 
 Note below that $q(\sff_{\V}(u,v))> 0$ for any $u,v\in \mathcal{T}$ by equation \eqref{TNB_Bundle} and the description of $\mathcal{N}$ thereafter. 

\begin{lemma}\label{Lemma:DevelopingMap} 
Denote $\tilde{\mathscr{Q}}$ as the universal cover of $\mathscr{Q}$. Then the section $\sigma \in \Gamma(\mathscr{Q}, \mathscr{X})$ by  
\begin{align}\label{DevSection}
	 \sigma( x, r, u, v) = \left[ \hat{\mu}(x)+ (r^2+1)^{1/2} \; u + r \; \frac{\sff_{\V}(u,v)}{q(\sff_\V(u,v) )^{1/2} }  \right ]. 
\end{align} 
defines a $(\rho \circ \pi_*)$-equivariant local diffeomorphism $f_{\sigma}: \tilde{\mathscr{Q}} \rightarrow \Eintwothree$. In other words, 
$f_{\sigma}$ defines a $(\Gtwosplit, \Eintwothree)$-structure on $\mathscr{Q}$ with holonomy $\hol = \rho \circ \pi_*$ . 
\end{lemma} 

\begin{proof}
The equivariance of $f_{\sigma}$ is addressed by the fundamental correspondence. The remainder of the proof shows that $f_{\sigma}$ is a local diffeomorphism. Recall that we have a natural identification $T_{u} \Eintwothree \cong \Hom(u, u^{\bot}/u )$. Alternatively, unnaturally choosing any $\hat{u} \in u$, we can identify $ u^\bot/ u \cong T_{u} \Eintwothree $ 
by $v + [u] \mapsto (\hat{u} \mapsto v + [u])$. Fix a local coordinate $z = x + i y$ on $\tilde{\Sigma}$ and coordinates $(\alpha, \theta, r)$ on $\tilde{\mathscr{Q}}$. Thus, $f:= f_{\sigma}$ is an immersion $\iff$ $(f_z, f_{\zbar},\, f_{r}, \, f_{\theta}, \, f_{\alpha}  )|_p$ is linearly independent in $T_{f(p)}\Ein^{2,3}$ $\iff$ $\spann_{\R} \langle f, \, f_z, f_{\zbar},\, f_{r}, \, f_{\theta}, \, f_{\alpha} \rangle$ is 6-dimensional $\iff$ 
\begin{align}\label{Dim6} 
	\dim_\C \spann_{\C} \langle  \sigma, \nabla_z \sigma, \nabla_{\zbar} \sigma, \nabla_{\alpha} \sigma, \nabla_{\theta} \sigma, \nabla_{r} \sigma \rangle \; = 6. \; \;
\end{align} 
 
The rest of the proof works in an open set $U \subset \Sigma$ on which we have trivialized $\V$. By the definition, $\sigma$ comes with a distinguished lift $\hat{\sigma}$, given by 
$$\hat{\sigma}( x, r, u, v) = \hat{\mu}(x)+ (r^2+1)^{1/2} \; u + r \; \frac{\sff_{\V}(u,v)}{q(\sff_\V(u,v) )^{1/2} } .$$
We split $\hat{\sigma}$ into two pieces and write $\sigma = [ \hat{\sigma}_1 + \hat{\sigma}_2 ] $, where $\hat{\sigma}_1= \hat{\mu} +  (r^2+1)^{1/2} \; u$ and $ \hat{\sigma}_2 =  r \; \frac{\sff_{\V}(u,v)}{q(\sff_{\V}(u,v))^{1/2}  }$
for $\hat{\sigma}_i \in \Gamma(\mathscr{Q}, \V^\R)$. Here, we write $u = u(\theta), v = v(\alpha)$ in local coordinates $(\alpha, \theta)$ for $Q_- \mathcal{T} \times Q_- \mathcal{T}$. Explicitly, $u(\theta) = \cos(\theta) w_4+ \sin(\theta) w_5$ and $v(\alpha) = \cos(\alpha) w_4+ \sin(\alpha) w_5$ in terms of the local frame \eqref{HUnitaryMultiplicationFrame} for $\V^\R$. To prove \eqref{Dim6} we show the stronger fact that 
$$\mathcal{S} := (\hat{\sigma}_1, \;\nabla_{r} \hat{\sigma}, \; \nabla_{\theta}\hat{\sigma},  \;\nabla_{\alpha }\hat{\sigma} , \; \hat{\sigma}_2,\; \nabla_{z} \hat{\sigma},  \; \nabla_{\zbar} \hat{\sigma}\; ) .$$ 
is a spanning set for $\V$. Denote the elements of $\mathcal{S}$ in order as $ (b_i)_{i=1}^7$. We show that $\mathcal{S}$ spans by inductively
proving it spans the following sub-bundles $(\mathcal{B}^\C, \mathcal{O} \oplus \mathcal{B}^\C,  \mathcal{O} \oplus \mathcal{B}^\C \oplus \mathcal{T}^\C, \V)$. In
particular, we will prove the following: 
\begin{enumerate}
	\item $ \proj_{\mathcal{B}^\C} \, b_6, \proj_{\mathcal{B}^\C}\, b_7$ span $\mathcal{B}^\C$. 
	\item $ \proj_{\mathcal{O}} \, b_1 $ spans $\mathcal{O} $ and $\proj_{\mathcal{B}^\C}b_1 = 0$.
	\item $ \proj_{\mathcal{T}}\, b_2, \proj_{\mathcal{T}} \, b_3$ span $\mathcal{T}$ and $\proj_{\mathcal{O} \oplus \mathcal{B}^\C} \, b_2 = 0 = \proj_{\mathcal{O} \oplus \mathcal{B}^\C} \, b_3$. 
	\item $\, b_4, b_5$ span $\mathcal{N}$. 
\end{enumerate}
We prove (2),(3), and (4) quickly and then spend more time on (1). \\

To prove (2), observe that $\proj_{\mathcal{O}} (\hat{\sigma}_1) = \hat{\mu}$ spans $\mathcal{O}$ and $\proj_{\mathcal{B}^\C}\hat{\sigma}_1 = 0$. For (3), we note 
$\proj_{\mathcal{T}}b_3 = \proj_{ \mathcal{T}} ( \der{\theta} (\hat{\sigma}_1+ \hat{\sigma}_2) \,) = \proj_{ \mathcal{T}}( \der{\theta} \hat{\sigma}_1) $, where last equality is because $\der{\theta} \hat{\sigma}_2(p) \in \mathcal{N}$ for $p \in \mathscr{Q}|_U$. Thus, 
$$\spann_{\R} \langle \, \proj_{\mathcal{T}} b_2,  \proj_{\mathcal{T}} b_3 \rangle \;=\; \spann_{\R}  \langle \cos(\theta) w_4+ \sin(\theta) w_5,  \, -\sin(\theta) w_4+ \cos(\theta) w_5 \rangle \; = \mathcal{T}. $$
Since $b_2, b_3$ are local sections of $ \mathcal{T} \oplus \mathcal{N} $, their projections on $\mathcal{O} \oplus \mathcal{B}^\C$ are trivial.

We now prove (4). We need to show $ \hat{\sigma}_2$ and $ \proj_{\mathcal{N}} (\nabla_{\alpha} \hat{\sigma}) =  \proj_{\mathcal{N}} ( \nabla_{\alpha} \hat{\sigma}_2)=  \nabla_{\alpha} \hat{\sigma}_2 = \nabla_{\alpha} \hat{\sigma}$ span $\mathcal{N}$. Note that
\begin{align}
	\nabla_{\alpha} \hat{\sigma}_2 &= \der{\alpha} \hat{\sigma}_2 = \der{\alpha}\left( \frac{1}{q(\sff_{\V}(u,v) )^{1/2} } \right) \, \sff_{\V}(u,v) +  \frac{1}{q(\sff_{\V}(u, v))^{1/2} } \der{\alpha} \sff_{\V}(u, v(\alpha))  \\
		&= \der{\alpha} \sff_{\V}(u, v(\alpha) \, ) \mod (\hat{\sigma}_2)  =   \sff_{\V}\left ( u, \der{\alpha} v(\alpha) \, \right) =  \sff_{\V}(u, -\sin(\alpha)w_4 + \cos(\alpha) w_5).
\end{align} 
If we take any $u_p, v_p, w_p \in \mathcal{T}_p$ with $u_p \neq 0$ and $v_p \neq w_p$, then the elements $\sff_\V(u_p, v_p), \sff_\V(u_p, w_p)$ span $\mathcal{N}_p$.
It follows that $\hat{\sigma}_2, \nabla_{\alpha} \hat{\sigma}_2$ are linearly independent and hence span $\mathcal{N}$. 

Finally, we prove point (1). We apologize to the reader as for the remainder of the proof $r$ is not a radial parameter, but a component of the hermitian metric $h$ from \eqref{HitchinMetric}. 
A computation shows that 
$0 = \mathsf{proj}_{\mathcal{B}^\C} (\nabla_{X} \hat{\sigma}_1) $ for $X \in T\Sigma^\C$. Indeed, $\mathsf{proj}_{\mathcal{B}^\C}\nabla_X \alpha = 0$ for $ \alpha \in \Gamma(S, \mathcal{T} \oplus \mathcal{O} )$.  
Hence, 
$$\mathsf{proj}_{\mathcal{B}^\C} (\nabla_{X} \hat{\sigma} ) = \mathsf{proj}_{\mathcal{B}^\C} (\nabla_{X} \hat{\sigma}_2) = \tff^\C_{\V}(X, \hat{\sigma}_2) .$$
Since $\hat{\sigma}_2$ is a local section of $\mathcal{N}$, it suffices to prove the more general fact that if $w \in \Gamma(U, \mathcal{N})$ is a nonzero local section, then
then $\tff^\C_\V (\der{z}, w), \tff^\C_\V ( \der{\zbar}, w)$ span $\mathcal{B}^\C$. Here, we use a block decomposition of $\tff_\V^\C$, splitting  
$\mathcal{N}^\C = \mathcal{N}' \oplus  \mathcal{N}'' $ and $\mathcal{B}^\C = \mathcal{B}' \oplus  \mathcal{B}''$, the analogues of the splitting of the complex Frenet frame. 
Then $ \tff_\V^\C: (\mathcal{N}' \oplus  \mathcal{N}'' ) \rightarrow (\mathcal{B}' \oplus  \mathcal{B}'' )$ takes the form $ \tff_\V^\C = \begin{pmatrix} q & \sqrt{3} s \\ \sqrt{3} & \frac{\overline{q}}{r^2s} \end{pmatrix} $, where each entry is a global tensor. For example $\sqrt{3} \in \Omega^{1,0}(\Sigma, \Hom(\mathcal{N}', \mathcal{B''})) = H^0(\Sigma, \mathcal{O}) $ since $\mathcal{N}' \cong \K^{-2}, \mathcal{B}'' \cong \K^{-3}$. 
Next, one finds $\det \tff^\C_{\V, w} = ( \frac{|q|^2}{r^2s} - 3s )$. Hence, $\tff_{\V,w}^\C|_p$ is a linear isomorphism if $\frac{|q|^2}{r^2s^2}(p) \neq 3$. By Proposition \ref{Prop:GlobalEstimates}, we have the global inequality $|q|^2_{\sigma}e^{-2\psi_1-2\psi_2} = \frac{|q|^2}{r^2s^2}< 3$, completing the proof. 
\end{proof}

We now introduce a definition before making some remarks on the Lemma \ref{Lemma:DevelopingMap}. Let $f: S \rightarrow \quadric$ be an alternating
almost-complex curve and $T_fS, N_fS, B_fS$ the tangent, normal, and binormal lines. We call a the second fundamental form
$\sff: T_fS \times T_fS \rightarrow N_fS$, to be \emph{non-degenerate} 
when for any non-vanishing local section $w = w_U \in \Gamma(U, T_fS)$, the map $T_fS |_U \rightarrow N_fS|_{U}$ by $X \mapsto \sff(w, X)$
a linear isomorphism. Analogously, we say $\tff: T_fS \times N_fS \rightarrow B_fS$ is non-degenerate when $N_fS |_U \rightarrow B_fS|_{U}$ by $X \mapsto \tff(w, X)$ is a linear isomorphism for all such $w$. 
We may also speak of non-degeneracy of $\sff_p, \tff_p$ at a single point $p$ in a similar fashion. 

\begin{remark}\label{NonDegeneracyRemark}
Note that that if $\hat{\nu}$ is an almost-complex curve in $\quadric$, then $\sff_p$ is non-degenerate $\iff$ $\sff_p \neq 0 $ 
by the $J$-invariance property $J_{\quadric} \sff(X,Y) = \sff(J_{\Sigma}(X), Y)$. 
Hence, the definition of $\sigma$ requires the non-degeneracy of the second fundamental form $\sff$ of $\hat{\nu}$. The proof that $f$ was an immersion in Lemma \ref{Lemma:DevelopingMap} 
also showed that $\tff$ is non-degenerate for $\hat{\nu}$ that is $\rho$-equivariant and $[\rho] \in \Hit(S, \Gtwosplit)$.
\end{remark} 

\begin{remark}\label{NonDegeneracyRemark2}
 In fact, $(\hat{\nu}, \rho) \in \mathcal{A}(S)$ has $[\rho] \in \Hit(S, \Gtwosplit)$ $\iff$ $\sff$ is non-degenerate \cite{CT23}. This holds because under the splittings
 $T = T' \oplus T'', N = N' \oplus N''$, the complex bilinear extension $\sff^\C $ decomposes as $\sff^\C = \begin{pmatrix} \sff' & 0 \\ 0& \overline{\sff'} \end{pmatrix}$, so $\sff^\C$ is determined by $\sff'$. On the other hand, $\sff' \in \Omega^{1,0}(\Sigma, \Hom(T', N')) = H^0(\K^2 \otimes N')$ is non-vanishing exactly when the holomorphic normal line $N'$ satisfies $N' \cong \K^{-2}$, which occurs if and only if $[\rho] \in \Hit(S,\Gtwosplit)$ by (\cite[Theorem 5.11]{CT23} and diagram (24) on page 23). 
\end{remark} 

We make some more identifications. There is a natural $\pi_1S$-equivariant vector bundle isomorphism $\pi^*\mathcal{T} \cong T_{\nu} = : T$. This identification is just a flip from viewing tangent vectors extrinsically in the Higgs bundle to intrinsically in the pullback bundle of the geometric tangent space of $\nu$. Then the identification $\pi^*\mathcal{T} \cong T $ induces a $\pi_1S$-equivariant circle bundle isomorphism $Q_-T \cong \overline{ Q_-\mathcal{T} }$, where $\overline{ Q_-\mathcal{T} }:= \pi^*Q_-\mathcal{T}$. Define $\overline{\mathscr{Q}}:= \pi^*\mathscr{Q}$. We are then led to a natural identifications $\overline{\mathscr{Q} } \cong Q_-T \oplus Q_-T \oplus \uline{\R_+} \cong UT\tilde{S} \oplus UT\tilde{S} \oplus \uline{\R_+}$ of ($\radialtorus$)-bundles.

\begin{remark}\label{DevDescends} 
The developing map $f_{\sigma}$ from Lemma \ref{Lemma:DevelopingMap} can be described without reference to the Higgs bundle. First, note that $f_{\sigma}: \tilde{\mathscr{Q}} \rightarrow \Eintwothree$ descends to a map $\overline{f}_{\sigma}: \overline{\mathscr{Q}} \rightarrow \Eintwothree$, since the holonomy of $\mathscr{X}$ is trivial along the fibers. Then 
$\overline{f}_{\sigma}:  Q_-T \oplus Q_-T \oplus \uline{\R_+}\rightarrow \Eintwothree$  takes the following form: 
\begin{align}\label{DevMapConstruction}
	\overline{f}_{\sigma}(p, u,v, r) = \left[ \, \hat{\nu}(p)+ \, (r^2+1)^{1/2} \,u \, + r \, \frac{ \sff(u,v)}{q(\sff(u,v))^{1/2} } \, \right].
\end{align}
Using the $\pi_1S$-equivariant vector bundle isomorphism $d\hat{\nu}: T\tilde{S} \rightarrow T$, the developing map takes the following form
if we explicitly emphasize the dependence on the almost-complex curve $\hat{\nu}$. We denote
$\overline{\dev}: UT\tilde{S} \oplus UT\tilde{S} \oplus \uline{\R_+} \rightarrow \Eintwothree$ as this map, given by 
\begin{align}\label{DevMapFromNu}
	\overline{\dev}(p, u,v,r) = \left[ \, \hat{\nu}(p)+ \, (r^2+1)^{1/2} \,\frac{ d\hat{\nu} (u) }{ |q(d\hat{\nu} (u)|^{1/2}} \, + r \, \frac{ \sff( d\hat{\nu}(u) , d\hat{\nu}(v) )}{q(\sff(d\hat{\nu}(u) ,d\hat{\nu}(v) ))^{1/2} } \, \right].
\end{align}
\end{remark} 

\begin{remark}
The $(\Gtwosplit, (\Eintwothree, \mathscr{D}))$-structure on $\mathscr{C}$ defines a 
$(2,3,5)$-distribution $D$ on $\mathscr{C} = UTS \oplus UTS \oplus \uline{\R_+}$, which 
is transverse to the fibration.
It seems difficult to give a clean algebraic description of this distribution. \end{remark} 

\subsection{Other Perspectives on $\dev$} \label{DevReinterpret} 
In this section, we discuss two different perspectives on the developing map $\overline{\dev}$ of the previous section. First,
we factor $\overline{\dev}$ as a composition of maps, which provides a 
simplified perspective on the image of $\overline{\dev}$. This perspective is used in Section \ref{FuchsianCase} to examine the developing map in the Fuchsian case. 
We then tweak the construction of $\overline{\dev}$, using that the radial parameter $r$ can be imagined in the fiber or in the base. The second construction  
leads to a different geometric interpretation of $\overline{\dev}$ -- as an interpolation between two simpler maps.\\

First, we need some more notation. Associated to $\hat{\nu}$, we have
the \emph{second extended osculating subspace map} $U: \tilde{S} \rightarrow \Gr_{(3,2)} \imoct$ by $ p \mapsto \mathscr{L}_{p} \oplus T_{p} \oplus N_{p}$, 
where $(\mathscr{L}, T, N, B)$ is the Frenet frame of $\hat{\nu}$ with $\hat{\nu}$ replaced by $\mathscr{L} = \R\{\hat{\nu}\}$. Since $\imoct = \mathscr{L}_p \oplus T_{p} \oplus N_{p} \oplus B_{p}$ is an orthogonal block
decomposition, we have orthogonal projection maps $\Pi_{T}, \Pi_{N}, \Pi_{B}$ (depending on $p$), mapping from $\imoct$ to each (respective) subspace in the decomposition. 

The Grassmannian $\Gr_{(3,2)}\imoct$ carries a tautological $\R^{3,2}$-sub-bundle $\mathscr{U} \rightarrow \Gr_{(3,2)}\imoct$
of the trivial bundle $\Gr_{(3,2)}\imoct \times \imoct$ with fiber $\mathscr{U}|_P = P$. The map $U$ defines an $\R^{3,2}$-bundle $U^*\mathscr{U}$ over $\tilde{S}$ by pullback. 
With the map $U$, we define the $\mathsf{Ein}^{2,1}$ bundle $\mathcal{Q} := \mathbb{P}Q_0(U^*\mathscr{U}) $ with fiber 
$\mathcal{Q}|_p := \mathbb{P} Q_0 U(p)$  
as well as an $(\mathbb{S}^1 \times \mathbb{S}^1 \times \R_+)$-subbundle $\mathcal{Q}_{\neq 0}$ of $\mathcal{Q}$ as follows:
\begin{align} \label{Qbundle}
	 (\mathcal{Q}_{\neq 0})|_{p} = \{ \, L \in \mathcal{Q}|_p \; | \; \Pi_{\mathscr{L}_p}(L) \neq 0, \; \Pi_{T_p}(L) \neq 0, \; \Pi_{N_p}(L) \neq 0 \}. 
\end{align} 
A simple argument later in Section \ref{RadialTorusFamilies} shows the fibers satisfy $ (\mathcal{Q}_{\neq 0})_{p} \cong \radialtorus$. 

Now, the map $\overline{\dev}$ factors through the innocuous fiber-forgetting map $\psi: \mathcal{Q}_{\neq 0} \rightarrow \Eintwothree$ by $ \psi(p, L) = L $. 
Denote $\overline{\mathscr{C}} = UT\tilde{S} \oplus UT \tilde{S} \oplus \uline{\R}_+$. 
The map $\overline{\dev}$ factors as $\overline{\dev} =\psi \circ \phi$, where $\phi: \overline{\mathscr{C}} \rightarrow \mathcal{Q}_{\neq 0}$ is a fiber-preserving diffeomorphism, which lifts the identity map $\id_{\tilde{S}}$. 
The map $\phi$ is just $\phi(p, u,v,r ) = (p, \overline{\dev}(p, u,v,r))$. 
Then $\phi$ is globally injective as it `remembers' the basepoint on $\tilde{S}$, 
so all questions of injectivity lie with $\psi$. 

It is essential that $\overline{\dev}$ factors through $\mathcal{Q}_{\neq 0}$, rather than $\mathcal{Q}$. 
Indeed, the map $\psi: \mathcal{Q}_{\neq 0} \rightarrow \Eintwothree$ extends continuously to a map $\overline{\psi}: \mathcal{Q} \rightarrow \Eintwothree$ by the same equation, but
$\overline{\psi}$ is \emph{not} an immersion. We explain why this is the case with an argument similar to the proof of Lemma \ref{Lemma:DevelopingMap}. At any point $X_p = [ \hat{\nu}(p) + t_p ]$ with $t_p \in Q_-(T_{p})$, we claim the 
differential $d\overline{\psi}|_{X_p}$ does not have full rank. Recall that 
$\Eintwothree \cong (\mathbb{S}^2 \times \mathbb{S}^3)/ \sim$, where $(x,y ) \sim (-x,-y)$. Denote $\mathbb{D} \subset \R^2$ as the unit disk and introduce local coordinates $\mathbb{D} \times U $ for $U \subset (0, 2\pi)$ on the fibers of $\mathcal{Q}$ as follows:
for $ (a,b) \in \mathbb{D}, \, \theta \in U$, 
define $X_p(a, b, \theta ) \in \mathcal{Q} $ as $ X_p = [ \sqrt{1-a^2-b^2} \, \hat{\nu}(p)+ a n_1 + b n_2 + \, t_p(\theta)  ] $, where $n_1, n_2$ is a local orthonormal frame for $N$.
Using local coordinates $x,y \in \tilde{S}$, one finds $ \spann_{\R} \langle d\overline{\psi}(\der{x}), \, d\overline{\psi}(\der{y}),\, d\overline{\psi}(\der{a}), \, d\overline{\psi}(\der{b}), \, d\overline{\psi}(\der{\theta}), \, \overline{\psi} \rangle \subset U_p$. In particular, these 6 elements are trapped in the 5-plane $U_p$ and hence are not linearly independent. As in the proof of Lemma \ref{Lemma:DevelopingMap}, this means $\overline{\psi}$ is not an immersion at $X_p$ if $\Pi_{N_p}(X_p) = 0$.
On the other hand, similar calculations show $\overline{\psi}$ \emph{is} an immersion at points $X_p \in \mathcal{Q}$ where $\Pi_{\mathscr{L}_p}(X_p) = 0$. As a consequence, $\mathcal{Q}_{\neq 0}$ does not contain all the immersed points of $\overline{\psi}$. We revisit this point shortly.\\

We now discuss a continuous extension $\overline{\overline{\dev}}$ of the developing map $\overline{\dev}$ to a compactly fibered space over $\tilde{S}$,
then we relate $\overline{\overline{\dev}}$ to $\overline{\psi}$. The map $\overline{\overline{\dev}}$ will give us a perspective of $\overline{\dev}$ as an interpolation. 
Going forward, we keep all the notation the same from Lemma \ref{Lemma:DevelopingMap}.
The first (small) step is to reconceive of the bundle $\mathscr{Q} \rightarrow S$ instead as a bundle $\mathcal{C} \rightarrow B$ over the base $B := S \times \R_+$. Then the fiber $\mathcal{C}_{(p,r)}$ at a point $(p,r) \in B$ is given by
$$\mathcal{C}_{(p, r)} = Q_{\sqrt{1+ r^2}}(\mathcal{T}_p ) \times Q_{r}(\mathcal{T}_p ) = \{ \,(u,v) \in \mathcal{T}_p \times \mathcal{T}_p\; | \; q(u) = -1-r^2, \; q(v) = -r^2 \}.$$ 
Recall the bundle $\mathscr{X} \rightarrow \mathscr{Q}$. Since $\mathscr{Q} \cong_{\mathbf{Diff}} \mathcal{C}$ via the map
$(p,(u,v,r)) \mapsto ( (p,r), (\sqrt{1+r^2}\,u, r \, v))$, we can regard $\mathscr{X}$ as a bundle over $\mathcal{C}$ as well. Now, the section $\sigma \in \Gamma(\mathscr{Q}, \mathscr{X})$ from \eqref{DevSection} is instead a map (by notational abuse) $\sigma :  \mathcal{C} \rightarrow \mathscr{X}$. 
Then $\sigma \in \Gamma( \mathcal{C}, \mathscr{X} )$ now takes the form: 
\begin{align}
	 \sigma(\, (p,r), (u,v) ) = \left[ \hat{\mu}(p) + u + |q(v)|^{1/2} \, \frac{ \sff_{\V}(u,v)}{ q(\sff_{\V}(u,v)) }\right] .
\end{align} 
While $r$ does not appear in the formula explicitly, it is still playing the role of determining $q(u)$ and $q(v)$.
Allowing degeneration of the radial parameter $r$ to $0$ and $\infty$, there is an evident continuous extension of this developing section to a space $\overline{ \mathcal{C}}$ fibered over $\overline{B} = S \times [0, \infty]$ with fibers at the new endpoints as follows: $ \overline{ \mathcal{C}}_{(p,0) } = Q_-\mathcal{T}_p$ and $\overline{ \mathcal{C}}_{ (p, \infty) } = Q_-\mathcal{T}_p \times Q_-\mathcal{T}_p$. 
Note, in particular, that the torus fibers of $\overline{\mathcal{C}}_{(p,r)}$ degenerate to a circle fiber at $r= 0$, so $\overline{\mathcal{C}}$
is not a fiber bundle over $B$, but is compact. Going forward, we sloppily write $\mathscr{X}$ to denote $\pi^*Ein$ for $\pi: F_r \rightarrow S \times \{ r\}$ the appropriate bundle over $S \times \{r\}$.
The map $\sigma$ can be viewed now as a family of sections $ \sigma_{r} \in \Gamma( \overline{ \mathcal{C}}_{S \times \{r\}}, \; \mathscr{X}) $ for $ r \in [0, \infty]$, which interpolate 
between the sections $ \sigma_0 \in \Gamma( \; Q_-\mathcal{T} , \mathscr{X} \; )$ and $ \sigma_{\infty} \in \Gamma( Q_-\mathcal{T}  \oplus Q_-\mathcal{T} , \mathscr{X} )$ given by 
\begin{align}
	\sigma_0(p, u) &= [ \hat{\mu}(p) + u]  \label{Dev0} \\
	\sigma_{\infty}(p, u, v) &= \left[ u + \frac{ \sff_\V(u,v) } { q(\sff_\V(u,v)) } \right].\label{DevInfty} 
\end{align} 
We emphasize that equation \eqref{Dev0} is exactly the \cite{CTT19} developing map construction, now seen as a developing section in a new geometrical context. 
Let us make some observations about the equivariant maps $\tilde{f}_0: \widetilde{Q_-\mathcal{T}  } \rightarrow \Eintwothree $ and $\tilde{f}_{\infty}:\widetilde{Q_-\mathcal{T}  \oplus Q_-\mathcal{T}} \rightarrow \Eintwothree$ associated to $\sigma_0$ and $\sigma_{\infty}$, respectively. Both $\tilde{f}_0$ and $ \tilde{f}_{\infty}$
are immersions by the proof of Lemma \ref{Lemma:DevelopingMap}. We now use the notation from before Remark \ref{DevDescends}. Also, $\tilde{f}_0$ and $ \tilde{f}_{\infty}$ are equivariant with respect to the holonomy $\rho \circ \pi_{*}$ and hence the maps descend to
 $UT \tilde{S} \cong Q_-T$ and $UT\tilde{S} \oplus UT\tilde{S} \cong Q_-T \oplus Q_-T$, respectively.

Similar to the conformally flat Lorentz structures of \cite{CTT19}, the map $\tilde{f}_0: UT \tilde{S}  \rightarrow \Eintwothree$ develops each fiber isomorphically into a timelike circle of $\Ein^{2,3}$, i.e., a copy of $\Ein^{0,1} \hookrightarrow \Eintwothree$. Indeed, if we fix $p \in \tilde{S}$, then $\tilde{f}_0: UT\tilde{S}_{p} \rightarrow \Eintwothree$ maps bijectively onto $\mathbb{P} Q_0(\mathscr{L}_p \oplus T_{p} ) \cong \Ein(\R^{1,2}) = \Ein^{0,1}$. 
The map $\tilde{f}_{\infty}$ has special fibering as well. To see this, we first observe that if $P \cong \R^{2,0}, R \cong \R^{0,2}$, then any $x \in \mathbb{P}Q_0(P \oplus R)$ 
can be written almost uniquely as $x = [u + v]$ with $u \in Q_+(P) , v \in Q_-(R)$, up to $x = [-u - v]$. Hence, 
$\tilde{f}_{\infty}: UT\tilde{S} \oplus UT\tilde{S} \rightarrow \Eintwothree $ develops the fiber at $p \in \tilde{S}$ in 2-1 fashion onto $\mathbb{P}Q_0( T_{p} \oplus N_{p} ) \cong \mathsf{Ein}^{1,1}$.\\

Recall that $\mathcal{Q}$, unlike $\mathcal{Q}_{\neq 0}$, contains lines $l \in \mathbb{P}Q_0(U)$ with vanishing projections onto either $\mathscr{L}$ or $N$. In this way,
the bundle $\mathcal{Q}$ decomposes into three disjoint sub-bundles $\mathcal{Q} = L_{0} \sqcup L \sqcup L_{\infty}$, 
where $ L_0|_p = \mathbb{P}Q_0( \mathscr{L}_p \oplus T_p ), \; L_p = \mathcal{Q}_{\neq 0}|_p, \; L_{\infty}|_p = \mathbb{P}Q_0( T_p\oplus N_p)$. 
Moreover, we saw that $\overline{\psi} $ is not an immersion at $p$ if and only if $p \in L_{0}$. Thus, while we could extend our geometric structure smoothly from $ L$ to $L \sqcup L_{\infty} $,
we prefer to keep the parametrization of $\overline{\dev}$ as in \eqref{DevMapFromNu} so that the developing map remains injective on fibers. 

\section{From Geometric Structures to Representations}  \label{GXToRepn}

In this section, we prove the main result in Theorem \ref{MainTheorem}. 
First, in Section \ref{RadialTorusFamilies} we discuss some
$\imoct$ geometry necessary for the technical definition of the geometric structures in $\mathscr{M}$. The essential new notion is that of an \emph{ ($\radialtorus$)-family} in $\Eintwothree$, which is a particular geometric locus that realized by the developed image of a fiber of the bundle $\overline{\mathscr{C}} \rightarrow \tilde{S}$ by $\overline{\dev} =\overline{\dev}(\hat{\nu}) $ constructed in the previous section. In fact, the ($\radialtorus$)-family $\mathscr{S}_p = \overline{\dev}({\overline{\mathscr{C}}_p})$ associated to $p \in \tilde{S}$ nearly carries the data of the Frenet frame. Indeed, $\mathscr{S}_p$ is in 1-1 correspondence with the tuple $(\R\{\hat{\nu}(p)\}, T_p, N_p, B_p)$. 
In Section \ref{GeometricStructuresDefinition}, we define the desired geometric structures, imposing multiple additional conditions on the developing map. 
We then define the moduli space $\mathscr{M}$ of such geometric structures, up to isomorphism, in Section \ref{ModuliSpace}. 
One such condition on the developing map, the \emph{cyclic-fibering} condition, 
which mimics the previously described fibering of $\overline{\dev}$, allows us to define 
a map $H: \mathscr{M} \rightarrow \mathcal{H}(S)$, so that our geometric structures have associated almost-complex curves. 
The developing map $f_{\sigma}$ from Lemma \ref{Lemma:DevelopingMap} satisfies 
all of the conditions we impose and descends to define a map $s: \Hit(S, \Gtwosplit) \rightarrow \mathscr{M}$. 
We prove $s$ is continuous in Section \ref{Continuity}. Now, the holonomy of a geometric structure $G \in \mathscr{M}$ factors through $\pi_1 S$ to $\overline{\hol}: \pi_1 S \rightarrow \Gtwosplit$ by definition. 
In Section \ref{GXtoRepn}, we show the `descended holonomy map' $\alpha: \mathscr{M} \rightarrow \Hit(S, \Gtwosplit) $ by $[ \, (\dev, \hol) \, ] \mapsto [ \overline{\hol} ] $ is inverse to $s$. We use crucially that both $\alpha$ and $s$ factor through $\mathcal{H}(S)_{6g-6}$. \\

 \subsection{$\Gr_{(3,2)}\imoct$ and ($\radialtorus$)-Families in $\Eintwothree$} \label{RadialTorusFamilies} 
 
The following lemma describes the structure of $(3,2)$-planes in $\imoct$. In particular, they come with a distinguished line.
One should imagine $U \in \Gr_{(3,2)}(\imoct)$ as the second extended osculating subspace $U = \ell \oplus T \oplus N$ of $\hat{\nu}$ and $\ell \in \mathbb{P} U$ as the (projective) almost-complex curve $\nu$. 

\begin{lemma}\label{32-Lemma}
$\Gtwosplit$ acts transitively on $\Gr_{(3,2)}(\imoct)$. 
Moreover, given $U \in\Gr_{(3,2)}(\imoct)$, there exists a unique line $\ell \in \mathbb{P}U$ such that $\ell \times U \subset U$. In fact, $\ell \in \mathbb{P}Q_+(U) \subset \Stwofour$. 
\end{lemma} 

\begin{proof}
There is a natural $\Gtwosplit$-equivariant diffeomorphism $F: \Gr_{(3,2)}(\imoct) \rightarrow \Gr_{(0,2)}(\imoct)$ by $U \mapsto U^\bot$. Now, one can show $\Gtwosplit$ acts transitively on $\Gr_{(0,2)}(\imoct)$ using a superficial alteration to the standard Stiefel model $V_{(+,+,-)}$ from Proposition \ref{StiefelTripletModel}. Define the Stiefel manifold 
$$V_{(-,-,+)} : = \{ (u, v, w) \in (\imoct)^3 \; | \; q(u) = q(v) = -1 = -q(w), \; u \cdot v = u \cdot w = v \cdot w =\; (u \times v) \cdot w = 0 \}.$$
The $\Gtwosplit$-equivariant map $f: V_{(-,-,+)}  \rightarrow V_{(+,+,-)}$ by $(u,v,w) \mapsto (u \times v, w, u)$ is a bijection, thus defining an isomorphism of $\Gtwosplit$-spaces. Here, one needs only the fact that
$ w \times (u \times v) = - u \times (w \times v)$, so that $((u \times v)\times w) \cdot u= 0$ and the map is well-defined, with inverse $f^{-1}: V_{(+,+,-)} \rightarrow V_{(-,-,+)}$ given by $(x,y,z) \mapsto (z, z \times x, y)$. 
Alternatively, by equivariance, $f$ is uniquely constrained by $f(l, li, j) = (i, j, l)$. 
By Proposition \ref{StiefelTripletModel}, $\Gtwosplit$ acts simply transitively on $V_{(-,-,+)}$. 
As a corollary, $\Gtwosplit$ acts transitively on $\Gr_{(0,2)}(\imoct)$ and $\Gr_{(3,2)}(\imoct).$

By the transitivity, it suffices to prove the uniqueness statement for $ U_0 = \spann_{\R} \langle i, j, k, l, li \rangle$. 
Clearly, $\ell_0= [i] \in \mathbb{P}(U_0)$ satisfies $\ell_0 \times U_0 \subset U_0$. Using the relations amongst $\ell_0 = \, \langle i \rangle $, $ T_0 = \, \spann_{\R} \langle l, li \rangle$, $N_0 = \, \spann_{\R} \langle j, k \rangle $, $B_0 = \, \spann_{\R} \langle lj, lk \rangle $, we show the uniqueness of $\ell_0$.
For the contrapositive, suppose $x \in \mathbb{P}(U_0)$ satisfies $x \times U_0 \subseteq U_0$ and $x \neq \ell_0$. Write $ x= [\hat{x}]$ for $ \hat{x} = x_{\ell} + x_T + x_N$ with
$x_{\ell} \in \ell_0,\, x_T \in T_0, \,x_N \in N_0$. Since $\ell_0\neq x$, then $x_T \neq 0 $ or $x_N \neq 0 $. Using $L_0 \times T_0 = T_0, L_0 \times N_0 = N_0, T_0 \times N_0 = B_0$, in either case of $x_T \neq 0$ or $x_N \neq 0$, there exists $w \in U_0$, with  
$w \in N_0$ or $w \in T_0$, respectively, such that $\proj_{B_0}( \hat{x} \times w) \neq 0$. Thus, $x \times U_0 \nsubseteq U_0$. This proves the uniqueness of $\ell_0$
and also that $\ell_0 \in \mathbb{P}Q_+(U_0)$. 
\end{proof} 

The following corollary needs no assumption on the holonomy of the almost-complex curve. 

\begin{corollary}
If $\hat{\nu}: \tilde{S} \rightarrow \quadric$ is an injective $\rho$-equivariant alternating almost-complex curve, 
then the second extended osculating subspace map $U:  \tilde{S} \rightarrow \Gr_{(3,2)} \imoct$ is injective.
\end{corollary} 

\begin{proof}
Suppose that $U:= U(x) = U(y)$. 
The small remark here is that the Frenet frame relations stated in Lemma \ref{GTGeometricModel} hold at $p \in \tilde{S}$
even if $\sff_p = 0$. Thus, $\R \{\hat{\nu}(p) \} \times U_p \subset U_p$, which then means 
$\hat{\nu}(x) \times U \subset U $ and $ \hat{\nu}(y) \times U \subset U$. Then Lemma \ref{32-Lemma} says $\hat{\nu}(x) = \hat{\nu}(y)$, so that $x = y$.
\end{proof} 

We will also need the double cover $\widehat{\Ein}^{2,3} \cong_{\mathbf{Diff}} \mathbb{S}^2 \times \mathbb{S}^3$ of $\Eintwothree$. Be advised that the identification
$\widehat{\Ein}^{2,3} \cong \mathbb{S}^2 \times \mathbb{S}^3$ is not $\Gtwosplit$-equivariant. 

\begin{definition}
Define space $\widehat{\Ein}^{2,3}$ of null rays in $\imoct$ by $\widehat{\Ein}^{2,3} := (Q_0(\imoct) - \{ 0 \} )/\R_+$. That is, $u, v \in Q_0(\imoct)$ satisfy $u \sim v$ when there exists $\lambda \in \R_+$ such that $v = \lambda u$. 
\end{definition} 

In the following definition, we use again the model splitting $ \imoct = \ell_0 \oplus T_0 \oplus N_0 \oplus B_0$ and $U_0 = \ell_0 \oplus T_0 \oplus N_0$ from the proof of Lemma \ref{32-Lemma}.

  \begin{definition}\label{S1C*family}
  Consider the model subset $ \widehat{\mathscr{S}}_0 \subset \widehat{\Ein}^{2,3}$ given by 
  \begin{align}\label{ModelFamily}
 	 \widehat{\mathscr{S}}_0 = \{  \; [i + (r^2+ 1)^{1/2} \, u_T + r \, u_N ] \in \widehat{\Ein}^{2,3}\; | \; r \in (0,\infty), \; u_T \in Q_-(T_0), \; u_N \in Q_{+}(N_0) \}.
 \end{align} 
 Let $\pi: \widehat{\Ein}^{2,3} \rightarrow \Eintwothree$ denote the quotient map and we define $\mathscr{S}_0 := \pi( \widehat{\mathscr{S}}_0 )$. 
 We call a subset $\mathscr{S}$ an \textbf{($\mathbb{S}^1 \times \mathbb{S}^1 \times \R_+)$-family} in $\Eintwothree $ (resp. $\widehat{\Ein}^{2,3}$) when $\mathscr{S} = \varphi (\mathscr{S}_0)$ or $\mathscr{S} = \varphi (\widehat{\mathscr{S}}_0)$, respectively, for some $\varphi \in \Gtwosplit$. \end{definition} 
 
 Shortly, we offer an equivalent definition of an ($\mathbb{S}^1 \times \mathbb{S}^1 \times \R_+)$-family. Here, we observe that the linear span of the points in $\mathscr{S}_0$ gives $U_0$
 back: $\spann_{x \in \mathscr{S}_0} \, x = U_0$.
By the $\Gtwosplit$-transitivity on ($\mathbb{S}^1 \times \mathbb{S}^1 \times \R_+$)-families in $\Eintwothree $ and
Lemma \ref{32-Lemma}, an ($\radialtorus)$-family $\mathscr{S}$ comes with the following associated data:
 	\begin{enumerate}[label=(\roman*)]
		\item A subset $U = U(\mathscr{S}) := \spann_{x \in \mathcal{S} } \, x\, \in \Gr_{(3,2)}(\imoct)$, 
		\item A unique line $\ell \in \mathbb{P}Q_+ U$ such that $\ell \times U \subset U$. 
	\end{enumerate}  
The model subset $\mathscr{S}_0$ from \eqref{ModelFamily} is also realized as the following open subset of $\mathbb{P}Q_0(U_0) \cong \Ein^{2,1}$, for $U_0 = (\ell_0 \oplus T_0 \oplus N_0)$:
 \begin{align}\label{FamilyReformulation}
 	\mathscr{S}_0 = \{ \, x \in \mathbb{P}Q_0(U_0)\; | \; x = [x_{\ell} + x_{t} + x_{n} ],  \; x_{\ell},x_{t}, x_{n} \neq 0, \; x_\ell \in \ell_0, \; x_t \in T_0, \; x_N \in N_0 \}.
\end{align}
Of course, if $\ell \in  \mathbb{P}Q_0(U_0)$, then $\Pi_{T_0}(\ell) \neq 0$ automatically. 

By the description of $\mathscr{S}_0$ in \eqref{FamilyReformulation}, if we begin with a triplet of pairwise orthogonal subspaces $(L,T,N)$ with $L \in \Gr_{(1,0)}\imoct,$ $ T \in \Gr_{(0,2)}\imoct$, $N \in \Gr_{(2,0)}\imoct$ such that 
$L \times T =T$, $L \times N = N$, then 
we can construct an ($\mathbb{S}^1 \times \mathbb{S}^1 \times \R_+$)-family $\mathscr{S}$ in $\Eintwothree$ by 
\begin{align}\label{RadialTorusFamily}
	\mathscr{S} := \mathscr{S}(L,T,N) =  \{ x \in \mathbb{P}Q_0(U)\; | \; x = [x_{\ell} + x_{t} + x_{n} ],  \; x_{\ell},x_{t}, x_{n} \neq 0, \; x_\ell \in L, \; x_t \in T, \; x_N \in N \}.
\end{align} 
Indeed, choose $x \in L, y \in T, z \in N$ of appropriate norm $\pm 1$ and by Proposition \ref{StiefelTripletModel} there is a transformation $\varphi \in \Gtwosplit$ such that
$\varphi(i) = x, \varphi(l) = y, \varphi(j) = z$, forcing $\varphi(\mathscr{S}_0) = \mathscr{S}(L,T,N)$. 
Conversely, by equation \eqref{FamilyReformulation}, every $(\radialtorus)$-family obtains this form. 
Thus, \eqref{RadialTorusFamily} is an equivalent definition of a ($\radialtorus)$-family. Crucially, we can also uniquely recover the triplet $(L,T,N)$ from $\mathscr{S}$ alone by decomposing $U(\mathscr{S})$ into subspaces 
so that the projections from $U$ onto these subspaces never vanish. 

\begin{lemma}\label{USplitting}
Let $\mathscr{S}$ be an $(\mathbb{S}^1 \times \mathbb{S}^1 \times \R_+)$-family in $\Eintwothree$ and let $U \in \Gr_{(3,2)} \imoct$ and $\ell \in \mathbb{P}Q_+U$ be the associated subspaces from (i), (ii) after Definition \ref{S1C*family}. Then there exists a unique orthogonal splitting $ U = \ell \oplus T \oplus N$ with $N \in \Gr_{(2,0)}\imoct, \, T \in \Gr_{(0,2)}\imoct$, $\ell \times T = T, \; \ell \times N = N$, such that the projection maps 
$\Pi_{\ell}: \mathscr{S} \rightarrow \ell, \; \Pi_{T}: \mathscr{S} \rightarrow T, \; \Pi_{N}: \mathscr{S} \rightarrow N$ are all non-vanishing. 
In other words, $\mathscr{S} = \mathscr{S}(\ell, T, N)$ uniquely. 
\end{lemma} 

\begin{proof}
By the tautological transitivity of $\Gtwosplit$ on ($\mathbb{S}^1 \times \mathbb{S}^1 \times \R_+$)-families in $\Eintwothree$, it suffices to prove the claim for the model $\mathscr{S}_0$. 
Equation  \eqref{FamilyReformulation} proves existence of such a decomposition, showing $\mathscr{S}_0 = \mathscr{S}(\ell_0,T_0,N_0) $.
Now, to prove uniqueness, suppose for contradiction there were another splitting
$U_0= \ell_0 \oplus T' \oplus N'$ such that $\mathscr{S}_0 = \mathscr{S}(\ell_0,T',N')$. 

Note that the pair $(N',T')$ is determined by a single nonzero element $v \in N'$. 
Indeed, the subspaces are realized as $N' = \spann_\R \langle v, i \times v \rangle$ and $T' = \left[ \, (N')^\bot \subset (T_0 \oplus N_0) \, \right] $. Thus, if $N' = N_0$, then $T' = T_0$. We may then suppose $N' \neq N_0$. Hence, there exists
$v \in Q_+(N')$ of the form $v= c_1 v_T + c_2 v_N$ with $v_T \in Q_-(T), v_N \in Q_+(N)$,\; $c_2^2-c_1^2 = 1$, and $c_1 \neq 0$. 
Note that $v^\bot:= i \times v= c_1 (i \times v_T) + c_2 (i \times v_N)$. One then finds
$T' = \spann_{\R} \langle u, u^\bot \rangle$, where $u = c_2 v_T + c_1 v_N$ and $u^{\bot} =  c_2( i \times v_T ) +c_1 (i \times v_N)$.
Since $| \frac{c_2}{c_1} | > 1$, choose $t \neq 0 $ to solve $ \frac{ (t^2+1)^{1/2} }{ t } = \frac{c_2}{c_1} $. Take any $x_l \neq 0 \in \ell_0$, then define 
$$p := [ \, x_{\ell} + (t^2+1)^{1/2} u + t (-v) ].$$ 
Since $q(p) = 0$, we have $p \in \mathscr{S}(\ell_0,T',N')$. On the other hand, one can rewrite $p =  [  x_{\ell} + x_{t} ]$ for $x_{t} \in T$,
so that $p \notin \mathscr{S}(\ell_0, T_0, N_0)$, a contradiction. 
We conclude that $\mathscr{S}_0 = \mathscr{S}(\ell_0,T_0,N_0) $ uniquely.
\end{proof}
Thus, even though there are many splittings of $U \in \Gr_{(3,2)} \imoct $ into $U = \ell \oplus T \oplus N$ with $ \ell \times T = T, \, \ell \times N = N$,
if we fix the additional data of $\mathscr{S}$ with $U(\mathscr{S} ) = U$ and demand further that the projections from each line $p \in \mathscr{S}$ don't vanish on the component subspaces $\ell, T, N$,
then there is a unique such splitting $U = \ell \oplus T \oplus N$.

As a corollary to Lemma \ref{USplitting}, an ($\radialtorus$)-family in $\widehat{\Ein}^{2,3}$ has a similar associated splitting, but with slightly finer data. 
Below, we may naturally identify $\quadric = Q_+\imoct \cong Q_{+} \imoct /\R_+$ by $v \mapsto [v]$. 

\begin{corollary}\label{Cor:USplitting}
Let $\widehat{\mathscr{S}}$ be an $(\mathbb{S}^1 \times \mathbb{S}^1 \times \R_+)$-family in $\widehat{\Ein}^{2,3}$. Then let $U = \ell \oplus T \oplus N$ be the associated splitting 
of $\pi(\widehat{\mathscr{S}})$. Then there is a unique element $\hat{x} \in Q_+(\ell)$ such that every point $L \in \widehat{\mathscr{S}}$ uniquely obtains the form 
$L = [ \hat{x} + (r^2+1)^{1/2} t + r n ]$ for some $t \in Q_-T, \, n \in Q_+ N, \, r \in \R_+$. Thus, associated to $\widehat{\mathscr{S}}$ is the tuple $(\hat{x}, T, N)$. 
\end{corollary} 

\begin{proof}
The projection maps $\Pi_\ell, \Pi_T, \Pi_N$ each lift to $Q_+(\imoct), Q_-(\imoct), Q_+(\imoct)$, respectively. In particular, we have a lift $\hat{\Pi}_\ell : \mathscr{S} \rightarrow Q_+(\ell)$ by
Lemma \ref{32-Lemma} and Lemma \ref{USplitting}. But 
$Q_+(\ell) = \{\hat{x}, -\hat{x}\}$ for some $\hat{x} \in \quadric$. Since $\hat{\Pi}_{\ell} $ is continuous, it must be constant. 
\end{proof} 

\subsection{Cyclic Surfaces and $\Gtwo^\C/T$ } \label{CyclicSpace}
Before defining the geometric structures of interest, we give one more essential 
definition, that of a \emph{cyclic surface}, which will serve as an intermediary between the developing map $\dev$ and its associated almost-complex curves $\nu$. 
We then give a geometric model for $\Gtwo^\C/T$ and interpret the complex Frenet frame of \cite{CT23} in this model.

To give the definition of cyclic surfaces from \cite{CT23}, we need an aside on Lie theory. 
Going forward, we fix the basis \eqref{BaragliaBasis} for $\imoct^\C$, which determines a representation 
$\g_2^\C \hookrightarrow \gl_7\C$. We then fix the compact real form $\mathfrak{k}$ of $\g_2^\C$ associated to the involution $\rho:= (A \mapsto -\overline{A}^T)$
and the split real form $\g_2' $ associated to the involution $\tau(A) = Q\overline{A}Q$ with 
$$Q = \begin{pmatrix} & & & & & &1\\
					& & & & &1&\\
				   	& & & & 1& & \\
				  	& & & 1& & & \\
				  	& & 1&& & &\\
				   	& 1& & & & & \\
				   	1 & & & & & & \end{pmatrix}.$$
One finds $\tau = \sigma \rho = \rho \sigma$, so that $\rho$ and $\tau$ define compatible compact and split real forms, respectively.

Let $\h < \g_2^\C$ be the maximal torus of diagonal transformations in the basis \eqref{BaragliaBasis}. 
Then $\mathfrak{t} := \Fix(\tau) \cap \Fix(\rho) \cap \h$ is a maximal torus 
in the maximal compact subalgebra $\mathfrak{k} \cap \g_2'$ of $\g_2'$. In terms of the Cartan subalgebra $\mathfrak{h} = \mathfrak{t} \oplus i \mathfrak{t}$ of $\g_2^\C$,
we have the root space decomposition $\g_2^\C = \mathfrak{h} \oplus \bigoplus_{\alpha \in \Delta} \g_{\alpha}$, where $\Delta$ is the set of roots of $\h$
and $\g_{\alpha}$ the $\alpha$-root space. 
Fix a set of primitive roots $\Pi = \{ \alpha, \beta \} \subset \Delta$, with $\beta$ the short root, to discuss $\Pi$-height of roots.
Then any root $\delta \in \Delta $ is of the form $\delta = a \alpha + b \beta $ for $a,b \in \mathbb{Z}$ of the same sign and $\mathsf{height}_{\Pi}(\delta) := a+b$. 

The Lie algebra $\g_2^\C$ admits a $\mathbb{Z}_6$-grading, which will descend to $\g_2^\C/\mathfrak{t}$. 
For $[j] \in \mathbb{Z}/6\mathbb{\mathbb{Z}}$, define $\g_{ j} : = \bigoplus_{\mathsf{height}_{\Pi}(\alpha) \equiv - j \, \text{mod} \, 6} \; \g_{\alpha}$ and $\g_0 = \h$. 
The highest root $\gamma = 2\alpha+ 3\beta$ in $\Delta$ has height 5. 

We then get a decomposition 
\begin{align}\label{CyclicDecomposition}
		\g_2^\C = \bigoplus_{k \in \Z_6 } \g_{k}. 
\end{align} 
Since $ \g_{0} = \mathfrak{h}$, the decomposition \eqref{CyclicDecomposition} descends to the decomposition 
\begin{align}\label{CyclicDecomposition_2}
\g_2^\C/\mathfrak{t} \cong i \mathfrak{t} \oplus \bigoplus_{k \neq 0 \in \Z_6 } \g_{k}. 
\end{align} 
Denote $\mathfrak{t}^\bot :=  i \mathfrak{t}\oplus  \bigoplus_{k \neq 0 \in \Z_6 } \g_{k}$.
Consider the space $Y := \Gtwo^\C/T$, where $T \cong \mathsf{U}(1) \times \mathsf{U}(1)$ is the maximal torus in $\Gtwo^\C$ corresponding to $\mathfrak{t}$. 
Then $T_{eT} Y \cong \frakt^\bot$. Let $\pi_{\frakt^\bot}: \g_2^\C \rightarrow \frakt^\bot$ denote orthogonal projection and $\omega:T\Gtwo^\C \rightarrow \g_2^\C$ the Maurer-Cartan form. Define the 1-form $\omega_{\frakt^\bot} := \pi_{\frakt^\bot} \circ \omega \in \Omega^1(\Gtwo^\C, \frakt^\bot)$. 
Now, the natural projection $\pi: \Gtwo^\C \rightarrow Y$ realizes $\Gtwo^\C$ as a principal $T$-bundle over $Y$. The form $\omega_{\frakt^\bot}$ allows us to identify $TY  \cong [\frakt^\bot ] $, where $[\frakt^\bot ] := \Gtwo^\C \times_{\Ad} \mathfrak{t}^\bot$ is the associated
vector bundle over $Y$ to the $T$-bundle $\Gtwo^\C$ via the adjoint representation $ \Ad: T \mapsto \mathsf{GL}(\frakt^\bot)$. 
That is, the map $\Gtwo^\C \times \frakt^\bot \rightarrow TY$ by $(g, X) \mapsto d\pi \circ dL_g(X)$ descends to the isomorphism $[\frakt^\bot]\cong TY$. 
 Since $T$ preserves the root spaces $\g_{\alpha}$ as well as the subspaces $\g_{k}$, there are well-defined
sub-bundles $[\g_{\alpha}]$ and $[\g_{k}]$ of $[\frakt^\bot]$ given by $ [\g_{\alpha}] := \Gtwo^\C \times_{\Ad} \g_{\alpha} $ and $[\g_{k}] := \Gtwo^\C \times_{\Ad} \g_{k} $. 
Correspondingly, for roots $\sigma \in \Delta$, we have a 1-forms $\omega_{\sigma}: TY \rightarrow [\g_{\sigma}] $ as well as 1-forms $\omega_{k}: TY \rightarrow [\g_{k}].$
Since the adjoint action of $T$ on $\frakt^\bot$ and $\sigma, \rho, \tau$ commute, each involution descends to the bundle $[\frakt^\bot]$. 

 Now, consider the distribution $\mathcal{D}$ on $TY$ defined by $\Fix(\tau) \cap [\g_{1} \oplus \g_{-1}] $. 
Define the map $J$ on $\g_{ 1} \oplus \g_{-1}$ by $J (x, y) = (\sqrt{-1} \, x, -\sqrt{-1} y)$. Then $J$ clearly commutes with $\Ad(T)$, $\sigma$, and $\rho$, so that $J$ commutes with $\tau$. Hence, $J$ descends to $\mathcal{D}$ and induces an almost-complex structure $\mathcal{J}: \mathcal{D} \rightarrow \mathcal{D}$. 
Now, following \cite[Definition 6.4]{CT23}, we define cyclic surfaces. Here, we let $\mathsf{S}$ be any smooth oriented surface. 
Note in particular that $\g_{1} = \g_{-\alpha} \oplus \g_{-\beta} \oplus \g_{\gamma}$. 
\begin{definition}
Let $f: \mathsf{S} \rightarrow Y$ be an orientation-preserving smooth map
tangent to the distribution $\mathcal{D}$ and moreover have tangent space $df_p(T_p\mathsf{S})$ that is $\mathcal{J}$-invariant. 
If we also have $f^*\omega_{\sigma}$ not identically zero for $\sigma \in \{-\alpha, -\beta\}$ and $f^*\omega_{-\beta}$ non-vanishing, then we call $f$
a (CT) cyclic surface. 
\end{definition} 

The link between (equivariant) cyclic surfaces and (equivariant) alternating almost-complex curves is the following powerful theorem; the original 
theorem covers also the non-equivariant case, but we state only the relevant portion. 

\begin{theorem}[\cite{CT23} Theorem A] \label{FrenetFrameTheorem} 
Let $[\rho] \in \chi(\pi_1S, \Gtwosplit)$. The Frenet frame furnishes a bijection between isomorphism classes of [$\rho$]-equivariant alternating almost-complex curves $\nu: \tilde{S} \rightarrow \quadric$ and isomorphism classes of [$\rho$]-equivariant (CT) cyclic surfaces $f: \tilde{S} \rightarrow \Gtwo^\C/T$.
\end{theorem} 

\begin{remark}
We remark that the non-vanishing of $f^*\omega_{\alpha}$ for a cyclic surface $f$ corresponds to the non-vanishing of the holomorphic
second fundamental form $\sff' \in \Omega^1(S, \Hom(T',N'))$ of the associated almost-complex curve, which in turn  
corresponds to the non-vanishing of $\sff$, as noted before in Remark \ref{NonDegeneracyRemark2}.\footnote{We apologize to the reader for the notational clash with \cite{CT23}. In their paper, 
the holomorphic data $(\alpha, \beta, \delta)$ that determines a cyclic surface has $\alpha$ and $\beta$ with roles reversed from our notation here.} 
On the other hand, $f^*\omega_{\beta}$ non-vanishing corresponds to $\hat{\nu}$ being an immersion. See \cite[Section 6.1]{CT23}. 
\end{remark} 

So far in this subsection, we have only recalled definitions of \cite{CT23}. 
We now try to geometrically reinterpret the complex Frenet frame. To this end, we give a (somewhat) clunky geometric model for the space $\Gtwo^\C/T$.  
First, a small definition: suppose $\imoct^\C =  \bigoplus_{i=3}^{-3} L_i$ for lines $L_i$ such that $L_i \times L_j \subseteq L_{i+j}$. Just as \cite[Proposition 8.23]{Eva22},
one can show $q^\C|_{L_0}$ non-degenerate is forced from this condition. 
Fix any $x_0 \in L_0$ with $q(x_0) = +1$ and we call such a decomposition \emph{normalized} when 
$L_{-1} \subset \mathbb{E}_{\sqrt{-1}}(\mathcal{C}_{x_0})$, where $\mathbb{E}$ denotes eigenspace. We use this technical condition in a necessary way in the following proof. 
 
 \begin{lemma}\label{GeometricCyclicSpace}
 Consider the space $\mathscr{T}$ of pairs $( \, (L_i)_{i=3}^{-3}, \, \sigma)$ such that  $\sigma: \imoct^\C \rightarrow \imoct^\C$ is the conjugation of a real subspace, the sesquilinear form $h(z,w) = q^\C(z,\sigma(w))$ is non-degenerate, 
 and $\imoct^\C = \bigoplus_{i=3}^{-3} L_i$ is an $h$-orthogonal line decomposition satisfying:
	\begin{enumerate}[label= (\roman*)] 
		\item $L_i \times L_j \subseteq L_{i+j}$
		\item $\sigma(x \times y ) = \sigma(x) \times \sigma(y)$
		\item $\sigma(L_i) = L_{-i}$
		\item $q^\C|_{\Fix(\sigma)}: \Fix(\sigma) \rightarrow \R$ is of signature (3,4) 
		\item $(L_i)$ is normalized. 
	\end{enumerate}
Then $\mathscr{T}$ is $\Gtwo^\C$-equivariantly diffeomorphic to $\Gtwo^\C/T$. 
 \end{lemma} 
 
 \begin{proof}
Consider the basis $(u_i)_{i=3}^{-3}$ from \eqref{BaragliaBasis}. Then associated to it, we have the basepoint $P_0 \in \mathscr{T}$ given by
$P_0 = (\, (L_i^0)_{i=3}^{-3}, \sigma_0)$, with $L_i^0 = \C \{  u_i \}$ and $\sigma_0: \imoct^\C \rightarrow \imoct^\C$ the complex conjugation 
induced from the real subspace $\imoct \hookrightarrow \imoct^\C$ (cf. \cite[Proposition 2.6]{Eva22}.) The stabilizer of all the lines $(L_i^0)_{i=3}^{-3}$ is 
the subgroup $T^\C < \Gtwo^\C$ of diagonal transformations in the basis $(u_i)$. Clearly, $\Stab(\sigma_0 ) = \Gtwosplit$. 
Thus, $\Stab(P_0 )= T^\C \cap \Gtwosplit$. Now, $\psi \in \Gtwo^\C$ preserves $\sigma_0$ if and only if $\psi$ preserves $h_0 = q^\C(\cdot, \sigma_0\cdot)$. On the other hand, $\psi \in T^\C $ preserves $h_0 = \sum_{i=3}^{-3} \mathsf{sgn}(i) z_i \overline{z}_i$ if and only if $\psi$ preserves the hermitian form $\sum_{i=3}^{-3} z_i\overline{z}_i $. 
Defining $\boldsymbol{K} :=\Gtwo^\C \cap \mathsf{U}(7)$, we then have $\Stab(P_0) = T^\C \cap \boldsymbol{K} $. We conclude that $\Stab(P_0) = T^\C \cap \boldsymbol{K} \cap \Gtwosplit$.
One can show $T^\C \cap \boldsymbol{K} \cap \Gtwosplit \cong \mathsf{U}(1) \times \mathsf{U}(1)$, 
so that $ \Stab(P_0) $ is a maximal torus $T$ in the maximal compact subgroup $K = \boldsymbol{K} \cap \Gtwosplit$ of $\Gtwosplit$. To finish the proof, we show $\Gtwo^\C$ acts transitively on $\mathscr{T}$
after a number of small observations. 

Take any $x_0 \neq 0 \in L_0$ such that $\sigma(x_0 ) = x_0$. 
Then $q^\C(x_0) \neq 0$ by the normalization (iv), but in fact, this follows directly from (i). If $u \in \imoct$ has $\mathcal{C}_{u}: \imoct^\C \rightarrow \imoct^\C$ diagonalizable, then $q(u) \neq 0$. One proves this by
contrapositive with the double cross-product identity \eqref{DCP}, which implies that
this implies that $\mathcal{C}_{u}$ is nilpotent if $q(u) = 0$. On the other hand, $\dim \Ann(u) = 3$ for $u$ isotropic by Proposition \ref{Annihilators}, so that $C_{u}$ is not diagonalizable. 
Since $\Fix(\sigma)$ is closed under $\times$ and $q|_{\Fix(\sigma)}$ is of signature (3,4), this means $(\Fix(\sigma), \times) \cong (\imoct, \times_{\imoct})$. 
Using $(\mathcal{C}_{x_0}\circ \mathcal{C}_{x_0})|_{x_0^\bot} = -q(x_0) \id_{x_0^\bot}$, we conclude $\mathcal{C}_{x_0}(L_i) = L_i$ for $i \neq 0$.

We claim that $L_2 \times L_1= L_3$ must occur. Suppose otherwise, so that $L_2 \times L_1 = \{0\}$. Hence, $L_1 \oplus L_2 \oplus L_3 \subseteq \Ann(L_2)$ by (i).
Using 3-form $\Omega$, we see $\Omega(L_1, L_2, L_{-3}) = (L_1 \times_{\imoct^\C} L_2) \cdot L_3 = 0 $. Observe that
$0 = \Omega(L_2, L_{-3}, L_1) $, so $(L_2 \times L_{-3}) \bot L_1$, which combined with $L_2 \times L_{-3} \subset L_{-1}$, is only possible if
$L_2 \times L_{-3} =0$ by (iii) and the non-degeneracy of $h$. Hence, $L_{-3} \subset \Ann(L_2)$, implying $\dim_{\C} \Ann(L_2) \geq 4$, a contradiction to Proposition \ref{Annihilators}. 
By (ii), (iii), we conclude $L_{-2} \times L_{-1} = L_{-3}$ as well. 

Let us show $q^\C(x_0) > 0$. Fix any $i \in \{1,2,3\}$ and take any $y_i \in L_i$. 
Next, write $y_i = y_i^+ + \sqrt{-1} y_i^-$, with $y_i^{\pm} \in \Fix(\sigma)$. Then one 
finds $q^\C(y_i^+) = q^\C( y_i^-)$ since $q^\C(y_i) = h(y_i, \sigma(y_i) )=0$ and $q|_{\Fix(\sigma)}$ is real-valued. 
In particular, 
$\Fix(\sigma)|_{L_i \oplus L_{-i} }$ has signature (2,0) or (0,2). Now, if $q^\C(x_0) < 0$, then we would have 
$q^\C|_{\Fix(\sigma) \cap L_0^\bot }$ of signature (3,3), but this is impossible by the previous observation.

Next, define (suggestively) $T:= \Fix(\sigma)_{L_1 \oplus L_{-1}}$, $N:= \Fix(\sigma)_{L_2 \oplus L_{-2}}$, $B:= \Fix(\sigma)_{L_3 \oplus L_{-3}}$.
Then one finds $T \times N = B$. Recall that $q^\C|_{T \oplus N \oplus B} $ is of signature (2,4). Since $T, N$ are orthogonal, by the multiplicativity of $q$, 
the only possibility is that the signatures alternate: $\sig \, T = (0,2), \;\sig\,  N = (2,0), \; \sig \,{B} = (0,2)$. 
Similarly, set $V^0 =L_0, \; V^1 = \mathbb{E}_{\sqrt{-1}}(\mathcal{C}_{x_0}), \, V^2 = \mathbb{E}_{-\sqrt{-1}}(\mathcal{C}_{x_0})$
and one finds the $\mathbb{Z}_3$-cross-product grading $V^i \times_{\imoct^\C} V^j  \subseteq V^{i+j}$, indices mod 3, very similar to \cite[Lemma 8.2]{Eva22}. 
Also, $V^1,V^2$ are isotropic and $q^\C$ defines a non-degenerate pairing of $V_1 \times V_2 \rightarrow \C$. Condition (iii) and the orthogonality of $h$ then force $V^1 = L_{-1} \oplus L_{-2} \oplus L_3$ and 
$V^2 = L_{1} \oplus L_{2} \oplus L_{-3}$. 

Finally, we prove the transitivity. Take any point $P := ((L_i), \sigma) \in \mathscr{T}$. Define $T,N,B $ as above but now with respect to $P$. 
Choose $x_1 \in T$ such that $q^\C(x_1) = -1$ and $x_2 \in N$ such that $q^\C(x_2) = 1$. 
By condition (v), $z_{\pm 1} = \frac{1}{\sqrt{2}}( x_1 \pm \sqrt{-1} x_0 \times x_1)$ are generators for $L_{\pm 1}$. 
Similarly, $z_{\pm 2}: = \frac{1}{\sqrt{2}}(x_2 \pm \sqrt{-1} x_0 \times x_2)$ are generators for $L_{\pm 2}$. Finally, we set 
$x_3 := x_1 \times x_2. $ Then $z_{\pm 3} =\frac{1}{\sqrt{2}}( x_3  \pm \sqrt{-1} x_0 \times x_3)$ and $x_{\pm 3}$ span $L_{\pm 3}$.
By Proposition \ref{StiefelTripletModel}, there is a unique $\varphi \in \Gtwo^\C$ such that 
$\varphi(i) = x_0, \; \varphi(j) = x_2, \; \varphi(l) = x_1$. Then $\varphi$ maps $L_{i}^0$ to $L_i$. Moreover, 
$\varphi$ maps the $h_0$-orthonormal basis $(u_i)_{i=3}^{-3}$ from \eqref{BaragliaBasis} to the $h$-orthonormal basis $(z_i)_{i=3}^{-3}$, so that $\varphi$
maps $h_0$ to $h$ and hence $\sigma_0$ to $\sigma$.  \end{proof} 

\begin{corollary}
There is a natural projection $\pi_{\quadric}: \Gtwo^\C/T \rightarrow \quadric$ by $( \, (L_i)_{i=3}^{-3}, \, \sigma) \mapsto x_0$, 
where $x_0 \in Q_+(\Fix_{\sigma}|_{L_0})$ is the unique element satisfying $L_{-1} \subset \mathbb{E}_{\sqrt{-1}}(\mathcal{C}_{x_0})$. 
\end{corollary}

\begin{remark}\label{InverseFrenetFrame}
We now see the inverse operation of taking the complex Frenet frame. Let $\hat{\nu}: \tilde{S} \rightarrow \quadric$ be an alternating almost-complex curve
and $\mathscr{F}^\C: \tilde{S} \rightarrow \Gtwo^\C/T$ its complex Frenet frame. Then $\hat{\nu}= \pi_{\quadric} \circ \mathscr{F}^\C$.
\end{remark}

\begin{remark}\label{SymmetricSpaceInclusion}
Recalling Lemma \ref{GTGeometricModel}, we find the standard inclusion $\iota: \Gtwosplit/T \hookrightarrow \Gtwo^\C/T$ is given geometrically in the model spaces by  
$(x, T, N, B) \mapsto (\, ( B', N'', T'', \C\{x\}, T', N', B''), \, \sigma_0)$, where $\sigma_0$ is the standard conjugation fixing $\imoct \hookrightarrow \imoct \otimes_\R \C$
and $T',N',B'$ and $T'',N'',B''$ are the $\pm \sqrt{-1}$-eigenspaces of $\mathcal{C}_x$, respectively, in $T^\C, N^\C, B^\C$. 
In fact, one finds $$\iota(\Gtwosplit/T ) = \{ \,( (L_i)_{i=3}^{-3}, \sigma) \in \mathscr{T} \; | \; \sigma = \sigma_0 \}.$$\\
\end{remark} 
 
\subsection{Definition of the Geometric Structures}  \label{GeometricStructuresDefinition} 
In this section, we define the `decorated' geometric structures of interest on $\mathscr{C}$, 
using the notions defined 
in the previous two subsections. In particular, our $(\Gtwosplit, \Eintwothree)$-geometric structures come with extra \emph{cyclic-fibered}, \emph{compatible}, and \emph{radial} (CCR) conditions. These three conditions will allow us to define a moduli space $\mathscr{M}$ of geometric structures on which we can invert the construction of Section \ref{Subsection:DevMap}, and recover the almost-complex curve $\hat{\nu}$ from the geometric structure $\dev(\hat{\nu})$ in Lemma \ref{Lemma:DevelopingMap}.\\

One may view the following definition as a similar version to the \emph{fibered} and \emph{maximal} conditions
\cite[Definitions 4.19 and 4.20]{CTT19}, respectively. 

 \begin{definition}\label{CyclicCondition}
 Let $\dev: \widetilde{\mathscr{C}} \rightarrow \Eintwothree$ be an equivariant immersion and $\hol: \pi_1 \mathscr{C} \rightarrow \Gtwosplit$ the associated holonomy. We say $\dev$ is \textbf{cyclic-fibered} when:
	 \begin{enumerate}
 		\item (fibered) The holonomy factors as $\hol = \rho \, \circ \, \pi_{*} $, where $\rho: \pi_1S \rightarrow \Gtwosplit$ and $\pi: \mathscr{C} \rightarrow S$ is the bundle projection and  		
		the map $\dev$ lifts $\hol$-equivariantly to $\widehat{\dev}: \widetilde{\mathscr{C}} \rightarrow \widehat{\Ein}^{2,3}$ such that  
		 for $p \in \tilde{S}$, the map $\overline{\dev}$ surjectively maps the fiber $\widetilde{\mathscr{C}}_p$ onto an $(\mathbb{S}^1 \times \mathbb{S}^1 \times \R_+)$-family 
		 $\widehat{\mathscr{S}}_p$ in $ \widehat{\Ein}^{2,3}$ with associated splitting $\mathscr{F}_p := (\hat{x}_p , T_p , N_p , B_p)$ (Recall Corollary \ref{Cor:USplitting}).
		 
		\item (cyclic) The map $\mathscr{F}^\C: \tilde{S} \rightarrow \Gtwo^\C/T$ by $\mathscr{F}^\C = \iota \circ \mathscr{F}$ is a (CT) cyclic surface, where $\mathscr{F}: \tilde{S} \rightarrow \Gtwosplit/T$ is given by $p \mapsto (\hat{x}_p , T_p , N_p , B_p)$
		and $\iota: \Gtwosplit/T \hookrightarrow \Gtwo^\C/T$.
	 \end{enumerate}
\end{definition} 
Suppose $(\dev, \hol)$ defines a cyclic-fibered $(\Gtwosplit, \Eintwothree)$-structure on $\mathscr{C}$. 
We denote $\overline{\hol} := \rho$ and $\overline{\dev}: \overline{\mathscr{C}} \rightarrow \Eintwothree$ as the descended developing map
to the $\pi_1S$-cover $\overline{\mathscr{C}} = UT\tilde{S} \oplus UT\tilde{S} \oplus \uline{\R_+}$ of $\mathscr{C}$. 
It follows that $\overline{\dev}$ maps the fiber $\overline{\mathscr{C}}_p$ bijectively onto $\mathscr{S}_p = \pi( \widehat{\mathscr{S}}_p)$. Also, note that the lift
$\widehat{\dev}$ in the above definition is unique, as $-\widehat{\dev}$ has associated splitting $(-\hat{x}_p , T_p , N_p , B_p)$. Thus,
the associated map $-\mathscr{F}^\C$ is not a (CT) cyclic surface due to the orientation condition. \\

In the case we have a cyclic-fibered structure $(\dev, \hol)$ on $ \mathscr{C}$, the equivariance of $\dev$ implies that the map 
$\mathscr{F}$ is $\overline{\hol}$-equivariant. Moreover, there are a number of associated maps to $\mathscr{F}$. 
By Theorem \ref{FrenetFrameTheorem} and Remark \ref{InverseFrenetFrame}, the $\overline{\hol}$-equivariant cyclic surface $\mathscr{F}^\C$ projects to a $\overline{\hol}$-equivariant alternating almost-complex curve $\hat{\nu}: \tilde{S} \rightarrow \quadric$, where $\hat{\nu}(p) = \hat{x}_p$ in our notation. 
Using the Frenet frame splitting, there are also $\hol$-equivariant maps
$\Pi_T: \widetilde{\mathscr{C}} \rightarrow \Ha^{3,3} $ by $ \Pi_T( X_p ) := \Pi_{T_p} \circ \dev(X_p) $ as well as 
$\Pi_N: \widetilde{\mathscr{C}} \rightarrow \Stwofour $ by $ \Pi_N( X_p) := \Pi_{N_p} \circ \dev(X_p) $.
Moreover, using $ \hat{\nu}: \tilde{S} \rightarrow\hat{\mathbb{S}}^{2,4}$ instead of $[\nu$], 
we can construct construct lifts $\hat{\Pi}_T, \hat{\Pi}_N: \widetilde{\mathscr{C}} \rightarrow \imoct$ of $ \Pi_T, \Pi_N$, respectively. 
Denoting $\pi: \overline{\mathscr{C}} \rightarrow \tilde{S}$ as the bundle projection, these maps are related by 
\begin{align}\label{DevEquation}
	 \overline{\dev}(x) = [ \hat{\nu}( \, \pi(x) \, ) + \hat{\Pi}_T(x) + \hat{\Pi}_N(x) ].
\end{align}
We may also define a function $g_R: \widetilde{\mathscr{C}}  \rightarrow \R_+$ by 
\begin{align}\label{RadialFunction} 
	g_R(x) = q_{\imoct}( \, \hat{\Pi}_N(x) \,)^{1/2}. 
\end{align}
Using that $\overline{\dev}(x) \in \Eintwothree$ is a null line, the equation \eqref{DevEquation} may be refined to the following form:
\begin{align}\label{DevReduction} 
	 \overline{\dev}(x) =  \left[  \hat{\nu}(\, \pi(x) \, ) + (g_R^2(x) +1)^{1/2} \, \tilde{\Pi}_T(x) + g_R(x) \, \tilde{\Pi}_N(x) \right],
\end{align} 
 where $\tilde{\Pi}_T(x) := \frac{ \hat{\Pi}_T(x)}{|q( \hat{\Pi}_T(x) )|^{1/2}}$ and $ \tilde{\Pi}_N(x) :=\frac{ \hat{\Pi}_N(x)}{q( \hat{\Pi}_N(x) )^{1/2}}$; this division is possible by Lemma \ref{USplitting}. 
 Observe that by equation \eqref{DevReduction}, the developing map $\overline{\dev}$ of a cyclic-fibered structure on $\mathscr{C}$ is determined by the tuple $(\hat{\nu}, \tilde{\Pi}_T, \tilde{\Pi}_N, g_R)$. We now introduce additional conditions 
for the maps $\tilde{\Pi}_T, \tilde{\Pi}_N, g_R$. In the definition below, we regard $g_R$ as a map $g_R: \overline{\mathscr{C}} \rightarrow \R_+$. 
 
 \begin{definition} 
 Suppose $(\dev, \hol)$ defines a cyclic-fibered structure on $\mathscr{C}$. We say that the structure is \textbf{radial} when two further conditions occur. 
 \begin{itemize}
	\item For any point $p \in \tilde{S}$ and $r \in \R_+$, we have $g_R(p, (u, v, r )) = g_R(p, (u', v', r )) $ for all $(u,v) \in T_p\tilde{S} \times T_p\tilde{S} $ and $(u',v') \in T_p\tilde{S} \times T_p\tilde{S} $. That is, 
	 $g_R$ descends to a map $g_R: \tilde{S} \times \R_+ \rightarrow \R_+$. 
	 \item For $x \in \tilde{S}$, the map $ g_R|_{ \{x\} \times \R_+}: \R_+ \rightarrow \R_+$ is an orientation-preserving diffeomorphism. 
\end{itemize} 
\end{definition} 
In fact, if $(\dev, \hol)$ is a cyclic-fibered, radial geometric structure on $\mathscr{C}$, then by $\hol$-equivariance of $\dev$, the map $g_R$, descends to a map $g_R: S \times \R_+ \rightarrow \R_+$. 

We need one final constraint on our geometric structures. 
\begin{definition} 
We say a cyclic-fibered structure is \textbf{compatible} with its associated almost-complex curve $\hat{\nu}$ when the following rays coincide:
 \begin{align}
 	\R_+ \{ \, \tilde{\Pi}_T( p, (u, v, r)) \} \, &=  \R_+ \, \{ d\hat{\nu}_x(u) \, \} \label{Compatible1} \\
	\R_+ \{ \, \tilde{\Pi}_N(p, (u, v, r))) \, \} &= \R_+ \, \{ \, \sff (d\hat{\nu}(u), d\hat{\nu}(v)) \, \} \label{Compatible2}.
\end{align} 
\end{definition} 

Combing the three conditions together, we finally have the geometric structures of interest.

\begin{definition}\label{Def:CCRPairs}
Call a development-holonomy pair $\dev: \widetilde{\mathscr{C}} \rightarrow \Eintwothree, \, \hol: \pi_1 \mathscr{C} \rightarrow \Gtwosplit$, that is cyclic-fibered, compatible, and radial to be a \textbf{CCR pair} on $\mathscr{C}$ going forward. 
\end{definition} 

Next, we note that the developing map constructed in Section \ref{Subsection:DevMap} satisfies the CCR conditions.

\begin{lemma}\label{DevCCR}
The developing map $\dev$ from Lemma \ref{Lemma:DevelopingMap} and its and its holonomy define a CCR pair on $\mathscr{C}$. 
\end{lemma} 

\begin{proof} 
Lemma \ref{Lemma:DevelopingMap} shows that $\dev$ has holonomy that factors through $\pi_1S$. It is clear that $\overline{\dev}$ develops the fiber $\overline{\mathscr{C}}_p$ bijectively onto the ($\radialtorus$)-family $\mathscr{S}(\R\{\hat{\nu}(p)\}, T_p, N_p)$ in $\Eintwothree$, which lifts to the ($\radialtorus$)-family
in $\widehat{\Ein}^{2,3}$ associated to $(\hat{\nu}(p), T_p, N_p).$ Since $\hat{\nu}$ is an equivariant alternating almost-complex curve, Theorem \ref{FrenetFrameTheorem} says the Frenet frame $\mathscr{F}^\C: \tilde{S} \rightarrow \Gtwo^\C/T$
of $\hat{\nu}$ defines a (CT) cyclic surface. Thus, $\overline{\dev}$ is cyclic-fibered. The compatibility and radial conditions are obviously satisfied by
\eqref{DevMapFromNu}. \end{proof}

\vspace{2ex}

\subsection{The Moduli Space $\mathscr{M}$ of Geometric Structures}\label{ModuliSpace}

In this section, we define the moduli space $\mathscr{M}$ of CCR geometric structures up to isomorphism. Once $\mathscr{M}$ is defined,
we show that Lemma \ref{Lemma:DevelopingMap} naturally yields a map 
$s: \Hit(S,\Gtwosplit) \rightarrow \mathscr{M}$, which, on the level of representatives, turns $\rho$ into a $\rho$-equivariant almost-complex curve $\nu$,
then turns $\nu$ into a $(\rho \, \circ \, \pi_*)$-equivariant developing map $\overline{\dev}$, whose equivalence class defines a point in $\mathscr{M}$. \\

To start, we recall some standard terminology. We call a smooth manifold $M$ endowed with a maximal atlas of diffeomorphic charts to $X$ with locally constant transitions in $G$ a \emph{$(G,X)$-manifold}. Let $M$ be a smooth manifold and $\phi: M \rightarrow Y$ be a diffeomorphism of $M$ onto a $(G,X)$-manifold $Y$. Then we call $(M, \phi)$ a \emph{marked} $(G,X)$-manifold. 
A map $\psi: M_1 \rightarrow M_2$ between two $(G,X)$-manifolds is a \emph{${(G,X)}$-map} when $\psi$ is locally expressed in any pair of charts by 
the restriction of a single element $g \in G$. A diffeomorphism $\psi: M_1 \rightarrow M_2$ such that both $\psi$ and $\psi^{-1}$ are $(G,X)$-maps is a
\emph{$(G,X)$-isomorphism}.

We now need one more formal definition before we may define $\mathscr{M}$. 
\begin{definition} 
Suppose that $F_i \rightarrow M$ are fiber-bundles over a manifold $M$. Form the direct sum bundle $F := \bigoplus_{i=1}^n F_i$. Let 
$f_i: F_i \rightarrow F_i$ be fiber-preserving diffeomorphisms of $F_i$, each lifting the same diffeomorphism $\overline{f}: M \rightarrow M$ of the base. Then we say $(f_i)_{i=1}^n$ induce the map
$f:= \bigoplus_{i=1}^n f_i$, where $f: \bigoplus_{i=1}^n F_i \rightarrow \bigoplus_{i=1}^n F_i$ is given by $f = (f_1, \dots, f_n)$. 

We call a diffeomorphism $f \in \Diff( \bigoplus_{i=1}^n F_i)$ \textbf{factor-preserving} when $f = \bigoplus_{i=1}^n f_i$ for $(f_i)$ fiber-preserving diffeomorphisms
of $F_i$ lifting the same diffeomorphism $\overline{f}$ of the base $M$. More generally, if $F_i \rightarrow M$ and $F_i' \rightarrow N$ are $X_i$-fiber bundles, with $f_i: F_i \rightarrow F_i'$ fiber-preserving diffeomorphisms lifting the same diffeomorphism $\overline{f}: M \rightarrow N$, then $(f_i)$ induce a diffeomorphism $f:= \bigoplus f_i$ that we call factor-preserving. 
\end{definition} 

Here, we recall the well-known fact that a $(G,X)$-structure, in the chart-atlas sense, is equivalent data to a $G$-orbit of developing-holonomy 
pairs. Here, the $G$-action is by $g \cdot (\dev, \hol) = (L_g \circ \dev, C_g \circ \hol)$, where $L_g: X \rightarrow X$ is $L_g(x) = g \cdot x$ and $C_g(h) = ghg^{-1}$ (cf. \cite{Gol22} Part 2, on locally homogeneous geometric structures.) 

The notion of a CCR pair $(\dev, \hol)$ on $\mathscr{C}$ is a delicate one. We cannot precompose such a developing map $\dev$ by any map  $f \in \Diff(\mathscr{C})$ and expect
$(\dev \circ \tilde{f}, \hol)$ to still be a CCR pair. This leads to the following remark and subsequent definitions. 

\begin{remark}\label{RadialTori}
Given a diffeomorphism $F: M_1 \rightarrow M_2$ between two smooth manifolds, $F$ induces a diffeomorphism
$\widehat{dF}: UTM_1 \stackrel{\cong}{\longrightarrow} UTM_2$ between unit tangent bundles (where $UTM_i$ should be regarded
as $TM/\R_+$.) The map $F$ also induces a factor-preserving diffeomorphism $RT(F)$ between the ``radial tori'' $RT(M_i): = UTM_i \oplus UTM_i \oplus \uline{\R_+}$ given by $RT(F) = \widehat{dF} \oplus \widehat{dF} \oplus \mathsf{id}$. Here, $\mathsf{id}: M_1 \times \R_+ \rightarrow M_2 \times \R_+$ means $\mathsf{id}(p, r) = (F(p), r)$. 
\end{remark} 

We use the notion of such induced maps $\widehat{dF}$ in the following definition. 

\begin{definition} \label{CCRStructure}
Suppose that $M$ is a $(\Gtwosplit, \Eintwothree)$-manifold with underlying smooth manifold $UT\mathcal{S} \oplus UT\mathcal{S} \oplus \uline{\R_+}$,
where $\mathcal{S}$ is an oriented surface. Let $\phi: \mathscr{C} \rightarrow M$ be a factor-preserving diffeomorphism  
of the form $\phi = \widehat{df} \oplus \widehat{df} \oplus \phi_+$, where $f: S \rightarrow \mathcal{S}$ is orientation-preserving and $\phi_+$ is orientation-preserving on the $\R_+$-fibers. 
Pull back the $(\Gtwosplit, \Eintwothree)$-structure on $M$ to $\mathscr{C}$, denote it as $\mathcal{G}$. 
We then call the pair $(M, \phi)$ a \textbf{marked CCR structure on} $\mathscr{C}$ when any $(\dev, \hol)$ pair of $\mathcal{G}$ is a
CCR pair. \end{definition}

Let us interpret equation \eqref{DevMapFromNu} in a more general sense. Given $(\hat{\nu}, \rho) \in \mathcal{A}(\mathcal{S})$ and $g_R:\mathcal{S} \times \R_+ \rightarrow \R_+$,
where $\mathcal{S}$ is an oriented surface, we can define $\overline{\dev}: RT(\mathcal{S}) \rightarrow \Eintwothree $ by 
\begin{align}\label{Eqn:GeneralDevMap}
	\overline{\dev}(\hat{\nu}, g_R)(p,u,v,r) : = \left[ \hat{\nu}(p) +(g_R^2(p,r)+1)^{1/2} \frac{ d\hat{\nu}(u) }{ |q(d\hat{\nu}(u))|^{1/2}} + g_R(p,r)  \frac{ \sff_{\hat{\nu}}(d\hat{\nu}(u), d\hat{\nu}(v) )}{ q(\sff_{\hat{\nu}}(d\hat{\nu}(u), d\hat{\nu}(v) )^{1/2}} \right] .
\end{align} 
The map $\overline{\dev}(\hat{\nu}, g_R)$ is a CCR developing map that is a $(\rho \circ \pi_*)$-equivariant immersion by the proof of Lemma \ref{Lemma:DevelopingMap}. 
Then in the notation of the previous definition, $\phi$ induces a map $\overline{\phi}: RT(\tilde{S}) \rightarrow RT(\tilde{\mathcal{S}})$, where
$\overline{\phi}:= \widehat{d \tilde{f}} \oplus \widehat{d \tilde{f}} \oplus \tilde{\phi}_+$, $\tilde{f}: \tilde{S} \rightarrow \tilde{\mathcal{S}}$ lifts $f$, and $\tilde{\phi}: \tilde{S} \times \R_+ \rightarrow \tilde{\mathcal{S}} \times \R_+$ is uniquely constrained to lift $\phi$ 
and lift $\tilde{f}$. Then one finds the equality 
$$\overline{\dev}(\hat{\nu}, g_R) \circ \overline{\phi} = \overline{\dev}(\hat{\nu} \circ \tilde{f}, g_{R} \circ \tilde{\phi}).$$
Thus, such maps $\phi$ in Definition \ref{CCRStructure} are the natural ``CCR-preserving maps'': they pull back CCR developing maps on 
$RT(\mathcal{S})$ to CCR developing maps on $\mathscr{C} = RT(S)$. 

We now define the relevant notion of isomorphism to build our moduli space $\mathscr{M}$ of interest. 

\begin{definition}\label{ModuliSpaceDefn}
The moduli space $\mathscr{M}$ is the quotient of the space $\widehat{\mathscr{M}}$ of all marked CCR structures on $\mathscr{C}$ by the equivalence relation $\sim$.
Here, $(M_i, \phi_i) \in \widehat{\mathscr{M}}$ has underlying smooth manifold $M_i = UT\mathcal{S}_i \oplus UT\mathcal{S}_i \oplus \uline{\R_+}$
and $\phi_i = \widehat{df}_{i} \oplus \widehat{df}_{i} \oplus \phi_{i,+}$. 
Declare $(M_1, \phi_1) \sim (M_2, \phi_2)$ when there exists a $(\Gtwosplit, \Eintwothree)$-isomorphism $\psi: M_1 \rightarrow M_2$ such that
\begin{enumerate}[label=(\roman*)] \label{Defn:ModuliSpace} 
	\item  $\psi = \widehat{df} \oplus \widehat{df} \oplus \psi_+$ is a factor-preserving diffeomorphism, for $f: \mathcal{S}_1 \rightarrow \mathcal{S}_2$ orientation-preserving.
	\item $\psi_+$ is orientation-preserving on the $\mathbb{R}_+$-fibers.
	\item The maps $f \circ f_1 $ and $ f_2$ are isotopic and $\psi_+ \circ \phi_{1, +}$ and $\phi_{2,+}$ are isotopic via fiber-preserving diffeomorphisms that preserve orientation on the $\R_+$-fibers.  
\end{enumerate} 
\end{definition} 
With the $(\dev, \hol)$ perspective, we 
topologize $\mathscr{M}$ now.  \\

Now, let $(M, \phi)$ be a marked CCR structure on $\mathscr{C}$. 
Denote $ \dev_M : \tilde{M} \rightarrow \Eintwothree$ as one of its developing maps. Lift the marking $\phi$ to a diffeomorphism
$\tilde{\phi}: \tilde{\mathscr{C}} \rightarrow \tilde{M}$ and define $ \dev: = \dev_M \circ \tilde{\phi} $. Then $\dev$ is 
equivariant with respect to $\hol :=  \hol_M \circ \phi_*$. In particular, we see that $(M,\phi) \in \widehat{\mathscr{M}}$ corresponds to the 
$\Gtwosplit$-orbit $\{ g\cdot (\dev, \hol) \; | \; g \in \Gtwosplit\}$ of development-holonomy pairs. Now, consider the space $\mathscr{M}'$ of 
CCR pairs on $\mathscr{C}$, in the sense of Definition \ref{Def:CCRPairs}. Then we see $\widehat{\mathscr{M}}$ is naturally bijective 
with $\mathscr{M}'/\Gtwosplit$, under the previously described $\Gtwosplit$-action.

Next, we
reinterpret the equivalence relation $\sim$ on $\widehat{\mathscr{M}}$ instead on $\mathscr{M}'/\Gtwosplit$. 
Fixing a basepoint $p_0 \in \mathscr{C}$, we define $\Diff_{F}(\mathscr{C})  < \Diff_0(\mathscr{C}, p_0)$ as the subgroup of factor-preserving diffeomorphisms $\phi$ of $\mathscr{C}$ of the form $\phi = \widehat{df} \oplus \widehat{df} \oplus \phi_+$ such that $f \in \Diff_0(S)$, and $\phi_+$ is isotopic to 
$\id_{\uline{\R_+}}$ via fiber-preserving diffeomorphisms that preserve the orientation on $\R_+$-fibers. 
Fixing a lift $\tilde{p}_0 \in \tilde{\mathscr{C}}$, we get a unique lift of $\phi$ to $\tilde{\phi} \in \Diff( \tilde{\mathscr{C} } , \tilde{p}_0)$ via the standard lifting lemma. In this way, $\Diff_F(\mathscr{C})$ acts on the space of developing maps by pre-composition. 
The action of $\Diff_{F}(\mathscr{C} )\times \Gtwosplit$ on $\mathscr{M}'$ by $(\phi ,g ) \cdot (\dev, \hol) = ( L_g \circ \dev \circ \tilde{\phi}, \; C_g \circ \hol)$ 
captures the equivalence relation $\sim$. That is, 
there is a natural bijection $\mathscr{M}'/(\Diff_{F}(\mathscr{C})\times \Gtwosplit) \rightarrow \mathscr{M}$.
We topologize 
$\mathscr{M}'$ with the topology of $C^{\infty}$-convergence on compacta on the developing maps and 
$\mathscr{M}$ inherits the quotient topology from $\mathscr{M}'$.

We are heading towards a definition of the map $s: \Hit(S, \Gtwosplit) \rightarrow \mathscr{M}$.  
Now, the map $s: \Hit(S, \Gtwosplit) \rightarrow \mathscr{M}$ we wish to define factors through $\mathcal{H}(S)$. By Theorem \ref{HitchinAC}, we have an inverse map $s_{AC}: \Hit(S, \Gtwosplit) \rightarrow \mathcal{H}(S)_{6g-6}$ of the holonomy map. Then the map $s$ to be defined can instead be written $s = s_G \circ s_{AC}$, where $s_G:  \mathcal{H}(S)_{6g-6} \rightarrow \mathscr{M}$. The map $s_G$ is (implicitly) defined via Section \ref{Subsection:DevMap}. To show the well-definedness of $s$, it suffices to show $s_G$ is well-defined, which is completed below.

\begin{proposition}
Lemma \ref{Lemma:DevelopingMap} yields a well-defined map $s_G: \mathcal{H}(S)_{6g-6} \rightarrow \mathscr{M}$. 
\end{proposition} 

\begin{proof} 
By Lemma \ref{DevCCR}, the developing map defined in Lemma \ref{Lemma:DevelopingMap} amounts to a map 
$\widehat{s}: \mathcal{A}(S)_{6g-6} \rightarrow \mathscr{M}'$, where
$\mathscr{M}'$ is the space of all CCR
pairs on $\mathscr{C}$ and $\mathcal{A}(S)_{6g-6}= \{ \, (\hat{\nu}, \rho) \in \mathcal{A}(S) \; | \; b(\hat{\nu}) = 6g-6\}$. 
The map $\widehat{s}$ is described by $\widehat{s}(\hat{\nu}, \rho) = \overline{\dev}(\hat{\nu}, g_{R}^0)$, where $g_R^0(p, r) = r$, in the sense of equation \eqref{Eqn:GeneralDevMap}. 
We show $\widehat{s}$ descends to a well-defined 
map $s: \Hit(S,\Gtwosplit) \rightarrow \mathscr{M}$. 

Take $[\rho] \in \Hit(S, \Gtwosplit)$ and denote $[(\hat{\nu}, \rho)] := s_{AC}(\,  [\rho] \, )$ and choose representatives\\
 $(\hat{\nu}_1, \rho_1), (\hat{\nu}_2, \rho_2) \in [(\hat{\nu}, \rho)]$. Hence, write $\rho_2 = C_g \circ \rho_1$, where $C_g: \Gtwosplit \rightarrow \Gtwosplit$ is conjugation by some element $g \in \Gtwosplit$ and $\hat{\nu}_2 = L_g \circ \hat{\nu}_1 \circ \tilde{f} $, where $\tilde{f} \in \Diff(\tilde{S})$ is a lift of $f \in \Diff_0(S, p_0)$. 
Define $\overline{\dev}_i := \widehat{s}(\hat{\nu}_i, \rho_i)$. 

Now, let $RT(f) \in \mathsf{Diff}_F(\mathscr{C})$ be induced by $f$. Note that $RT(f)$ is isotopic to $\id_{\mathscr{C}}$ by fiber-preserving diffeomorphisms,
as $RT(f_t)$ is such an isotopy, where $f_t$ is an isotopy of $f$ to $\id_{S}$. 
The equality $\overline{\dev}_2 = L_g \circ \overline{\dev}_1 \circ RT(\tilde{f})$ implies that 
$ (RT(f), g) \cdot (\dev_1, \hol_1) = (\dev_2, \hol_2)$,
so we conclude $[(\dev_1, \hol_1)] = [(\dev_2, \hol_2)] $ in $\mathscr{M}$.
 \end{proof} 

\vspace{2ex}
\subsection{Continuity of the Map $s: \Hit(S, \Gtwosplit) \rightarrow \mathscr{M}$}\label{Continuity} 

Before we give the formal proof that $s$ is continuous, we explain the idea of factoring the map. 
We imagine the associations in the following order:
from $[\rho] \in \Hit(S, \Gtwosplit)$, using Labourie \cite{Lab17}, we construct 
the pair $( [\Sigma], [q]) $ with $[\Sigma] \in \Teich(S)$ and $q \in H^0(K_{\Sigma}^6)$. Next, we prove the continuous $C^{\infty}$-dependence of the solutions $(\psi_1, \psi_2)$ to Hitchin's equations for cyclic $\Gtwosplit$-Higgs bundles with respect to the input data $(\sigma, q)$,
where $\sigma$ is a conformal metric on $\Sigma$. 
The invariants $r,s$ of the almost-complex curve $\hat{\nu}$ from Hitchin's equation \eqref{HitchinsEquations_rs} are then smoothly constructed from $(\psi_1, \psi_2, \sigma, q)$ and then $\hat{\nu}$ can be smoothly constructed from these invariants. This shows the continuity of the map $s_{AC}: \Hit(S, \Gtwosplit) \rightarrow \mathcal{H}(S)_{6g-6}$.
The process of constructing the geometric structure from the almost-complex curve is then continuous map because the $C^{\infty}$-convergence
of the almost-complex curves $\hat{\nu}$ forces the $C^{\infty}$-convergence of the differentials $d\hat{\nu}$ and second fundamental forms $\sff_{\hat{\nu}}$, which altogether controls the developing map.

\begin{lemma} \label{s_continuity}
The map $s: \Hit(S, \Gtwosplit) \rightarrow \mathscr{M}$ is continuous.
\end{lemma} 

\begin{proof}

To start, we use Labourie's map $\Psi^{-1}: \Hit(S, \Gtwosplit)\rightarrow \V_{6}$, (as discussed in Section \ref{HiggsBundlePreliminaries}) 
where $\V_6 \rightarrow \Teich(S)$ is the bundle over Teichm\"uller space with
fiber at $\Sigma = (S,J)$ given by $H^0(\Sigma, \K_{\Sigma}^6)$. 

We first prove the continuity of the map $s_{AC}: \Hit(S, \Gtwosplit) \rightarrow \mathcal{H}(S)_{6g-6}$. 
Choose any sequence of representations $[\rho_i] \in \Hit(S, \Gtwosplit)$ such that  $[\rho_i ]\rightarrow [\rho_{\infty}] $. Write $\Psi( [\rho_i] = ([\sigma_i], [q_i])$. Then the unique hyperbolic representatives
$\sigma_i \in [\sigma_i]$ satisfy $\sigma_i \stackrel{C^{\infty}}{\longrightarrow} \sigma_{\infty}$ and by continuity of $\Psi$, the associated holomorphic sextic differentials
$q_i$ converge smoothly to $q_{\infty}$. Denote
$(\psi^1_i, \psi^2_i)$ as the unique solution to Hitchin's equations \eqref{G2Hitchin_GeneralMetric} with respect to $(\sigma_i, q_i)$. 
In Appendix \ref{Appendix:ImplicitFunctionTheorem} in Lemma \ref{HitContinuity}, using the implicit function theorem, we show that $(\psi^1_i, \psi^2_i)$ 
converges in $C^{\infty}(S)\times C^{\infty}(S)$ to $(\psi^1_{\infty}, \psi^2_{\infty})$, the solution to \eqref{G2Hitchin_GeneralMetric} with respect to $(\sigma_{\infty}, q_{\infty})$. 
Next, the tensors $(r_i, s_i) $ given by  $s_i= e^{2\psi^2_i} \sigma_i$ and $r_i = e^{\psi^1_i-\psi^2_i}\sigma_i^2$ converge smoothly to $s_{\infty} = e^{2\psi^2_{\infty}} \sigma_{\infty}$
and $r_{\infty} =  e^{\psi^1_{\infty} -\psi^2_{\infty} }\sigma_{\infty}^2$, respectively. 
Since $\Diff_0(S)$ acts freely on the space of all complex structures on $S$,
then there is a unique $\Gtwosplit$-orbit $\Gtwosplit \cdot \{ (\hat{\nu}_i, \rho_i) \}$ of almost-complex curves (in the conventional sense) $\hat{\nu}_i: (\tilde{S}, \tilde{\sigma}_i) \rightarrow \quadric$ with the invariants $(r_i, s_i, q_i)$. 
By Proposition \ref{StiefelTripletModel}, we can remove the $\Gtwosplit$-freedom and choose a distinguished representative $\hat{\nu}_i$ as follows: fix $p_0 \in \tilde{S}, X \in T_{p_0} \tilde{S}$ and demand $\hat{\nu}_i(p_0) = i, \, d \hat{\nu}_i|_{p_0}(X) = l, \, \sff_{\hat{\nu}_i}(X, X) = j$. Note that by abuse $i$ is both an index $i \in \mathbb{Z}_+$ and an element $i \in \imoct$ in the previous sentence. Then 
it follows that the almost-complex curves $\hat{\nu}_i$ satisfy $ \hat{\nu}_i  \stackrel{C^{\infty}}{\longrightarrow} \hat{\nu}_{\infty}$. This proves the continuity of $s_{AC}$. 

Next, we consider the map $s_G: \mathcal{H}(S)_{6g-6} \rightarrow \mathscr{M}$.
Recall the space $\mathcal{A}(S)$ from \eqref{LiftedModuliHS}. Define $\mathcal{A}(S)_{6g-6} = \{ \,(\hat{\nu}, \rho) \in \mathcal{A}(S) \; | \; \deg \mathcal{B}_{\hat{\nu}}  = 6g-6\},$
so that $\mathcal{H}(S)_{6g-6}  = \mathcal{A}(S)_{6g-6} / (\Diff_0(S) \times \Gtwosplit) $. 
Pick any point $[\hat{\nu}_{\infty}] \in \mathcal{H}(S)_{6g-6} $. 
Choose a sequence $[\hat{\nu}_i] \in \mathcal{H}(S)$ such that
$[\hat{\nu}_i] \rightarrow [\hat{\nu}_{\infty}]$ as well as representatives $\hat{\nu}_i \in[\hat{\nu}_i] $ such that 
$\hat{\nu}_i \stackrel{C^{\infty}}{\longrightarrow} \hat{\nu}_{\infty}$. To prove the continuity of $s_{G}$, we show 
$\dev(\hat{\nu}_i) \stackrel{C^{\infty}}{\longrightarrow} \dev(\hat{\nu}_{\infty})$. 
By the $C^{\infty}$-convergence of $\hat{\nu}_i$ to $\hat{\nu}_{\infty}$, the maps $d\hat{\nu}_i: T \tilde{S} \rightarrow T\quadric$ also converge smoothly on compacta
to $d\hat{\nu}_{\infty}$. Recall the unit norm projection $\tilde{\Pi}_{T}: T\tilde{S} \rightarrow Q_-(\imoct)$ given by $\tilde{\Pi}_{T}(X) = \frac{d\hat{\nu}(X)}{(-q( d\hat{\nu}(X))^{1/2}}$.
The maps $(\tilde{\Pi}_T)_i $ converge $C^{\infty}$ on compacta to $(\tilde{\Pi}_T)_{\infty}$.
Recall also the unit norm projection $\tilde{\Pi}_N: T\tilde{S} \oplus T\tilde{S} \rightarrow \quadric$ by $\tilde{\Pi}_{N}(p, X,Y) = \frac{ \sff_p(X,Y)}{q(\sff_p(X,Y))^{1/2}}$. 
Here, we regard the second fundamental form of $\hat{\nu}_i$ as a map 
$\sff_i: T\tilde{S} \oplus T\tilde{S} \rightarrow \imoct $ by 
$$\sff_i(p, X, Y) = \nabla^{\quadric}_{d\hat{\nu}_i(X)} (d\hat{\nu}_i(Y))
- \mathsf{proj}_{d\hat{\nu}_i(T_p\tilde{S}) } \nabla^{\quadric}_{d\hat{\nu}_i(X)} (d\hat{\nu}_i(Y)).$$
By the $C^{\infty}$-convergence of $\hat{\nu}_i, d\hat{\nu}_i$, it follows that $\sff_{i} \stackrel{C^{\infty}}{\rightarrow} \sff_{\infty}$ and then 
 $(\tilde{\Pi}_N)_{i} \stackrel{C^{\infty}}{\rightarrow} (\tilde{\Pi}_N)_{\infty}$ as well. 
The radial functions $(g_R)_{i}, (g_R)_{\infty}$ from \eqref{RadialFunction} are all just $g_{R}(p, r) = r$ by the
definition of the map $s_G$. 
Since $\hat{\Pi}_T = (g_R^2+1)^{1/2} \tilde{\Pi}_T, \, \hat{\Pi}_N= g_R\, \tilde{\Pi}_N$, it follows that $\hat{\nu_i } \stackrel{C^{\infty}}{\rightarrow} \hat{\nu}_\infty$ implies 
\begin{align}\label{TupleConvergence} 
	(\hat{\nu}_i, (\hat{\Pi}_T)_i,  (\hat{\Pi}_N)_i) \stackrel{C^{\infty}}{\longrightarrow}  (\hat{\nu}_{\infty}, (\hat{\Pi}_T)_{\infty},  (\hat{\Pi}_N)_{\infty}),
\end{align}
where now we regard $\hat{\Pi}_T, \hat{\Pi}_N$ as maps $\overline{\mathscr{C}} \rightarrow \imoct$. 
Now, we have a lift $\hat {\dev}_i: \overline{\mathscr{C}} \rightarrow \imoct$ by $ \hat{\dev}_i= \hat{\nu}_i+ (\hat{\Pi}_T)_i + (\hat{\Pi}_N)_i$. Thus, 
\eqref{TupleConvergence} says that $\hat{\dev}_i \stackrel{C^{\infty}}{\rightarrow} \hat{\dev}_{\infty}$, which then implies $\dev_i \stackrel{C^{\infty}}{\rightarrow} \dev_{\infty}$. 
\end{proof} 

\subsection{The Descended Holonomy Map $\alpha: \mathscr{M} \rightarrow \Hit(S,\Gtwosplit)$} \label{GXtoRepn}
In this section, we show the descended holonomy map $\alpha: \mathscr{M} \rightarrow \Hit(S,\Gtwosplit)$ is the inverse of the map $s$ and thereby prove the main theorem. 
The map $s$ satisfies by definition that $\alpha \circ s = \id_{\Hit(S,\Gtwosplit)}$, so we need only verify $s \circ \alpha = \id_{\mathscr{M}}$ here. \\

We now show that $\alpha$ factors through $\mathcal{H}(S)_{6g-6}$ via 
the following map $H$.

\begin{proposition}\label{Hmap} 
There is a well-defined map $H: \mathscr{M} \rightarrow \mathcal{H}(S)_{6g-6}$.
\end{proposition} 

\begin{proof}
We first define a lift $\hat{H}: \mathscr{M}' \rightarrow \mathcal{A}(S) $ of $H$ on the level of representatives. 
Let $(\dev, \hol)$ be a CCR pair on $\mathscr{C}$. 
By the cyclic-fibered condition from definition \ref{CyclicCondition}, the 
associated map $\mathscr{F}^\C: \tilde{S} \rightarrow \Gtwo^\C/T$ is a $\overline{\hol}$-equivariant cyclic surface. 
We can then invoke Theorem \ref{FrenetFrameTheorem} and Remark \ref{InverseFrenetFrame} to see that $\pi_{\quadric} \circ \mathscr{F}^\C$ 
is a $\overline{\hol}$-equivariant alternating almost-complex curve $\hat{\nu}: \tilde{S} \rightarrow \quadric$. Define $\hat{H}( \dev, \hol) := (\hat{\nu}, \overline{\hol})$. 

We claim $\hat{H}$ descends to a well-defined map $H: \mathscr{M} \rightarrow \mathcal{H}(S)_{6g-6}$ by $ H( \, [(\dev, \hol)]\, ) = [ \, \hat{H}(\dev, \hol) \, ]$. 
First, observe the following naturality properties of $\hat{H}$:  
\begin{itemize}
	\item If $\phi \in \Diff_F(\mathscr{C})$ lifts $f \in \Diff_0(S)$, then 
	$\hat{H}( \phi \cdot (\dev, \hol) \, ) = f \cdot \hat{H}( \dev, \hol)$
	\item If $g \in \Gtwosplit$, then $\hat{H}( g \cdot (\dev, \hol) ) = g \cdot \hat{H}(\dev, \hol)$. 
\end{itemize} 
These equivariance conditions imply the well-definedness of $H$. Indeed, take $[(\dev, \hol)] \in \mathscr{M}$ and two representatives 
$(\dev_i, \hol_i) $ for $ i \in \{1,2\}$. Then the discussion following Definition \ref{ModuliSpaceDefn}, we can write $(\dev_2, \hol_2) = (\phi, g) \cdot (\hol_1, \dev_1)$, 
where $\phi = \widehat{df} \oplus \widehat{df} \oplus \phi_+$ lifts 
some $f \in \Diff_0(S)$. But then we have
$(\hat{\nu}_2, \overline{\hol}_2) = (f, g) \cdot (\hat{\nu}_1, \overline{\hol}_1)$. Since $f \in \Diff_0(S)$, we conclude 
$[ \, \hat{H}(\dev_1, \hol_1) \, ] = [\,   \hat{H}(\dev_2, \hol_2) \, ]$ in $\mathcal{H}(S)$. 

Now, by equation \eqref{Compatible2} of the compatibility condition, along with Remark \ref{NonDegeneracyRemark2},
we see any pair $[(\hat{\nu}, \rho)] = H( \, [(\dev, \hol) ] \, )$ has $\Gtwosplit$-Hitchin holonomy, since the second fundamental form of $\hat{\nu}$ is non-vanishing. 
 \end{proof} 

The descended holonomy map $\alpha: \mathscr{M} \rightarrow \chi(\pi_1S, \Gtwosplit)$ via $ \alpha( \, [(\dev, \hol)] \, ) = [ \, \overline{\hol} \, ]$ factors through the map $H$ as $ \alpha = \hol \circ H$, where $\hol:  \mathcal{H}(S)\rightarrow \chi(\pi_1S, \Gtwosplit)$ is 
the holonomy map $[ \, (\hat{\nu}, \rho) \,] \mapsto [ \rho ]$. Hence, Proposition \ref{Hmap} says $\alpha$ obtains the form $\alpha: \mathscr{M} \rightarrow \Hit(S,\Gtwosplit)$ by Theorem \ref{HitchinAC}. We now prove the main theorem. 

\begin{theorem}\label{MainTheorem}
The descended holonomy map $\alpha:  \mathscr{M} \rightarrow \Hit(S, \Gtwosplit)$ by $[ \,(\dev, \hol )\,] \mapsto [ \, \overline{\hol} \, ]$ is a homeomorphism
onto the $\Gtwosplit$-Hitchin component. 
\end{theorem} 

\begin{proof}
The map $\alpha$ is continuous and the candidate inverse $s: \Hit(S, \Gtwosplit) \rightarrow \mathscr{M}$ is also continuous by Lemma \ref{s_continuity}. 
By the construction of $s$ and $ \alpha$, we have $\alpha \circ s = \id_{\Hit(S,\Gtwosplit)}$. It remains to argue that $s \circ \alpha = \id_{\mathscr{M}}$. We show this now using the CCR conditions. 

Take $G_1 = [ \, (\dev_1, \hol_1) \, ] \in \mathscr{M}$ and define $G_2:= s \, \circ\,  \alpha(G) $. Since $\alpha \circ s = \id$, we have $\alpha(G_1) = \alpha(G_2)$, hence $H(G_1) = H(G_2)$. Thus, choose representatives $(\dev_1, \hol_1) \in G_1, (\dev_2, \hol_2) \in G_2$ such that $\hat{H}(\dev_1, \hol_1) = \hat{H}(\dev_2, \hol_2)$, using the notation from Proposition \ref{Hmap}.  
The key going forwards is the fact from Section \ref{GeometricStructuresDefinition} that the developing map $\dev$ of a CCR structure in $\Eintwothree$ is determined by the tuple $(\hat{\nu}, \tilde{\Pi}_T, \tilde{\Pi}_N, g_R)$ of equivariant maps from \eqref{DevReduction}. The compatibility of the almost-complex curve and the developing map in equations \eqref{Compatible1}, \eqref{Compatible2} says that the projections ${\Pi}_T, {\Pi}_N$ are determined by $\hat{\nu}$. That is, $\overline{\dev}_i$ must obtain the form \eqref{Eqn:GeneralDevMap}. In other words, Since $\hat{\nu}_1 = \hat{\nu}_2$, we actually have $(\tilde{\Pi}_T)_1 = (\tilde{\Pi}_T)_2$ and $(\tilde{\Pi}_N)_1 = (\tilde{\Pi}_N)_2$. Hence, we need only align the radial functions with a re-gauging to finish the proof.  

Recall that the subgroup $\Diff^+(\R^n)$ of orientation-preserving diffeomorphisms of $\R^n$ deformation retracts onto $\{\id_{\R^n}\}$. 
It follows that $\Diff^+(\R_+)$ is contractible too.

We are now ready to show $(\dev_1, \hol_1)$ and $(\dev_2, \hol_2)$ represent the same point in $\mathscr{M}$. Define $f: \mathscr{C} \rightarrow \mathscr{C}$ by
$f = \id \oplus \id \oplus ({g_R})_2$. Then $f$ is a factor-preserving diffeomorphism that is orientation-preserving on the $\R_+$-fibers.
 Since $\Diff^+(\R_+)$ is contractible, $(g_R)_2$ is isotopic to $\id_{\uline{\R_+}}$ via fiber-preserving diffeomorphisms. 
 Then $f$ lifts to $\overline{f}: \overline{\mathscr{C}} \rightarrow \overline{\mathscr{C}}$ and 
one finds that $ \overline{\dev}_2 = \overline{\dev}_1 \circ \overline{f} $, meaning $[ \, (\dev_1, \hol_1) \, ] = [\,  (\dev_2, \hol_2) \, ]$ in $\mathscr{M}$. 
\end{proof}

\section{Explicit Computations in the Fuchsian Case} \label{FuchsianCase}

In this section, we examine the geometric structures associated to $\Gtwosplit$-Fuchsian representations. 
We first review a description of the \emph{principal embedding} $ \iota: \mathsf{PSL}_2(\R) \hookrightarrow \Gtwo'$, in which $\PSL_2\R$ acts on $\R^7$ via the action on symmetric homogeneous sextic polynomials. 
We then describe the Fuchsian almost-complex curves in terms of polynomials and examine the resulting developing maps and study the injectivity of $\overline{\dev}$. 
An exceptional feature of
the Fuchsian case is that there is a `universal' almost-complex curve $f: \Ha^2 \rightarrow \quadric$ that is not just $\rho$-equivariant, 
but $\PSL_2\R$-equivariant; this phenomenon is not particular to the $\Gtwosplit$-Fuchsian setting, as it occurs also in \cite{CTT19, GW08}. By uniqueness, then all $\Gtwosplit$-Fuchsian almost-complex curves 
have the same image, just the identification $\tilde{S} \cong \Ha^2$ changes. 
Then, examining the Guichard-Wienhard domain $\Omega \subset \Eintwothree$ in the Fuchsian case, we show our geometric structures are distinct from theirs: the image of $\overline{\dev}$
intersects both $\Omega$ and $\Eintwothree \backslash \Omega$. We discuss the structure of these intersection in detail. 
				     
\subsection{The Principal Embedding $ \iota: \mathsf{PSL}_2 \R \hookrightarrow \Gtwosplit$ }

Let us recall the principal embedding $ \iota:  \mathsf{PSL}_2 \R \hookrightarrow \mathsf{SL}_7 \R $. First, we 
identify $\R^7 \cong_{\mathbf{Vec}} \Sym^6(\R^2)$ and take a canonical basis $X,Y$ for $\R^2$ in which $  \mathsf{PSL}_2 \R $ acts by the fundamental representation. 
Then $\mathcal{P} = (X^i Y^{6-i})_{i=0}^6$ is a basis for $ \R^7$ in which $ g= \pm \begin{pmatrix} a & b \\ c & d \end{pmatrix}$ acts via $ g \cdot (X^i Y^{6-i}) = (g\cdot X)^i (g\cdot Y)^{6-i} $. 
This representation turns out to preserve the bilinear form $Q_6$ of signature $(3,4)$, which in the basis $\mathcal{P}$ is represented by the following matrix \cite[Page 50]{CTT19}: 
$$ Q_6 = \begin{pmatrix} & & & & && 1 \\
			& && & & -\frac{1}{6} &  \\
			& && & \frac{1}{15}& &  \\
				& && -\frac{1}{20} && & \\
			& & \frac{1}{15}& & &&   \\
			& -\frac{1}{6}&& & &&  \\
			1 & && & &&    \end{pmatrix} .$$ 
We will also refer to the induced quadratic form $Q_6(x) = Q_6(x,x)$ by $Q$ as well. 
Thus, the irreducible representation $\iota$ is a representation into $\mathsf{SO}_0(3,4)$. 
In fact, $\iota$ is a $\Gtwosplit$-representation if identifications are made correctly: we can identify $ \R^7 \cong_{\mathbf{Vec}} \imoct$ 
in such a way that $ \iota(\PSL_2\R)$ preserves the cross-product on $\imoct$. Hence, $\iota$
upgrades to a representation $\iota: \PSL_2\R \hookrightarrow \Gtwosplit$. One such basis is the following $\R$-cross-product basis for $\imoct$: 
\begin{align}\label{CrossProductBasis}
 B = \left( \, \frac{i+li}{\sqrt{2}} , \;  \frac{lj-j}{\sqrt{2}\sqrt{6}}, \; \frac{k-lk}{\sqrt{2}\sqrt{15}} , \; -\frac{l}{\sqrt{20}}, \;\frac{k+lk}{\sqrt{2}\sqrt{15}}, \;\frac{j+lj}{\sqrt{2}\sqrt{6}}, 
\;\frac{i-li}{\sqrt{2}} \right) . 
\end{align} 
As in \cite[Definition 8.21]{Eva22}, we call an ordered basis $(x_i)_{i=3}^{-3}$ an $\R$-cross-product basis for $\imoct$ when $x_i \times x_j = c_{i,j} x_{i+j}$ for constants $c_{i,j} \in \R$. 

	\begin{table}[ht]
 	\begin{center}
			\begin{tabular}{||c | c | c||} 
				 \hline
			$\Sym^6(\R^2) $ & $\R^7$  & $ \imoct $ \\ [0.5ex] 
			 \hline\hline
			$X^6$    &  $\mathbf{e}_3$ & $ \frac{i+li}{\sqrt{2}} $ \\ 
			 \hline
		 	$X^5 Y$  &  $\mathbf{e}_2$  & $\frac{lj-j}{\sqrt{2}\sqrt{6}}$ \\
			 \hline
			$X^4 Y^2$ &  $\mathbf{e}_1$ & $\frac{k-lk}{\sqrt{2}\sqrt{15}}$ \\
					\hline 
			$X^3 Y^3 $ & $\mathbf{e}_0$ & $ -\frac{l}{\sqrt{20}},$ \\
					\hline 
			$X^2 Y^4$ & $\mathbf{e}_{-1} $ &  $\frac{k+lk}{\sqrt{2}\sqrt{15}}$ \\
					\hline 
			$X Y^5$ &  $\mathbf{e}_{-2} $ & $\frac{j+lj}{\sqrt{2}\sqrt{6}}$ \\
					\hline
			$Y^6$ &  $\mathbf{e}_{-3} $  & $\frac{i-li}{\sqrt{2}}$ \\
					\hline 
		\end{tabular}
		\caption{The identifications between ordered bases. \label{BasisIdentification}}
	\end{center} 
	\end{table} 
	
We now discuss why the cross-product is preserved by $\iota(\PSL_2 \R)$. Straightforward calculations show that under the identification $\imoct \cong \Sym^6\R^2$, we see the generators 
\begin{enumerate}
	\item $ \begin{pmatrix} \lambda & 0 \\ 0 & \frac{1}{\lambda} \end{pmatrix}$, \; $ \lambda \in \R$,
	\item $ \begin{pmatrix} 1 & s \\ 0 & 1 \end{pmatrix}$, \; $s \in \R$,
	\item $ \begin{pmatrix} 0 & 1 \\ -1 & 0 \end{pmatrix} $,
\end{enumerate} 
of $\mathsf{PSL}_2 \R$ identify, respectively, as the following matrices in the basis $B'$, where we re-normalize the basis $B$ as follows:
\begin{align}\label{RealCrossProductBasis}
	 B' = \left( \, \frac{i+li}{\sqrt{2}} , \;  \frac{lj-j}{\sqrt{2} }, \; \frac{k-lk}{\sqrt{2} } , \; l , \;\frac{k+lk}{\sqrt{2} }, \;\frac{j+lj}{\sqrt{2} }, 
\;\frac{i-li}{\sqrt{2}} \right). 
\end{align} 
The matrices are: 
\begin{enumerate}
	\item $ \mathsf{diag} ( \lambda^6, \lambda^4, \lambda^2, 1, \lambda^{-2}, \lambda^{-4}, \lambda^{-6} )$. \\
	\item $\mathsf{exp} \left(  s \cdot \begin{pmatrix} 0&\sqrt{6} & & & & &\\
				      & 0& \sqrt{10}& & & & \\
			  	      & &0 &-\sqrt{12}  & & &\\
			 	      & & &0 &-\sqrt{12} & &\\
				      & & & & 0&\sqrt{10}&\\
				      & & & & &0& \sqrt{6} \\
				      & & & & &&0  \end{pmatrix}   \right ) $
	\item $ \begin{pmatrix} & & & & && 1 \\
			& && & & -1 &  \\
			& && & 1& &  \\
				& && -1 && & \\
			& & 1& & &&   \\
			& -1&& & &&  \\
			1 & && & &&  \end{pmatrix} $.
\end{enumerate} 
The matrices (1) and (2) are seen to be in $\Gtwosplit$ in the basis $B'$ by examining the matrix representation $[\g_2']_{B'}$ of $\g_2'$ in the basis $B'$ \cite[page 89]{Bar10}. Then for (3),
observe that the given transformation $g$ is equivalent to $g = +1$ on $\mathsf{Im} \Ha$ and $g = -1$ on $l\Ha$, which is indeed a $\Gtwosplit$-transformation, seen in the 
model $V_{(+,+,-)}$ from Proposition \ref{StiefelTripletModel} as the tuple $(i, j, -l)$. 
Going forward, there is a fixed 3-fold identification of vector spaces $ \Sym^6(\R^2) \cong \R^7 \cong \imoct $, shown in Table \ref{BasisIdentification}. For indexing purposes, we denote the standard basis of $\R^7$ as $(\mathbf{e}_i)_{i=3}^{-3}$.  

\subsection{A Description of $\nu$ and $\dev$ in Coordinates} \label{FuchsianDev} 
In this section, we describe the almost-complex curve explicitly in coordinates in the Fuchsian case. 

The first ingredient is the \emph{Veronese embedding}, which identifies as the $\rho$-equivariant (projective) almost-complex curve $[ \hat{\nu} ]$ for 
$\rho$ a $\Gtwosplit$-Fuchsian representation. Before defining this map, we introduce some more notation. 
Set $\mathbb{S}^{0,2} := \mathbb{P}Q_+ \R^{1,2}$, where we identify $\R^{1,2} \cong (\Sym^2(\R^2), Q_2)$. 
Here, we equip $\Sym^2(\R^2)$ with the $\mathsf{PSL}_2\R $-invariant quadratic form $Q_2 = \begin{pmatrix} & & 1 \\ & - \frac{1}{2} & \\ 1 & & \end{pmatrix}$,
in the basis $(X^2, XY, Y^2)$, which equivalently defines the isomorphism $\PSL_2\R \stackrel{\cong}{\longrightarrow} \mathsf{SO}_0(1,2)$.
Here, $-4Q_2 = \Delta$ is just the discriminant of the polynomial $P \in \Sym^2(\R^2)$. The evidently $\PSL_2\R$-equivariant map $\mathbb{P}\Sym^2(\R^2) \rightarrow \mathbb{P}\Sym^2(\R^6)$ by $ [P] \mapsto [P^3]$ restricts to the smooth 
\emph{Veronese embedding} $\overline{f}: \mathbb{P}Q_+\Sym^2(\R^2)  \rightarrow \mathbb{P}Q_+\Sym^6(\R^2) $ by $[P] \mapsto [P^3]$. An easy calculation shows that 
that for $\mathsf{S} := \mathsf{image}(\overline{f})$, the tangent space is given by $T_{[P^3]} \mathsf{S} = \{ \; P^2 Q \; \; | \; \; Q \bot P, \; Q \in \Sym^2(\R^2) \}$. 
At the point $P_0 = [X^2 +Y^2]$, using $T_{[P_0]} \mathbb{S}^{0,2} = \spann_{\R} \langle XY, X^2- Y^2 \rangle$, we find that 
 $$T_{[P_0^3]} \mathsf{S} =  \spann_{\R} \langle (X^2+Y^2)^2(X^2-Y^2), (X^2 + Y^2)^2(XY) \rangle $$ 
 A calculation under the previous identification with $\imoct$ shows that we indeed have $ [P_0^3] \times_{\imoct} T_{[P_0^3]} S = T_{[P_0^3]} S $. 
 Indeed, $\overline{f}(P_0) = [\sqrt{5}i+ \sqrt{3}k]$ and $
 T_{ \overline{f}(P_0)}\mathsf{S} = \spann_{\R} \langle \sqrt{5}\,  lj - \sqrt{3} \, l, \; \sqrt{15} \, li - lk \rangle $. Hence, $\overline{f}$ 
 is an almost-complex curve by $\PSL_2\R$-equivariance. We show shortly that $f$ is, in fact, alternating. 

Next, we recall the the following $\PSL_2\R$-equivariant map $\hat{g}: \Ha^2 \rightarrow Q_+(\Sym^2(\R^2))$, first defined by Baraglia \cite[Page 63]{Bar10}: 
$$ \hat{g}(z) = \frac{1}{\sqrt{2} \, y} \, (zX+Y )(\zbar X + Y) ,$$
where $z = x + iy$. The map $\hat{g}$ provides an explicit identification of the upper half-plane $\Ha^2 \subset \C$ with $Q_+(\Sym^2\R^2) = \hat{\mathbb{S}}^{0,2}$. 
A simple and direct calculation verifies the equivariance of $\hat{g}$. One easily checks $Q_2( \hat{g})  = + 1$, so that $\hat{g} \in \hat{\mathbb{S}}^{0,2}$. 

The partial derivatives of $\hat{g}$ are shown below: 
\begin{align}
	\hat{g}_x &=   \frac{\sqrt{2}\, x}{y} \, X^2 + \frac{\sqrt{2}}{y} \, XY \\
	\hat{g}_y &=  \frac{y^2- x^2}{\sqrt{2} \, y^2} \, X^2 - \sqrt{2} \frac{x}{y^2} XY - \frac{1}{\sqrt{2}\, y^2} Y^2.
\end{align} 

Next, define the $\PSL_2\R$-equivariant map $ f: Q_+(\Sym^2(\R^2)) \rightarrow Q_+(\Sym^6(\R^2))$ by $ f(P) = c P^3$, lifting the Veronese embedding, by choosing an appropriate constant $c \in \R^*$. Then the map 
$\hat{f}: \Ha^2 \rightarrow \quadric$ by $\hat{f} = f \circ \hat{g}$ is a `universal' alternating almost-complex curve that is $\PSL_2\R$-equivariant. 
While the alternating condition has not been verified yet, this point will be addressed momentarily. 
The point $[\Sigma]$ in Teichm\"uller space merely moves the marking $\phi_{\Sigma}: \tilde{\Sigma} \rightarrow \Ha^2$ around, but the image of any Fuchsian almost-complex curve remains the same:
it is $\mathsf{image}( f).$ 

The Frenet frame splitting of $\hat{f}$ is given as follows:
\begin{align} \label{FuchsianFrenetFrame}
	\begin{cases} 
	 \mathscr{L} &= \spann_{\R} \langle \hat{g}^3 \rangle \\
	 T &=  \spann_{\R} \langle \hat{g}^2 \, \hat{g}_x, \; \hat{g}^2 \, \hat{g}_y \rangle \\ 
	 N & =  \spann_{\R} \langle \, \hat{g}\, \hat{g}_x \, \hat{g}_y, \; \proj_{\mathscr{L}^\bot} (\hat{g} \, \hat{g}_x^2) \; ) \rangle  \\
	 B &= \spann_{\R} \langle \,  \proj_{(\mathscr{L} \oplus T \oplus N)^\bot} (\hat{g}_x^2 \,  \hat{g}_y), \;  \proj_{(\mathscr{L} \oplus T \oplus N)^\bot}( \hat{g}_y^2 \,\hat{g}_x) \; \rangle.
	 \end{cases}
\end{align} 
Up to scalars $c,c', c'' \in \R^*$, one finds $ \sff( \der{x}, \der{y}) = c \, \hat{g} \hat{g}_x \hat{g}_y$ and $\sff(\der{x}, \der{x}) = c' \, \proj_{\mathscr{L}^\bot} (\hat{g}\, \hat{g}_x^2 ) $
and $\sff(\der{y}, \der{y}) = c'' \, \proj_{\mathscr{L}^\bot} (\hat{g}\, \hat{g}_y^2 ) $. We discuss the equality for $B$ momentarily. 

The identities \eqref{FuchsianFrenetFrame} yield coordinate-invariant expressions for the Frenet frame splitting purely in terms of
divisibility of polynomials and orthogonal complements. Two such facts we will need are: 
\begin{align} \label{LTNNull}
	U_p &= \{ P \in \Sym^6\R^2 \; | \; \; \hat{g}(p) \, | \, P \} \\
	\mathscr{L}_p \oplus T_p &= \{ P \in \Sym^6\R^2 \; \;| \; \; \hat{g}^2(p) \, | \, P \}. \label{LTNull} 
\end{align} 
The alternating nature of $\hat{f}$ is verified by the following orthonormal frame refining the Frenet frame splitting at $i \in \mathbb{H}^2$: 
\begin{align}\label{FuchsianFrenetFrameGenerators} 
	\begin{cases}
	  \hat{f} = \frac{\sqrt{5}}{4} (X^2+Y^2)^3 \text{ with }  q(\hat{f}) = +1 . \\
	 \hat{f}_x = \frac{ \sqrt{15} }{ \sqrt{8} }( \, X^5Y + 2X^3Y^3 + XY^5 \, )   \text{ with }   q(\hat{f}_x) = - 1. \\
	 \hat{f}_y = \frac{ \sqrt{15} }{ \sqrt{32} }( \, X^6 +X^4Y^2 -X^2 Y^4 - Y^6\, )  \text{ with }   q(\hat{f}_y) = -1. \\
	 \widehat{ \sff(\der{x}, \der{y}) } = \sqrt{3} ( \, X^5Y - XY^5 \, )  \text{ with }   q= +1 . \\
	 \widehat{ \sff(\der{x}, \der{x})  }= \frac{ \sqrt{3} }{4} ( \, X^6-5X^4Y^2-5X^2Y^4+Y^6 \, )  \text{ with }   q= +1 . \\
	 b_1 = (X^2-Y^2)(X^2-4XY+Y^2)(X^2+4XY+Y^2) = X^6-15X^4Y^2+15X^2Y^4-Y^6.\\
	 b_2 = XY(X^2-3Y^2)(3X^2-Y^2) =3 (X^5Y+XY^5)-10X^3Y^3 , \end{cases}
\end{align}
Here, $b_1$ and $b_2$ were computed with cross-products, using that $B = T \times_{\imoct} N$. 
Thus, $\hat{f}$ is an alternating almost-complex curve. Since $\hat{f}$ is an alternating almost-complex curve equivariant under a $\Gtwosplit$-Hitchin representation, the 
third fundamental form $\tff$ is non-degenerate, as discussed in Remark \ref{NonDegeneracyRemark}. Hence, defining the non-vanishing section $\alpha := p \mapsto  \hat{g}\, \hat{g}_x \, \hat{g}_y(p)$ of $N$, then $\tff(\der{x}, \alpha), \, \tff(\der{y}, \alpha)$ generate $B$; these expressions yield the desired equation 
for $B$ by deleting any terms divisible by $\hat{g}$.

Recall the natural map $\psi:  \mathcal{Q}_{\neq 0} \rightarrow \Eintwothree$ by $\psi(p, L) = L$, used to factor the developing map, from Section \ref{DevReinterpret}. 
Now, in the Fuchsian case, each $g \in \PSL_2\R$ acts as a fiber-preserving diffeomorphism of the bundle $\mathcal{Q}_{\neq 0} \rightarrow \Ha^2$ and the map $\psi$ is $\PSL_2\R$ equivariant.
As noted in Section \ref{DevReinterpret}, $\mathsf{image}(\psi) =\mathsf{image}(\dev)$.

To study $\psi$, we first examine the second extended osculating subspaces $U_p = \mathscr{L}_p \oplus T_p \oplus N_p$ and the bundle $\mathcal{Q}$ containing $\mathcal{Q}_{\neq 0}$. 
\begin{proposition}
For any $p \neq q \in \Ha^2$, we have $\dim(U_p \cap U_q) =3$ and $\sig(U_p \cap U_q) = (1,2). $
\end{proposition}

\begin{proof} 
Take any two subspaces $U_i := \{ \, P \in \Sym^6(\R^2) \; | \; P_i \, | \, U_i \; \}$, by \eqref{LTNNull}, of polynomials divisible by a fixed
 irreducible quadratic $P_i \in \Sym^2(\R^2)$. Clearly, $\dim U_i = 5$ and $ \dim( \mathbb{P} (U_1 \cap U_2) \, ) = 2$, which tells us 
 $U_1$ and $U_2$ are transverse. Indeed, $[P] \in \mathbb{P}(U_1 \cap U_2)$ has 4 of its 6 projective roots fixed. In the case of $P_1 = X^2 +Y^2$ and $P_t = tX^2+ \frac{1}{t}Y^2$, we show that $ \mathsf{sig}( U_1 \cap U_2 ) = (1,2)$.
 Observe that $P_1P_t\,XY $ is timelike and orthogonal to the spacelike elements $P_1P_t\, X^2, \; P_1P_t Y^2.$ One then
 calculates that the two-plane $\langle P_1P_t\, X^2, P_1P_t Y^2 \rangle$ has signature (1,1). Since $[P_1] = [ \hat{f}(i)]$ and $[P_t] = [\hat{f}(t\,i) ]$, by the action of $\PSL_2\R$ and
 basic hyperbolic geometry, we conclude, $\mathsf{sig}(U_p \cap U_q) = (1,2)$ in general for any $p\neq q \in \Ha^2$. \end{proof} 
 
Next, we examine $\mathcal{Q}$. Since $U_1 \cap U_2 \cong \R^{1,2}$, we see $\mathcal{Q}_1 \cap \mathcal{Q}_2 \cong \mathbb{P}Q_0\R^{1,2} = \Ein^{0,1}$ is a topological circle. The question remains of what is $(\mathcal{Q}_1)_{\neq 0} \cap (\mathcal{Q}_2)_{\neq 0} $, which we address now. 

\begin{proposition} 
For any $p \neq q \in \Ha^2$, we have $(\mathcal{Q}_{\neq 0})_p \cap (\mathcal{Q}_{\neq 0})_q \cong \mathbb{S}^1 \backslash D$, where $|D| \leq 4$.
\end{proposition} 

\begin{proof} 
We compute first in the case that $P_1 = X^2+Y^2$ and $P_t = tX^2+ \frac{1}{t}Y^2$, which again is completely general by the $\PSL_2\R$-transitivity. 
Note that 
$ (\mathcal{Q}_1)_{\neq 0} \cap (\mathcal{Q}_2)_{\neq 0} = (\mathcal{Q}_1 \cap \mathcal{Q}_2 ) \backslash D$,
where the degenerate set $D$ where the projection onto $\mathscr{L}, T, $ or $N$ vanishes either at $P_1$ or $P_t$. 
The degenerate locus at $p$ is given by $\mathcal{Q}_p \backslash (\mathcal{Q}_{\neq 0})_p = \mathbb{P}Q_0(\mathscr{L}_p \oplus T_p)
\sqcup \mathbb{P}Q_0( T_p \oplus N_p)$. 
Thus, the set $D$ naturally decomposes into four subsets $$D = \mathbb{P}Q_0 ( \,( \mathscr{L}_1 \oplus T_1) \cap U_2) 
\cup  \mathbb{P}Q_0 ( \, (\mathscr{L}_t \oplus T_t )\cap U_1) \cup  \mathbb{P}Q_0 ( (T_1 \oplus N_1) \cap U_t)  \cup \mathbb{P}Q_0 ( (T_t \oplus N_t) \cap U_1). $$  
Using the Frenet frame expressions from \eqref{FuchsianFrenetFrame}, one finds 
$$\mathbb{P}\left( (\mathscr{L}_t \oplus T_t )\cap U_1 \right) = [P_1P_t^2] \; \; \text{and} \; \; \mathbb{P} \left( (\mathscr{L}_1 \oplus T_1 )\cap U_t \right) =  [P_tP_1^2].$$
Since
$Q(P_1P_t^2) > 0$ and $ Q(P_tP_1^2) >0$, the first two sets in the definition of $D$ are empty. On the other hand, 
each of $ \mathbb{P}Q_0 ( (T_1 \oplus N_1) \cap U_t) \, , \, \mathbb{P}Q_0 ( (T_t \oplus N_t) \cap U_1) $ are non-empty, as we now show. 
Take $v \in U_1 \cap U_t$, so $ v $ is of the form $v= P_1P_t (aX^2+b XY+ cY^2)$. We can compute the coordinates of $v$ in the above basis for $U_1$
 from \eqref{FuchsianFrenetFrameGenerators} as well as in the corresponding basis for $U_t$ achieved by the action of $ \begin{pmatrix} \sqrt{t} & 0 \\ 0 & \frac{1}{\sqrt{t}} \end{pmatrix} $ on the basis for $U_1$. A (computer assisted) calculation shows $\proj_{\mathscr{L}_t} v = 0$ when
$c(1+ 3t^2) + a(1+\frac{3}{t^2} ) = 0$ and  $\proj_{\mathscr{L}_1}(v) = 0$ when $a(1+3t^2) + c (3 + t^2) =0 $. These equations are mutually exclusive unless $t =1$ or 
$a = c = 0$, but in the latter case, $Q_6(v) < 0$. Hence, $ \mathbb{P} Q_0( \, (T_1 \oplus N_1) \cap (T_{t} \cap N_{t}) \, )= \emptyset$ for $t \neq 1$,
which means $ \mathbb{P} Q_0( \, (T_1 \oplus N_1) \cap U_t) $ and  $\mathbb{P} Q_0( \, (T_t \oplus N_t) \cap U_1) $ are disjoint.
Denote $W_{1,t} := (T_1 \oplus N_1) \cap U_t$ and $W_{t,1} := (T_t \oplus N_t) \cap U_1$. Note that $\mathbb{P}Q_0(W_{t,1}), \,\mathbb{P}Q_0(W_{1,t})$
are each finite sets. We determine the possibilities for $| \mathbb{P} Q_0( \, W_{1,t} )|$.  
Define $w = P_1P_t\left ( \, (3+t^2) X^2 -(1+3t^2) Y^2 \right)$ so that $w$ and $u:= P_1P_t XY$ span $W_{1,t}$. 
Since $Q_6(w, u) =0 $ and $Q_6(u) < 0$, we need only determine the sign of $Q_6(w)$ to determine the signature $\mathsf{sig}(W_{1,t})$. 
Note that if $Q_6(w) = 0$, then $\mathsf{sig}(W_{1,t}) = (0,1,1)$, so that $| \mathbb{P} Q_0(W_{1,t})| =1$. Here $\mathsf{sig}(W) = (p,k, n)$ denotes 
$p$ positive, $k$ negative, and $n$ null parts in the signature. If $Q_6(w) > 0$, then $\mathsf{sig}(W_{1,t}) = (1,1)$ 
and $| \mathbb{P} Q_0( W_{1,t}) | = 2$,
and if $Q_6(w) < 0$, then $\mathsf{sig}(W_{1,t}) = (0,2)$ and $| \mathbb{P} Q_0(W_{1,t})| =0$. Another calculation shows 
$Q_6(w) = \frac{4}{15}(3 - 24 t^2 - 86 t^4 - 24 t^6 + 3t^8)$, so that for $t \in \R_+$, all three aforementioned cases occur
for $|\mathbb{P}Q_0(W_{1,t})|$. By equivariance, for every $t \in \R_+$, there is $s \in \R_+$ so that $W_{t,1} = W_{1,s}$. Hence, 
we conclude that for any $p\neq q$, we have $(\mathcal{Q}_{\neq 0})_p \cap (\mathcal{Q}_{\neq 0})_q \cong \mathbb{S}^1 \backslash D$, where $|D| \leq 4$. \end{proof}

In particular, we see the fact that for any $p \neq q \in \Ha$, the developed images  $ \psi( \, (\mathcal{Q}_{\neq 0})_p \, )$ and $\psi( \, (\mathcal{Q}_{\neq 0})_q \, )$ of the fibers intersect 1-dimensionally even though $\psi$ is a local diffeomorphism.  Let us illustrate geometrically what is happening with an example (here, $i \in \Ha^2$ again). While we can find points $X_t \in (\mathcal{Q}_{\neq 0})_{t \,i}$ for any $t > 0$ such that 
 $X_t \in (\mathcal{Q}_{\neq 0})_i$, if we take any such sequence with $t \rightarrow 1$, then $X_t$ subconverges to $ X_{\infty} \in \mathcal{Q} \backslash \mathcal{Q}_{\neq 0}$. Indeed, the limiting point $X_{\infty}$ will be divisible by $(X^2 +Y^2)^2$ and hence satisfy $X_{\infty} \in \mathscr{L}_i \oplus T_i$. 
The degeneration in the example above shows geometrically why the map $\mathcal{Q} \rightarrow \Eintwothree $ is not 
 an immersion in the Fuchsian case, as a complementary explanation to the Higgs bundle proof of the general case in Section \ref{DevReinterpret}.
 
The map $\psi$ is not proper onto its image in the Fuchsian case. Indeed, if $\psi$ were proper, then it would be a covering map. 
On the other hand, this is impossible since the point $[(X^2+Y^2)Y^4]$ has one pre-image under $\psi$, while any point 
$P \in \Eintwothree$ satisfying $P \in (\mathcal{Q}_{\neq 0})_p \cap (\mathcal{Q}_{\neq 0})_q$ for $p \neq q$ has two pre-images under $\psi$.
In fact, the map $\psi$ is either 3-1, 2-1, or 1-1, but but never $k$-1 for $k > 3$. 
By $\PSL_2\R$-transitivity, if $| \psi^{-1}([P]) | = 2$, then $P$ is of the form $P= Q_1Q_2R_1R_2$, where $Q_i$ are distinct irreducible quadratics
and $R_i \in \Sym^1(\R^2)$ are linear factors, and if $| \psi^{-1}([P]) | = 3$, then $P$ is of the form $P =Q_1Q_2Q_3$, where $Q_i$ are distinct irreducible quadratics;
both such cases occur. To see that 3-1 points do occur, note that the the set $\Omega_3$ of points of the form $[Q_1Q_2Q_3] \in \Eintwothree$ is a non-empty open set in $\Eintwothree$ and $\Omega_3 \subset \mathsf{image}(\overline{\psi})$, hence $\mathsf{image}(\psi)\cap \Omega_{3} \neq \emptyset$ too. A  final remark here is that if $u,v,w \in \Ha^2$ lie on a mutual geodesic then $Q_6( \hat{g}_u \hat{g}_v \hat{g}_w) > 0$. Thus, if $[P] \in \mathcal{Q}_u \cap \mathcal{Q}_v \cap \mathcal{Q}_w$, then $u,v,w$ are not collinear in $\Ha^2$.

\subsection{Relation to Guichard-Wienhard}\label{GW}
In \cite{GW12}, Guichard and Wienhard prove that every $G$-Hitchin component is realized as the holonomies of certain (unknown) 
$(G,X)$-structures on an (unknown) \emph{compact} manifold $M$, where the holonomy on $M$ factors through a homomorphism $\pi_1M \rightarrow \pi_1S$. They prove this result by constructing domains of discontinuity $\Omega$ in a flag manifold $X$ of $G$. 
 The theory developed in \cite{GW12} was then extended by Kapovich, Leeb, and Porti in \cite{KLP18}. 
 In this section, we discuss the relationship between the $(\Gtwosplit, \Eintwothree)$-structures from \cite{GW12} 
 and the $(\Gtwosplit, \Eintwothree)$-structures of this paper. We begin by recalling some relevant notation.\\

Denote $\mathcal{F}_0 = \mathbb{P}Q_0(\R^{3,4}) = \Eintwothree$ as the space of isotropic lines in $\R^{3,4}$ and $\mathcal{F}_1 = \Iso_{3}(\R^{3,4}) = \{ P \in \Gr_{3}(\R^{3,4}) \; | q|_P \equiv 0  \}$ as the space of \emph{maximally} isotropic planes in $\R^{3,4}$. Denote $Q_i := \Stab_{\mathsf{SO}_0(3,4)}(x_i)$ as the stabilizers of points $x_i \in \mathcal{F}_i$, each of which is a maximal parabolic subgroup of $ \mathsf{SO}_0(3,4)$. Finally, denote $\mathcal{F}_{01} : = \{ \, (l, P) \; |\; l \in \mathcal{F}_0,\;  P \in \mathcal{F}_1,
\; l \subset P \}$. There are natural projections $\pi_i: \mathcal{F}_{01} \rightarrow \mathcal{F}_i$ for $i \in \{0,1\}$. As a final notion, we need the ``flip'' defined in \cite{GW12} as follows:
given $A \subset \mathcal{F}_i$, we define $K_{A} := \pi_{1-i}(\, \pi_i^{-1} (A))$. Then, for example, if $A \subset \mathcal{F}_0$, then $K_A:=$ the set of isotropic 3-planes containing some point $l \in A$; the roles reverse if $A \subset \mathcal{F}_1$. 

We now recall the relevant notion of transversality as it pertains to the Anosov definition \cite[Definition 2.10]{GW12}. 
 Let $P$ be a parabolic subgroup conjugate to its opposite subgroup. Call a pair 
$(p_1, p_2)\in G/P \times G/P$ \emph{transverse} when $\Stab(p_1) \cap \Stab( p_2) \cong L$ is the Levi subgroup $L$ of $P$. 
Equivalently, this means $P_1 := \Stab(p_1)$ and $P_2:= \Stab(p_2)$ are opposite parabolic subgroups.\footnote{See \cite[page 15]{GW12} for a brief review on the structure of parabolic subgroups
and subalgebras or \cite[Chapter 7, $\S$ 7]{Kna96} for more comprehensive details.} 
Denote $\Gamma = \pi_1S$ and $\partial_{\infty}\Gamma$ as the Gromov boundary of $\Gamma$. 
In the cases of interest here of $P_1 := \Stab_{\Gtwosplit}(x)$ for $x \in \Eintwothree$ and $Q_1 < \mathsf{SO}_0(3,4)$, both parabolic 
subgroups are conjugate to their opposite subgroup. 
We shall use going forwards that $G$-Hitchin representations are $B_G < G$-Anosov, where $B_G$ is the Borel subgroup, and 
hence $(P^+, P^-)$-Anosov for any parabolic subgroup $P^+ = P <G$ by \cite[Lemma 3.18, Theorem 6.2]{GW12}. 
In particular, this implies that for each $\rho \in \Hit(S,G)$ and parabolic subgroup $P < G$ conjugate to its opposite, there is a (unique) continuous $\rho$-equivariant map $\xi_P: \partial_{\infty}\Gamma \rightarrow G/P$ 
that is \emph{transverse}, meaning $(\xi_P(s), \xi_P(t))$ are transverse for $s \neq t \in \partial_{\infty}\Gamma$, and satisfying an additional contraction property. 

Let $B < \Gtwosplit$ denote a copy of the Borel subgroup. Then $\Gtwosplit/B \hookrightarrow \mathsf{SL}_7\R/\mathsf{SO}_7\R \cong \mathsf{Flag}(\R^7)$ as a space of full flags $F = (F_1 \subset F_2 \subset \cdots \subset F_7)$ in $\imoct$ \cite[Lemma 2.5.3]{Eva24}. The rank two nature of $\Gtwosplit$ causes the flag $F$ to be determined by the pair $(F_1, F_2)$. Before we describe such full flags, recall from Section \ref{G2Prelims} that for $u \in \imoct$, we define the \emph{annihilator} of $u$ as $\Ann(u) := \{ v \in \imoct \; | \; u \times v = 0 \}.$
By Proposition \ref{Annihilators}, if $ [u] \in \mathcal{F}_0$, then $ \mathbb{P}\Ann( [u]) \in \mathcal{F}_1$. To simplify notation, we may conflate $\Ann( [u]) \subset \imoct$ with
$\mathbb{P}\Ann( [u]) \in \mathcal{F}_1$ when convenient.  

Now, by \cite[Corollary 2.5.4]{Eva24}, the flag $F \in \Gtwosplit/B$
obtains the following form for some $\R$-cross-product basis $(x_i)_{i=3}^{-3}$ for $\imoct:$ 
\begin{align}\label{FullFlagForm}
F = \left ( \; [x_3] \subset \spann_\R \langle x_3, x_2 \rangle \subset \Ann([x_3]) \subset \Ann([x_3])^\bot \subset (\spann_\R \langle x_3, x_2 \rangle)^\bot \subset [x_3]^\bot \subset \imoct  \;\right).
\end{align}
We observe that the subspaces $F_1, F_2$ from the flag $F$ have the following form: $F_1 \in \mathcal{F}_0$ and $F_2 \in \Gr_{2}Q_0(\imoct)$ is an isotropic 2-plane such that $F_2 \times_{\imoct} F_2 = \{0\}$. As it turns out, the space of such cross-product-trivial isotropic 2-planes $\mathsf{Iso}_2^\times \imoct$ 
is $\Gtwosplit$-equivariantly diffeomorphic to $\Gtwosplit/P_2$, where $P_2$ is a maximal parabolic subgroup of $\Gtwosplit$. The description \eqref{FullFlagForm} of the flag $F$ also says that 
$F_3= \Ann(F_1)$. As a consequence of this fact, if $\rho $ is $\Gtwosplit$-Hitchin and consequently $\mathsf{SO}_0(3,4)$-Hitchin, then the $Q_1$-Anosov boundary map 
is determined by the $Q_0$-Anosov map, as we now show. 

\begin{proposition} 
Let $\rho \in \Hit(S,\Gtwosplit)$. Then the $Q_i $-Anosov boundary maps $\xi_0, \xi_1$ for $\rho$ are related by $\xi_1 = \Ann \circ \xi_0$. 
\end{proposition} 

\begin{proof}
We recall equivalent conditions for transversality in $\Eintwothree$ and in $\mathcal{F}_1$.  
Two points $(\ell_1, \ell_2) \in \Eintwothree \times \Eintwothree$ are transverse $\iff$ $\ell_1 \oplus \ell_2$ is of signature (1,1)
$\iff$ $\Ann(\ell_1) \cap \Ann(\ell_2) = \emptyset $ $\iff$ $q_{\imoct}$ defines a non-degenerate pairing $\Ann(\ell_1) \times \Ann(\ell_2) \rightarrow \R$. 
On the other hand, $(R_1, R_2) \in \mathcal{F}_1 \times \mathcal{F}_1$ are transverse if and only if
$q_{3,4}$ defines a non-degenerate pairing $R_1 \times R_2 \rightarrow \R$. Thus, the map $\Ann \circ \xi_0: \partial_{\infty} \Gamma \rightarrow \mathcal{F}_1$ is a 
transverse map if $\xi_0$ is transverse. Since $\Ann$ is a continuous $\Gtwosplit$-equivariant map, if $\xi_0$ is $\rho$-equivariant and continuous, then so is $\Ann \circ \xi_0$. 

Now, consider an element $g = \exp(X)$ for $X \in \mathfrak{a}^+$, where $\mathfrak{a}^+$ is a Weyl chamber for an 
$\R$-split CSA $\mathfrak{a} < \g_2'$. Then the unique attracting fixed points $P_0$ and $P_1$ of $g$ in $\mathcal{F}_0$ and $\mathcal{F}_1$, respectively, are related by 
$P_1 = \Ann(P_0)$. Now, let $t_{\gamma}^+ \in \partial_{\infty} \Gamma$ be the unique attracting fixed point of $\gamma \in \Gamma$. The Anosov boundary map $\xi_0$ is uniquely constrained by 
$\xi_0( t_{\gamma}^+) $ being the unique attracting fixed point of $\rho(\gamma)$ in $\mathcal{F}_0$ \cite[Lemma 3.1, 3.3]{GW12}. It follows that  
$\Ann \circ \xi_0 $ satisfies the analogous constraint in $\mathcal{F}_1$ and thus $\xi_1 =\Ann \circ \xi_0$ by all our observations. \end{proof} 

Next, we recall the relevant result of Guichard-Wienhard on domains of discontinuity in $\Eintwothree$ for $\mathsf{SO}_0(3,4)$-Hitchin representations 
(and hence for $\Gtwosplit$-Hitchin representations as well.) The following result is proven by \cite[Proposition 8.3, Theorem 8.6, Theorem 9.12]{GW12}. 

\begin{theorem} 
Let $\rho \in \Hit(S, \mathsf{SO}_0(3,4) \, )$. Consider the unique $\rho$-equivariant 
Anosov boundary map $\xi_1: \partial_{\infty} \Gamma \rightarrow \mathcal{F}_1$. Then $\rho(\Gamma)$ acts properly discontinuously and co-compactly 
on the domain $\Omega_{\rho}^1 := \Eintwothree \diagdown K_{\im \, \xi_1}$. The topology of the quotient
$M = \rho(\Gamma) \backslash \Omega_{\rho}^1$ is independent of $\rho$. 
\end{theorem}

Let $\rho \in \Hit(S, \Gtwosplit)$ be $\Gtwosplit$-Fuchsian now and write $\rho = \iota \circ \rho_0$ for $\rho_0 \in T(S)$ a Fuchsian representation. 
In this case, since any orbit map of $\rho_0$ is a quasi-isometry, identify $\partial_{\infty} \Gamma \cong \partial_{\infty} \Ha^2$.
Again, we view hyperbolic space $\Ha^2$ as (note the sign change from earlier) 
 $\Ha^2 \cong Q_-( \Sym^2(\R^2), -Q_2) $, where $Q_2 $ is the quadratic form $Q_2 = \begin{pmatrix} 0 & 0 & -1 \\ 0 & \frac{1}{2} & 0 \\  -1 & 0 & 0 \end{pmatrix}$
 in the basis $(X^2, XY, Y^2)$, invariant under the $\PSL_2\R$ action on $\Sym^2(\R^2)$. 
 We then identify $\partial_{\infty} \Ha^2 \cong \Ein^{1,0} = \mathbb{P}Q_0( \Sym^2(\R^2), -Q_2 \, ) $, the set of projective homogeneous quadratic polynomials with repeated real root. 
Up to pre-composing by the identification $\partial_{\infty} \Gamma \cong \Ein^{1,0}$ depending on $\rho_0$, the $Q_0$-Anosov boundary map of $\rho$ is the $\mathsf{PSL}_2\R$-equivariant transverse map 
$\xi_0: \Ein^{1,0} \rightarrow \Eintwothree = \mathbb{P} Q_0( \Sym^6(\R^2), Q_6) $
 given by $[P] \mapsto [P^3]$. This follows from the uniqueness of $\xi_0$ and the dynamics of the action of $\Gamma$ on $\partial_{\infty} \Ha^2 = \Ein^{1,0}$ as a Fuchsian representation. 
Hence, 
$$\mathsf{image}(\xi_0) = \{ \;[L^6] \in \Eintwothree \; | \; L \in \Sym^{1}(\R^2) \}.$$ 

In general, the complement
$K_{\rho} := \Eintwothree \diagdown \Omega_{\rho} $ of the domain $\Omega_{\rho}$ is a locally trivial $\RP^2$-bundle over $\mathsf{image}( \xi_1)$ \cite[Page 36]{GW12}. 
However, in this simplified case, we can describe $K := K_{\rho}$, $\Omega := \Omega_{\rho}$ more precisely. 
Using the transitive $\PSL_2\R$ action on $\mathsf{image}( \xi_0)$, for any $L = aX + bY \in \Sym^1(\R^2)$,
one finds that
 \begin{align}\label{Xi1} 
 	\Ann( \, [L^6] \, ) \; = \; \{ \, [L^4Q] \; | \; Q \in \Sym^2(\R^2) \; \}, 
 \end{align} 
since \eqref{Xi1} holds at $L = X$ by our identifications from Table \ref{BasisIdentification} and the fact that $B$ from \eqref{CrossProductBasis} is an $\R$-cross-product basis. In particular, if $P \in \mathbb{P}\Sym^6(\R^2)$ has a multiplicity $\geq 4$ real root, then $Q_6(P) = 0$ automatically. Moreover, the bundle $K \rightarrow \mathsf{image}(\xi^1)$ is 
globally trivial in this case. Indeed, there is a 
global trivialization $\tau: \mathbb{S}^1 \times \RP^2 \rightarrow K$ as follows: identify $\mathbb{S}^1 \cong \mathbb{P}Q_0\Sym^2(\R^2)$ 
and then $$\tau( [P^2], \, [a:b:c] ) = [ P^4(aX^2+bXY+cY^2)] .$$ 

We now give a slightly different description of $K$. Note that $K= \bigcup_{i=1}^5 \mathcal{O}(P_i) $, where $\mathcal{O}(x)$ denotes the $\PSL_2\R$-orbit of $x \in \Eintwothree$ and 
$$ (P_i )_{i=1}^5 = ( \, X^6, \, X^5Y, X^4Y^2, \, X^4 Y(X-Y), \, X^4(X^2+Y^2) ).$$ 
Write $K_i := \mathcal{O}(P_i).$ 
The set $(P_i)$ serves as representatives for the combinatorial possibilities 
for the roots of a sextic with a multiplicity $\geq 4$ real root. 

Our developing map can be seen to be unrelated to the Guichard-Wienhard construction by showing that
in the case of $\rho$ a $\Gtwosplit$-Fuchsian representation, the image of $\overline{\dev}:=\overline{\dev}_{\rho}$ intersects both $K$ and $\Omega$.
Again, we factor $\overline{\dev} = \psi \circ \phi$ as in Section \ref{DevReinterpret}.  
To prove the claim, we show that $K_i \cap \mathsf{image}(\psi) = \emptyset$ for $1\leq i \leq 4$, but $K_5 \subset \mathsf{image}(\psi)$. 
By equivariance, we need only show $(X^2+Y^2)Y^4 \in \mathsf{image}(\psi)$. 
To see this, start we use the Frenet frame generators from \eqref{FuchsianFrenetFrameGenerators}. 
One then immediately finds $(X^2+Y^2)Y^4 \in (\mathcal{Q}_{\neq 0})_i$, since  
$$(X^2+Y^2)Y^4 = \frac{3}{8} (X^2+Y^2)^3 - \frac{1}{2} ( \, X^6 +X^4Y^2 -X^2 Y^4 - Y^6\, ) + \frac{1}{8}( \, X^6-5X^4Y^2-5X^2Y^4+Y^6 \, ).$$
Thus, the 3-dimensional degenerate locus $K_5$ is contained in the image of our developing map: 
$$ K_5 = \{ \,[ QR^4] \in \Eintwothree \; | \; Q \in Q_+(\Sym^2(\R^2)), \; R \in \Sym^1(\R^2) \} \subset \mathsf{image}(\psi) .$$
 On the other hand, the other generators $P_i$ for $i\in \{1,2,3,4\}$ each do not have an irreducible quadratic factor,
meaning $K_i \cap \mathsf{image}(\psi) = \emptyset$. 
We conclude that our $(\Gtwosplit, \Eintwothree)$-structures are distinct from those of Guichard-Wienhard.

We now briefly remark on $\mathsf{image}(\psi) \,  \cap \Omega$. The set $\Omega$ has four $\PSL_2\R$-invariant open subsets $\Omega_i$, for $i \in \{0,1,2, 3\}$
such that $\Omega \backslash (\cup_{i=1}^4\Omega_i) $ is nowhere dense in $\Eintwothree$,
where $\Omega_i \subset \Eintwothree$ is the set of projective polynomials $[P] \in \Eintwothree$ with $i$ distinct complex conjugate pairs of roots and the remaining roots real and distinct. Equation \eqref{LTNNull} from Section \ref{FuchsianDev} implies that $\mathsf{image}(\psi )\cap \Omega_0 = \emptyset$. 
On the other hand, for $i \in \{1,2,3\}$, the subsets $\mathsf{image}(\psi) \cap \Omega_i$ are open in $\Omega_i$, with the 
complement $\Omega_i \backslash (\mathsf{image} (\psi )\cap \Omega_i) $ difficult to describe explicitly. 

It remains an interesting open problem to geometrically clarify the $(\Gtwosplit, \Eintwothree)$-geometric structures of \cite{GW12}. 
A first step would seem to be understanding the topology of the compact quotient $M = \rho(\Gamma) \, \backslash \, \Omega_{\rho}$, for $\rho$ a $\Gtwosplit$-Fuchsian representation. 
This is currently work in progress. 

\appendix 
\section{$C^{\infty}$-Convergence of Solutions to Hitchin's Equations} \label{Appendix:ImplicitFunctionTheorem}

In this appendix, we prove the $C^{\infty}$-dependence of the solution to Hitchin's equations with respect to the input data of $\sigma, q$. 
This result is needed for the proof of Lemma \ref{s_continuity}.  
Going forward, we shall need some global estimates for the cyclic $\Gtwosplit$-Hitchin system. 
\begin{proposition}\label{Prop:GlobalEstimates}
Let $(\psi_1, \psi_2)$ be the unique solution to \eqref{G2Hitchin_GeneralMetric} with respect to $\sigma, q$. Then the following inequalities hold globally on $S$:
\begin{align}
	||q||_{\sigma}^2e^{-2\psi_1-2\psi_2} &< 3 \label{GlobalEstimate1} \\
	e^{\psi_1-5\psi_2} &< \frac{6}{5} \label{GlobalEstimate2} .
\end{align} 
\end{proposition} 

\begin{proof} 
The estimates are nearly the same as \cite[Lemma 4.5]{Eva22}. 
If $q \equiv 0$, \eqref{GlobalEstimate1} is trivial. 
Otherwise, set $\alpha :=||q||_{\sigma}^2e^{-2\psi_1-2\psi_2}$. Choose $p_1 \in S$ such that $\alpha(p_1) = \max_{x \in S} \alpha(x)$. On any open set $U \subset S$ such that $q(x) \neq 0 $ for $x \in U$, define $\beta: U \rightarrow \R $ by $\beta := \log \alpha $. Using \eqref{G2Hitchin_GeneralMetric},
one finds 
\begin{align}\label{EllipticEqn1}
	\Delta_{e^{2\psi_2} \sigma} \beta + (6 -2e^{\beta} ) = 0.
\end{align} 
Re-writing $\Delta_{e^{2\psi_2} \sigma } \sigma(p_1) \leq 0$ gives the desired unstrict inequality. Next, observe that the constant function $\beta_0 \equiv \log(1/3)$ is also a solution
to \eqref{EllipticEqn1} and $\beta \leq \beta_0$ on all of $U$. By the strong maximum principle 
\cite[Theorem 3.3.1]{Jos13}, the inequality $\beta \leq \beta_0$ is a strict global inequality or a global equality on any such set $U$. However, $q$ must have a zero, so we must have a strict inequality globally. 

 For \eqref{GlobalEstimate2}, we find that for 
$$\Delta_{e^{2\psi_2}\sigma} (\psi_1 -5 \psi_2) = \frac{1}{2} ( 30 e^{\psi_1 - 5\psi_2} - 30 - 2 |q|_{\sigma}^{-2\psi_1 -2\psi_2} )>  \frac{1}{2} ( 30 e^{\psi_1 - 5\psi_2} - 36 ).$$
Using $\Delta_{e^{2\psi_2 } \sigma} (\psi_1 -5 \psi_2)(p_2) \leq 0$ at a maximum $p_2$ of $\psi_1 -5 \psi_2$ proves the desired result. 
\end{proof} 

We now prove the dependence result of interest. The case of interest is local. So, fix a point $[\Sigma_0] \in T(S)$ and the unique hyperbolic metric $\sigma_0$ representing the conformal structure on $\Sigma_0 = (S, J_0)$. We recall some details from \cite{Tro92} on Teichm\"uller space. 
Define $\mathcal{M}_{-1}(S) , \mathcal{M}_{-1}^{k,\alpha}(S) $ to be the spaces of Riemannian metrics of constant curvature $-1$ 
with of regularities $C^{\infty}, C^{k,\alpha}$, respectively. We use the definition of Teichm\"uller space as $T(S) = \mathcal{M}_{-1}(S)/\Diff_0(S)$. Next, we select a local slice neighborhood 
for $T(S)$ in the Banach manifold $\mathcal{M}_{-1}^{k,\alpha}(S)$ for the following proof. 
By \cite[Theorem 2.4.2, Theorem 2.4.5]{Tro92}, there is a $(6g-6)$ dimensional neighborhood $U_0 \subset \mathcal{M}_{-1} \subset \mathcal{M}_{-1}^{k,\alpha}(S)$  
upon which $\Diff_0(S)$ identifies no two points, so that $U$ is
a lift of an open neighborhood $\overline{U}_0 =\pi(U_0) $ of $T(S)$, where $\pi: \mathcal{M}_{-1}(S) \rightarrow T(S)$ the quotient map. Then fix a holomorphic
sextic differential $q_0 \in H^0(\Sigma_0, \K^6_{\Sigma_0})$ as well as an open neighborhood $W_0$ of $(\sigma_0, q_0)$ in the bundle $\bigsqcup_{\sigma \in U_0} H^0(\Sigma_{\sigma}, \K^6_{\Sigma_\sigma})$, where $\Sigma_{\sigma}:= (S, J_{\sigma})$. 

Denote $\mathcal{H}^k(S): = W^{k,2}(S)$ as the Sobolev space of $u \in L^2(S)$ such that $X_1 X_2 \cdots X_j u \in L^2(S)$ for $j \leq k$ and any smooth vector fields $X_i \in \mathfrak{X}(M)$.  
For $k \geq 2$, the Sobolev embedding theorem says $\mathcal{H}^k(S) \hookrightarrow C^0(S)$. Thus, we can work with honest $C^{0}$-functions. 
The following proof is similar to the 
proof of \cite[Lemma 5.12]{DE23}, where an analogous regularity result was 
proven for the (cyclic) $\mathsf{SL}_3\R$-Hitchin equation.  

\begin{lemma}\label{HitContinuity}
The map 
$F: W_0 \rightarrow C^{\infty}(S) \times C^{\infty}(S)$ is continuous at $(\sigma_0, q_0)$, where $F(\sigma, q ):= (\psi_1, \psi_2)$ is the unique solution to \eqref{G2Hitchin_GeneralMetric} with respect to $(\sigma, q)$.  
\end{lemma} 

\begin{proof} 
Our goal is to use the implicit function theorem. Again, we take $k \geq N_0$ to be some large positive integer and $\alpha \in (0,1)$ arbitrary.
Define $\boldsymbol{\psi}_0 := (\psi_1^0, \psi_2^0) = F(\sigma_0, q_0)$. 
Now, consider the map 
$\hat{F}: W_0 \times C^{k, \alpha}(S) \times C^{k, \alpha}(S) \rightarrow C^{k-2,\alpha}(S) \times C^{k-2, \alpha}(S) $: 
$$ \hat{F}( \, (\sigma, q), (\psi_1, \psi_2) \, )= \left( 2\Delta_{\sigma}\psi_1 -5e^{\psi_1-3\psi_2}+2|q|^2_{\sigma}e^{-2\psi_1}- \frac{5}{2}\kappa_{\sigma}, \, 2\Delta_{\sigma}\psi_2+ 5 e^{\psi_1-3\psi_2} -6e^{2\psi_2} - \frac{1}{2}\kappa_{\sigma}  \right),$$
so that $\hat{F}( (\sigma, q), (\psi_1, \psi_2) ) = 0$ if and only if $(\psi_1, \psi_2)$ is the solution to \eqref{G2Hitchin_GeneralMetric}. 
Observe that the map $\hat{F}$ is smooth. 
Define the linear operator 
$A: C^{k,\alpha}(S) \times C^{k, \alpha}(S) \rightarrow C^{k-2,\alpha}(S) \times C^{k-2, \alpha}(S)$ 
by $A:= \deriv{\hat{F} }{\mathbf{\psi}} |_{ (\sigma_0, q_0, {\boldsymbol{\psi}_0} ) }$. The map $A$ is given by
$$A = \begin{pmatrix} (2\Delta_{\sigma_0}-5e^{\psi_1^0-3\psi_2^0}-4|q_0|^2_{\sigma_0}e^{-2\psi_1^0} ) & 5e^{\psi_1^0-3\psi_2^0} \\ 15e^{\psi_1^0-3\psi_2^0} & (2\Delta_{\sigma_0}-15e^{\psi_1^0-3\psi_2^0} -12e^{2\psi_2^0})\end{pmatrix}=: \begin{pmatrix} A_{1,1} & A_{1,2} \\ A_{2,1} & A_{2,2} \end{pmatrix}.$$ 
We first show $A$ is injective. Suppose that $A(w_1, w_2) = 0$ for some $w_1, w_2 \in C^{k, \alpha}(S)$. Then $ \langle  A(w_1, w_2), (w_1, w_2) \rangle_{L^2_{\sigma_0}(S) \times L^2_{\sigma_0}(S)} = 0$. Writing out this pairing,
one finds it has a sign. First, split $ \langle  A(w_1, w_2), (w_1, w_2) \rangle$ into two terms as follows:
\begin{align*}
  \langle  A(w_1, w_2), (w_1, w_2) \rangle_{L^2_{\sigma_0}(S) \times L^2_{\sigma_0}(S)} = &\left[ \langle 2\Delta_{\sigma_0}w_1-4|q_0|_{\sigma_0}^2 w_1, \, w_1 \rangle + \langle 2 \Delta_{\sigma_0} w_2,\, w_2 \rangle \right] \\
  	&+  \int_{S} 5e^{\psi_1^0-3\psi_2^0}( 4w_1w_2-w_1^2-3w_2^2-\frac{12}{5}e^{5\psi_2^0-\psi_1^0}\, w_2^2 )\, dV_{\sigma_0}
\end{align*} 
Observe that the first term is $\leq 0$ by the divergence theorem: 
$$\int_S u \Delta_{\sigma_0} u\, dV_{\sigma_0} = - \int_S \langle \nabla_{\sigma_0} u, \nabla_{\sigma_0} u \rangle dV_{\sigma_0} \leq 0,$$ for all $u \in \mathcal{H}^2(S)$. 
Of course, $4w_1w_2 \leq w_1^2 + 4w_2^2$. Hence, the second term is non-positive by the inequality $\frac{12}{5} e^{5\psi_2^0-\psi_1^0} > 2$ that follows from Proposition \ref{GlobalEstimate2}. 
These arguments show that $ \langle A(w_1, w_2), (w_1, w_2) \rangle  \leq 0$ with equality if and only if $w_1 = w_2 = 0$. Thus, $A$ is injective. 

We now show $A$ is surjective. 
Note that the linear transformation $K: \mathcal{H}^k(S) \times \mathcal{H}^k(S)  \rightarrow  \mathcal{H}^{k-2}(S) \times  \mathcal{H}^{k-2}(S)$ by $K := \begin{pmatrix} 0 & A_{1,2} \\ A_{2,1} & 0 \end{pmatrix}$ is compact. Indeed, this follows by the Rellich compactness theorem \cite[\S4 Proposition 3.4]{Tay11} 
that $\mathcal{H}^k(S) \hookrightarrow \mathcal{H}^{k-2}(S)$ is compact. Classical elliptic PDE techniques show that $T: \mathcal{H}^k(S) \times \mathcal{H}^k(S)  \rightarrow  \mathcal{H}^{k-2}(S) \times  \mathcal{H}^{k-2}(S)$ by $T = \begin{pmatrix} A_{1,1} & 0 \\ 0 & A_{2,2} \end{pmatrix}$ is a linear isomorphism. Indeed, each $A_{1,1}, A_{2,2}$ are injective by a pairing argument similar to the above proof for $A$. Then since $\Delta_{\sigma_0}: \mathcal{H}^k(S) \rightarrow \mathcal{H}^{k-2}(S)$ is Fredholm and of index 0 \cite[Theorem 19.2.1]{Hor07}, 
the maps $A_{i,i} = 2\Delta_{\sigma_0} + K_{i,i}$ are Fredholm and of index 0 too, showing both maps $A_{i,i}$ are isomorphisms. Hence, $A = T+K $ is of index 0 as well, so that $ A:  \mathcal{H}^k(S) \times \mathcal{H}^k(S)  \rightarrow  \mathcal{H}^{k-2}(S) \times  \mathcal{H}^{k-2}(S)$ is surjective since it is injective (by the same argument above). Hence, take any $(v_1, v_2) \in C^{k-2, \alpha}(S) \times C^{k-2, \alpha}(S)$ and we have $(w_1, w_2) \in  \mathcal{H}^k(S) \times \mathcal{H}^k(S) $ such that $A(v_1, v_2) = (w_1, w_2)$. But then elliptic regularity says that $(v_1, v_2) \in C^{k, \alpha}(S)$, meaning the original map $A: C^{k, \alpha}(S) \times C^{k,\alpha}(S) \rightarrow C^{k-2, \alpha}(S) \times C^{k-2,\alpha}(S) $ is surjective. Then by the open mapping theorem we conclude that $A$ is a topological linear isomorphism.

 Finally, we may apply the implicit function theorem for Banach spaces \cite[$\S1$ Theorem 5.9]{Lan99}. We get a neighborhood 
 $\mathcal{U}_k$ of $(\sigma_0, q_0)$ in $W_0$ such that
 $F_k: \mathcal{U}_k \rightarrow C^{k, \alpha}(S) \times C^{k, \alpha}(S)$ is smooth, where $F_k(\sigma, q)$ is the unique solution to \eqref{G2Hitchin_GeneralMetric}. 
Of course, $F_k $ is just $\iota_{k, \alpha} \circ F$, where $\iota_{k, \alpha}: C^{\infty}(S) \times C^{\infty}(S) \hookrightarrow C^{k, \alpha}(S) \times C^{k, \alpha}(S)$ denotes the inclusion map. We conclude the map $\iota_{k} \circ F: W_0 \rightarrow C^{k}(S) \times C^{k}(S)$
 is continuous at $(\sigma_0, q_0)$ for all $k$ sufficiently large. 
Thus, for any sequence $x_k:= (\sigma_k, q_k) \in W_0$ such that $x_k \rightarrow x_0 := (\sigma_0, q_0)$, then $F(x_k) \stackrel{C^k}{\rightarrow} F(x_0)$ for all $k$ and hence $F(x_k) \stackrel{C^\infty}{\rightarrow} F_i(x_0)$, so that $F$ is continuous at $(\sigma_0, q_0)$.  \end{proof} 

\printbibliography

\end{document}